\documentclass[11pt,reqno]{amsart}  
\usepackage{amsmath,amssymb,amsthm}
\theoremstyle{plain}
\newtheorem{theorem}{Theorem}[section]

\newtheorem{proposition}[theorem]{Proposition}

\newtheorem{corollary}[theorem]{Corollary}
\newtheorem{example}[theorem]{Example}

\numberwithin{equation}{section}
\allowdisplaybreaks[1]

\theoremstyle{definition}

\newtheorem{remark}[theorem]{Remark}

\theoremstyle{remark}

\newcommand{\fb}{{\mathfrak b}}
\newcommand{\fa}{{\mathfrak a}}
\newcommand{\free}{{\mathbb F}^{+}_{d}}
\newcommand{\bbC}{{\mathbb C}}

\newcommand{\bH}{{\mathbf H}}
\newcommand{\bK}{{\mathbf K}}

\newcommand{\bc}{{\mathbf c}}
\newcommand{\boldf}{{\mathbf f}}

\newcommand{\bv}{{\mathbf v}}

\newcommand{\cA}{{\mathcal A}}
\newcommand{\cB}{{\mathcal B}}
\newcommand{\cC}{{\mathcal C}}

\newcommand{\cE}{{\mathcal E}}

\newcommand{\cH}{{\mathcal H}}
\newcommand{\cI}{{\mathcal I}}

\newcommand{\cK}{{\mathcal K}}
\newcommand{\cL}{{\mathcal L}}
\newcommand{\cM}{{\mathcal M}}
\newcommand{\cN}{{\mathcal N}}
\newcommand{\cO}{{\mathcal O}}

\newcommand{\cR}{{\mathcal R}}
\newcommand{\cS}{{\mathcal S}}
\newcommand{\cT}{{\mathcal T}}
\newcommand{\cU}{{\mathcal U}}
\newcommand{\cV}{{\mathcal V}}

\newcommand{\cX}{{\mathcal X}}
\newcommand{\cY}{{\mathcal Y}}

\newcommand{\sbm}[1]{\left[\begin{smallmatrix} #1
		\end{smallmatrix}\right]}

\newcommand{\bev}{\boldsymbol{\rm ev}}

\begin{document}

\title[nc Reproducing kernel Hilbert spaces]
{Noncommutative reproducing kernel Hilbert spaces}
\author[J.A.\ Ball]{Joseph A. Ball}
\address{Department of Mathematics,
Virginia Tech,
Blacksburg, VA 24061-0123, USA}
\email{joball@math.vt.edu}
\author[G.\ MarX]{Gregory Marx}
\address{Department of Mathematics,
Virginia Tech,
Blacksburg, VA 24061-0123, USA}
\email{marxg@vt.edu}
\author[V.\ Vinnikov]{Victor Vinnikov}
\address{Department of Mathematics, Ben-Gurion University of the 
Negev, Beer-Sheva, Israel, 84105}
\email{vinnikov@cs.bgu.ac.il}

\begin{abstract}
The theory of positive kernels and associated reproducing kernel Hilbert spaces, 
especially in the setting of holomorphic functions, has been an 
important tool for the last several decades in a number of areas of complex analysis and 
operator theory. An interesting generalization of holomorphic 
functions, namely free noncommutative functions (e.g.,
functions of square-matrix arguments of arbitrary size satisfying 
additional natural compatibility conditions), is 
now an active area of research, with motivation and applications from 
a variety of areas (e.g., noncommutative functional calculus, free 
probability, and optimization theory in linear systems engineering).  
The purpose of this article is to develop a 
theory of positive kernels and associated reproducing kernel Hilbert 
spaces for the setting of free noncommutative function theory.
 \end{abstract}


\maketitle

\tableofcontents

\section{Introduction}  \label{S:Intro}
\setcounter{equation}{0}

The goal of the present paper is to incorporate the classical theory 
of positive kernels and reproducing kernel Hilbert spaces (see 
\cite{Aron, AMcC-book}) into the 
new setting of free noncommutative function theory (see 
\cite{KVV-book}),  

We use the following operator-valued adaptation of the notion of 
positive kernel developed in some depth by Aronszajn in \cite{Aron}.
Let $\Omega$ be a point set and $K$ a function from the Cartesian 
product set $\Omega \times \Omega$ into the space $\cL(\cY)$ of 
bounded linear operators on a Hilbert spaces $\cY$.  We say that $K$ 
is a \textbf{positive kernel} if
\begin{equation}   \label{posker1}
  \sum_{i,j = 1}^{N} \langle K(\omega_{i}, \omega_{j}) y_{j}, y_{i} 
  \rangle_{\cE} \ge 0 
\end{equation}
for all $\omega_{1}, \dots, \omega_{N} \in \Omega$, $y_{1}, \dots, 
y_{N} \in \cY$, $N = 1,2,\dots$. Equivalent conditions are:
\begin{itemize}
    \item There is a Hilbert space $\cH(K)$ consisting of 
    $\cY$-valued functions on $\Omega$ such that
$K$ has the following \textbf{reproducing kernel} 
	property with respect to $\cH(K)$:
    \begin{enumerate}
	\item   for any $\omega \in \Omega$ and 
	$y \in \cY$ the 
	function $K_{\omega, y}$ given by $K_{\omega, y}(\omega') = 
	K(\omega', \omega) y$ belongs to $\cH(K)$, and
	\item for all $f \in \cH(K)$ and $y \in \cY$, the reproducing property
\begin{equation}   \label{posker2}
   \langle f,\, K_{\omega, y} \rangle_{\cH(K)} = \langle f(\omega), y 
   \rangle_{\cY}
 \end{equation}
 holds.
\end{enumerate}

\item There is a Hilbert space $\cX$ and  a function $H \colon 
\Omega \to \cL(\cH(K), \cY)$ so that the following \textbf{Kolmogorov 
decomposition} holds:
\begin{equation} \label{posker3}
  K(\omega', \omega) = H(\omega') H(\omega)^{*}.  
\end{equation}
\end{itemize}
Moreover, in case $\Omega$ is a domain in ${\mathbb C}$ (or more 
generally in ${\mathbb C}^{d}$), then analyticity of $K$ in its first 
argument leads to analyticity of the elements $f$ of $\cH(K)$ and 
vice versa.  Such kernels and associated reproducing kernel Hilbert 
spaces actually have origins from the early part of the twentieth 
century (we mention in particular the original works \cite{Moore, 
Zaremba} as well as the introduction in \cite{Aron} and the survey 
paper \cite{Szafraniec} for an overview of the history); in 
particular, one can argue that the Kolmogorov decomposition property 
\eqref{posker3} is actually due to Moore \cite{Moore} but we follow 
what has become standard terminology in operator theory circles.
Over the following decades, reproducing kernel Hilbert spaces have served as a powerful 
tool of analysis in a number of 
areas (in particular, complex analysis and operator theory)---the recent book
\cite{AMcC-book} gives a  glimpse of the operator-theory side of this  activity.

Much more recent is the axiomatization of what is now being called 
\textbf{free noncommutative function theory}  (see \cite{KVV-book}).
One version of this  theory can be viewed as an extension of the 
theory of holomorphic functions of several complex variables $z = 
(z_{1}, \dots, z_{d})$ to a 
theory of functions of matrix tuples $Z = (Z_{1}, \dots, Z_{d})$, 
where now the tuple $Z = (Z_{1}, \dots, Z_{d})$ consists of freely 
noncommuting  $n \times n$ matrices with size $n \in {\mathbb N}$ 
also allowed to be arbitrary.  In brief, a noncommutative function is 
a mapping from $n \times n$ matrices over the domain set to $n \times 
n$ matrices over the range set which respects direct sums and 
similarities---see Section \ref{S:nc} for precise definitions and the 
monograph \cite{KVV-book} for a complete treatment. This axiomatic 
framework has appeared in the work of Taylor \cite{Taylor72} (see 
also \cite{Taylor73})
as a setting for the study of a functional calculus for tuples of 
freely noncommuting matrix variables, and independently in the work 
of Voiculescu and collaborators  (see \cite{Voiculescu1, Voiculescu2}) in the context 
of the needs of a free-probability theory, as well as in the work of 
Helton-Klep-McCullough and collaborators (see \cite{HKMcC1, HKMcC2, 
HMcC, HMcCV, BGtH}) on 
free noncommutative Linear-Matrix-Inequality relaxations and 
connections with stability analysis and optimization problems in 
linear systems engineering. It is also closely 
connected with the functional calculus arising from plugging in 
freely noncommuting operators for the free indeterminates in a formal 
power series, making use of the tensor products for the operator 
multiplication in each term of the series  (see \cite{BGM2}).  Such formal power 
series appeared much earlier in connection with the theory 
of automata and formal languages in the work of Fliess, Kleene, and 
Sch\"utzenberger in the middle part of the last century (see 
\cite{BR} for an overview), and have also 
occurred in Multidimensional System Theory as the transfer functions 
for  input/state/output linear systems with evolution along a 
tree \cite{BGM1}.  The recent monograph \cite{KVV-book} of Kaliuzhnyi-Verbovetskyi and the 
current third author completes the work of Taylor by developing the 
theory of free noncommutative functions from first principles. In 
particular, there it is shown how the ``respects direct sums'' and 
``respects similarities'' properties combined with some mild 
additional assumptions (``local boundedness'') leads to Taylor-Taylor 
series developments (at least locally) for a general noncommutative 
function, as is obtained globally when one starts with a formal power 
series and plugs in the components of a freely noncommutative operator tuple as arguments.
One can also view Free Noncommutative Function theory as a nonlinear 
extension of Operator Space theory (see \cite{BlM, ER, Paulsen, 
Pisier})---we explore some aspects of these connections in Section 
\ref{S:special} below.

In the present paper we extend the theory of 
noncommutative functions in \cite{KVV-book} to a theory of 
noncommutative kernels.  The affine version of such kernels already 
appears in \cite{KVV-book} in connection with the first order 
difference-differential calculus for noncommutative functions.   To 
develop a theory of kernels generalizing that of Aronszajn for the 
classical case, what is actually needed is a modification of these 
kernels (called here simply \textbf{nc kernels})
which treats the second variable as a conjugate variable (see 
condition \eqref{kerHerintertwine} below).
In analogy with the definition of noncommutative function, noncommutative 
kernels are required to satisfy a ``respect direct sums'' and 
``respect similarities'' condition, now with respect to both variables.
In addition, noncommutative (affine or general) kernels have values in 
$\cL(\cA, \cL(\cY))$ rather than just in $\cL(\cY)$, where $\cA$ is 
the vector space of point variations in the case of affine kernels, 
and is required to carry some additional order structure coming from 
another relevant $C^{*}$-algebra, or more generally, from an 
operator system $\cA$ (closed selfadjoint subspace of $\cL(\cE)$ for some 
Hilbert space $\cE$).   It turns out to be useful to introduce 
also the notion of {\em global function} and {\em global kernel}, 
where one discards the ``respects similarities'' constraint and demands only 
the ``respects direct sums'' condition. Then the analogue of an Aronszajn 
positive kernel is a {\em completely positive noncommutative (or 
global) kernel} for which one insists that the map $K(Z,Z) \colon 
\cA \to \cL(\cY)$ be a positive map from $\cA$ into $\cL(\cY)$ for 
each fixed $Z$ in the underlying noncommutative/global point set.  
Such positive kernels can be seen as more elaborate 
versions of the notion of {\em completely positive map} as occurs in 
the operator algebra literature (see \cite{Paulsen}) as well as of 
the notion of {\em completely positive kernel} as introduced and 
developed by Barreto-Bhat-Liebscher-Skeida \cite{BBLS} (see Section 
\ref{S:special} below for details).  

Specifically, given a completely positive
global or noncommutative kernel, we obtain a 
complete analogue of conditions \eqref{posker1}, \eqref{posker2}, 
\eqref{posker3} and their mutual equivalences (see Theorem 
\ref{T:cpker} below). 
In particular, the new condition (2) involves the existence of a 
Hilbert space whose elements now consist of noncommutative functions 
from the $C^{*}$-algebra $\cA$ to the Hilbert coefficient space $\cY$
which is also equipped with a canonical unitary $*$-representation $\sigma 
\colon \cA \to \cH(K)$.  The kernel elements $K_{\omega, y}$ in 
condition \eqref{posker2} are now more elaborate and reproduce not 
only the functional values of an element $f$ of the space $\cH(K)$ 
but also the functional value of the result of the action of an 
element $a \in \cA$ via 
the representation acting on $f$ ($\sigma(a) f$).   
We also obtain an equivalent ``lifted-norm'' 
characterization of such a spaces, where the factor $H$ in the 
Kolmogorov decomposition \eqref{posker3} has a prominent role
(see Section \ref{S:LiftedNorm}).

We also develop a number of additional structure properties of a 
space $\cH(K)$.
Section \ref{S:special} develops
automatic complete contractivity and complete boundedness properties
for the maps $K(Z,Z)$ and $f(Z)$ (for each $f \in \cH(K)$) and 
relates these results with similar such results in the 
operator-algebra literature (corresponding to the case where the point 
set $\Omega$ is the noncommutative envelope of a single point set $\Omega_1
 = \{s_{0}\}$. Section \ref{S:smooth} builds on results from \cite{KVV-book} on 
smoothness properties for noncommutative functions to develop the 
correspondence between smoothness properties of the kernel function 
$K$ and smoothness properties for elements $f$ of $\cH(K)$. Section 
\ref{S:NFRKHS} makes the identification between the notion of {\em 
formal reproducing kernel Hilbert space} introduced earlier by two of 
the present authors in \cite{NFRKHS} (see also \cite{BVMemoir}) and a 
{\em noncommutative reproducing kernel Hilbert space} introduced here.

Historically,  an important notion in applications of reproducing 
kernel Hilbert spaces to operator theory has been that of a 
{\em multiplier} between two reproducing kernel Hilbert spaces, 
i.e., an operator-valued function $S$ so that the operator $M_{S} 
\colon f(z) \mapsto S(z) f(z)$ maps the reproducing kernel Hilbert 
space $\cH(K')$ to the reproducing kernel Hilbert space $\cH(K)$.
Our final Section \ref{S:mult} develops the analogue of this notion 
for the free noncommutative setting.  In particular,  there is a free 
noncommutative analogue of the de Branges-Rovnyak kernel 
$$
 K_{S}(z,w) = K(z,w) - S(z) K'(z,w) S(w)^{*},
$$
complete positivity of which gives a characterization of the 
noncommutative function $S$ being a contractive multiplier from 
$\cH(K')$ into $\cH(K)$.  We also obtain the free noncommutative 
analogue of the theory of de Branges-Rovnyak spaces and the de 
Branges-Rovnyak theory of minimal decompositions and Brangesian 
complementary spaces.  In our followup paper \cite{BMV2} we use various 
parts of the material in the present paper to develop further the 
work of Agler-McCarthy \cite{AMcC-global, AMcC-Pick} on noncommutative Pick
interpolation and transfer-function realization and introduce a free noncommutative version
of ``complete Pick kernel''(see \cite{AMcC-book} for the classical 
case), as well as develop the interpolation and transfer-function 
realization theory for the more general noncommutative Schur-Agler 
class over the noncommutative domain ${\mathbb D}_{Q}$ associated 
with any noncommutative defining function $Q$.

\section{Global and noncommutative function theory}  \label{S:nc}

\subsection{Noncommutative sets} \label{S:ncsets}

Let $\cS$ be a set.  We define $\cS_{\rm nc}$ to be the disjoint 
union of $n \times n$ matrices over $\cS$:
$$
 \cS_{\rm nc} = \amalg_{n=1}^{\infty} \cS^{n \times n}.
$$
We let $\cS_n$ denote the intersection $\cS_{n} = \cS_{\rm nc} \cap 
\cS^{n \times n}$.
A subset $\Omega$ of $\cS_{\rm nc}$ is said to be a 
\textbf{noncommutative (nc) set} 
if $\Omega$ is closed under direct sums:
$$
Z  \in \Omega_{n},\, W \in \Omega_{m} \Rightarrow
Z \oplus W = \left[ \begin{smallmatrix} Z & 0 \\ 0 & W 
\end{smallmatrix} \right] \in \Omega_{n+m}.
$$

If $\Omega$ is a subset of $\cS_{\rm nc}$ which is not already a nc 
subset of $\cS_{\rm nc}$, it is convenient to introduce the notation 
$[\Omega]_{\rm nc}$ for the \textbf{nc envelope of $\Omega$}, i.e., the 
smallest nc subset of $\cS_{\rm nc}$ which contains $\Omega$.  
More precisely (see \cite[Proposition 2.9]{BMV2}), a 
point $Z \in \cS_{\rm nc}$ is in $[\Omega]_{\rm nc}$ exactly when it 
has a representation as $Z = \sbm{ Z^{(1)} & & \\ & \ddots & \\ & & 
Z^{(N)}}$ where each $Z^{(j)} \in \Omega$ ($j = 1, \dots, N$), or  
equivalently, $[\Omega]_{\rm nc}$ is the intersection of all 
noncommutative subsets of $\cS_{\rm nc}$ containing $\Omega$.

We shall from time to time assume that the set $\cS$ carries some 
additional structure. Examples are as follows:
\begin{itemize}
    \item $\cS$ is a \textbf{vector space $\cV$ over the complex numbers 
    ${\mathbb C}$}:  One can use the module structure of $\cV$ over 
    ${\mathbb C}$ to make sense of a matrix multiplication of the form
    $\alpha \cdot X \cdot \beta$ where $\alpha \in {\mathbb C}^{n \times m}$, $X 
    \in \cV^{m \times k}$, $\beta \in {\mathbb C}^{k \times n}$.
    Note that more generally one can replace $\cV$ with a module 
    $\cM$ over a unital abelian ring $\cR$ and then replace the role 
    of ${\mathbb C}$ with the ring $\cR$ (as is done in 
    \cite{KVV-book}), but we shall be mainly 
    interested in the case where $\cV$ is equipped with some 
    additional structure which forces a complex vector-space 
    structure on $\cV$.

 \item $\cV$ is a \textbf{concrete operator space}, i.e., there is a 
 Hilbert space $\cH$ such that $\cV$ is equal to a linear subspace of 
 $\cL(\cH)$ (our notation for the space of bounded linear operators 
 on $\cH$).   In this case $\cV$ inherits a norm from $\cL(\cH)$.  
 Moreover matrices over $\cL(\cH)$, i.e. $\cL(\cH)^{n \times m}$, can 
 be viewed as the space of bounded linear operators from the 
 Hilbert-space direct sum   $\cH^{m} = \cH \oplus \cdots \oplus \cH$ 
 ($m$-fold direct sum) into  $\cH^{n}$ and hence inherits a canonical 
 operator norm from its canonical identification with $\cL(\cH^{m}, 
 \cH^{n})$.  If $\cV$ is a subspace of $\cL(\cH)$, then there is a 
 canonical norm on the space $\cV^{n \times m}$ of $n \times m$ 
 matrices over $\cV$, namely, the norm inherited from it being a 
 subspace of $\cL(\cH^{m}, \cH^{n})$.  Thus any concrete operator 
 space $\cV$ comes equipped not only with a Banach-space norm but 
 also the space of matrices $\cV^{n \times m}$ comes equipped with 
 its own norm in this way.  These norms satisfy the {\em Ruan axioms}.
 Conversely, any Banach space $\cV$ such that the each space of matrices 
 $\cV^{n \times m}$ over $\cV$ is equipped with a norm $\| \cdot 
 \|_{n,m}$ with the collection of such norms satisfying the Ruan 
 axioms (such an object is called an (abstract) \textbf{operator 
 space}) is completely isometrically isomorphic to a concrete operator 
 space (the analogue of the Gelfand-Naimark theorem for operator 
 spaces) by a theorem of Ruan---see e.g.\ \cite[Theorem 2.3.5]{ER} 
 or \cite[Theorem 13.4]{Paulsen}.  Here a linear map 
 $\varphi$ between operator spaces $\cV$ and $\cV_{0}$ is said to be 
 \textbf{completely isometric} if not only $\varphi \colon \cV \to 
 \cV_{0}$ is isometric but also $\varphi^{(n,n)} \colon \cV^{n 
 \times n} \to \cV_{0}^{n \times n}$ is isometric for all $n \in 
 {\mathbb N}$, where in general we set
 \begin{equation}   \label{id-n,m}
  \varphi^{(n, m)} = {\rm id}_{{\mathbb C}^{n \times m}} 
  \otimes \varphi \colon [ v_{ij}]_{\begin{smallmatrix} i=1, 
  \dots, n \\ j=1, \dots, m \end{smallmatrix}} \mapsto 
  [\varphi(v_{ij})]_{\begin{smallmatrix} i=1, 
  \dots, n \\ j=1, \dots, m \end{smallmatrix}}
 \end{equation}
 for any natural numbers $n,m$.

 As a particular case of this setup as well as that in the previous 
 bullet, consider the case where $\cV = {\mathbb 
 C}^{d}$ (i.e., $d$-tuples $z = (z_{1}, \dots, z_{d})$ of complex 
 numbers).  Then it is convenient to make the identification $({\mathbb 
 C}^{d})^{ n \times n}$ of $n \times n$ matrices over ${\mathbb 
 C}^{d}$ with the space $({\mathbb C}^{n \times n})^{d}$ consisting 
 of $d$-tuples of $n \times n$ complex matrices.  Note that 
$\cV$ can be considered as a subspace of $\cL({\mathbb C}^{d})$ by the 
simple device of embedding the $d$-tuple of complex numbers as the 
first row of a $d \times d$ matrix with the other rows set equal to 
zero.  This embedding gives the canonical identification of ${\mathbb 
C}^{d}$ with an operator space.  Note that one could equally well 
identify ${\mathbb C}^{d}$ with the set of $d \times d$ matrices 
having all but the first column equal to zero.  One of the quirks of 
the theory is that, while these two operator spaces are isometric at 
level 1, they fail to be  isometric to each other at higher levels, i.e., 
there is no completely isometric map of ${\mathbb C}^{d}_{\rm row}$ 
onto ${\mathbb C}^{d}_{\rm col}$ (see \cite{ER} for details).
 
 \item $\cV$ is a \textbf{concrete operator system}, i.e., a 
 selfadjoint linear subspace of $\cL(\cH)$ containing the identity 
 operator $I_{\cH}$ on $\cH$. Any concrete operator system $\cV$, as 
 well as the space $\cV^{n \times n}$ consisting of square matrices 
 over $\cV$ of any size $n \times n$, is equipped with a special 
 cone, the cone of positive semidefinite elements, inducing an 
 ordering on $\cV^{n \times n}$ satisfying certain compatibility 
 conditions.  The Choi-Effros Theorem (see \cite[Theorem 
 13.1]{Paulsen} says that any such \textbf{abstract operator system} 
 is then completely order isomorphic to a concrete operator system.

 \item $\cV$ is a \textbf{concrete operator algebra}, i.e., a (not 
 necessarily selfadjoint) norm-closed subalgebra of $\cL(\cH)$ for 
 some Hilbert space $\cH$. Again the space of matrices $\cV^{n \times 
 m}$ over $\cV$ comes equipped with norms satisfying certain 
 compatibility conditions which also respect the algebra structure.  
 The converse result, that any such \textbf{abstract operator algebra} 
 is completely isometrically isomorphic to a concrete operator 
 algebra, is due to Blecher-Ruan-Sinclair (see \cite[Corollary 
 16.7]{Paulsen} or the original \cite{BRS}).

\end{itemize}

\subsection{Global and noncommutative functions}  \label{S:ncfunc}
We suppose that $\cS$ and $\cS_{0}$ are two sets and that $\Omega$ is 
a nc subset of $\cS_{\rm nc}$.  We say that a function $f$ from the 
nc set $\Omega \subset \cS_{\rm nc}$ into $\cS_{0,{\rm nc}}$
 is a \textbf{global function}  if 
\begin{itemize}
    \item $f$ is \textbf{graded}:  $f \colon \Omega_{n} = \Omega \cap 
    \cS^{n \times n} \to 
    \cV_{0, n}: = (\cS_{0})^{n \times n}$, and
    \item $f$ \textbf{respects direct sums}:
    $$ f\left( \left[ \begin{smallmatrix} Z & 0 \\ 0 & W  
\end{smallmatrix} \right] \right) = \left[ \begin{smallmatrix} f(Z) & 
0 \\ 0 & f(W) \end{smallmatrix} \right] \text{ for all }
Z \in \Omega_{n},\, W \in \Omega_{m}, n,m \in {\mathbb N}.
$$
\end{itemize}

 For the next definition, we assume that  $\cV$ and $\cV_{0}$ are 
 vector spaces over  ${\mathbb C}$ and we let 
 $\Omega$ be a nc subset of $\cV_{\rm nc}$. 
 Then, as explained above in Subsection \ref{S:ncsets},
 for $\alpha \in {\mathbb C}^{n 
 \times m}$, $V \in \cV_{0}^{m \times k}$ and $\beta \in {\mathbb C}^{k 
 \times \ell}$, we can use the module structure of $\cV_{0}$ over 
 ${\mathbb C}$ to make sense of the matrix multiplication $\alpha 
 \cdot V \cdot \beta \in \cV_{0}^{n \times \ell}$ and similarly $\alpha Z \beta$ 
 makes sense as an element of  $\cV^{n \times \ell}$ for $Z \in \cV^{m \times k}$.
  Given a function $f \colon \Omega \to 
 \cV_{0,nc}$, we say that $f$ is a \textbf{noncommutative (nc) function} if
 \begin{itemize}
     \item $f$ is \textbf{graded}, i.e., $f \colon \Omega_{n} \to 
     \cV_{0,n}$, and
     \item $f$  \textbf{respects intertwinings}:  
     \begin{equation}  \label{funcintertwine}
	 Z \in \Omega_{n}, \, 
     \widetilde Z \in \Omega_{m}, \, \alpha \in {\mathbb C}^{m \times n} 
     \text{ with } \alpha Z = \widetilde Z \alpha \Rightarrow \alpha f(Z) = f(\widetilde Z) 
     \alpha.
   \end{equation}
 \end{itemize}
 It can be shown (see  \cite[Section I.2.3]{KVV-book}) that equivalently:  
 $f$ is a nc function if and only if 
 \begin{itemize}
     \item $f$ is a \textbf{global function}, 
 i.e., $f$ is graded and $f$ respects direct sums, and 
 \item $f$  \textbf{respects 
 similarities}, i.e.: whenever $Z, \widetilde Z \in \Omega_{n}$, 
 $\alpha \in {\mathbb C}^{n \times n}$ with $\alpha$ invertible such that 
 $\widetilde Z = \alpha Z \alpha^{-1}$, then $f(\widetilde Z) = \alpha f(Z) \alpha^{-1}$.
 \end{itemize}
 
 In the concrete case where $\cV = {\mathbb C}^{d}$ for some positive 
 integer $d$ and we make the identification $({\mathbb C}^{d})^{n 
 \times n} \cong ({\mathbb C}^{n \times n})^{d}$ as discussed in 
 Subsection \ref{S:ncsets} above, then we view a nc function $f 
 \colon \Omega \to \cV_{0, {\rm nc}}$ as a function on 
$d$ matrix arguments $Z = (Z_{1}, \dots, Z_{d})$, where each $Z_{j} 
\in {\mathbb C}^{n \times n}$ and $n$ is free to vary over all natural numbers.

 The class of nc functions from $\Omega_{\rm nc}$ to $\cV_{0, \rm nc}$  
 is one of the main objects of study in the book \cite{KVV-book} where the 
 notation $\cT(\Omega; \cV_{0, \rm{nc}})$ is used for the set of all such 
 functions. In particular, there it is shown that under mild 
 boundedness conditions, such functions have a local (Brook) Taylor series 
 expansion, as already previewed in earlier work of Joseph Taylor 
 \cite{Taylor72}) and as to be suggested by the letter $\cT$ in the notation 
 $\cT(\Omega; \cV_{0, {\rm nc}})$ for this class.

\subsection{Global and noncommutative kernels}  \label{S:ncker}

Let $\cS$ be a set as in the previous subsection with associated 
noncommutative set $\cS_{\rm nc}$ and let $\Omega$ be a nc subset of 
$\cS_{\rm nc}$ as above. Let $\cV_{0}$ and $\cV_{1}$ be 
two vector spaces. Suppose that $K$ is a 
function from $\Omega \times \Omega$ to $\cL(\cV_{1}, 
\cV_{0})_{\rm nc}$ (where $\cL(\cV_{1}, \cV_{0})$ is the space of 
linear operators from $\cV_{1}$ to 
$\cV_{0}$).  We say that $K$ is a \textbf{global kernel} if 
\begin{itemize}
    \item $K$ is \textbf{graded} in the sense that  
\begin{equation} \label{kergraded}
    Z \in \Omega_{n}, \, W \in \Omega_{m}
      \Rightarrow K(Z,W) \in 
    \cL(\cV_{1}^{n \times m}, \cV_{0}^{n \times m}),
 \end{equation}
 and 
 \item $K$ \textbf{respects direct sums} in the sense that
 \begin{equation} \label{kerds}
     K\left( \sbm{ Z & 0 \\ 0 & \widetilde Z}, \sbm{W & 0 \\ 0 & 
 \widetilde W} \right) \left( \sbm{ P_{11} & P_{12} \\ 
P_{21} & P_{22}} \right) = \sbm{ K(Z,W)(P_{11}) & K(Z, \widetilde W)(P_{12}) \\ 
K(\widetilde Z,W)(P_{21}) & K(\widetilde Z, \widetilde W)(P_{22}) }
    \end{equation}.
\end{itemize}
We note that the direct sum condition \eqref{kerds} can be iterated 
to arrive at the more general form:  $K$  respects direct sums if 
and only if: {\em
whenever  $Z_{i} \in \Omega_{n_{i}}$, $W_{j} \in \Omega_{m_{j}}$ and
    $P_{ij} \in \cV_{1}^{n_{i} \times m_{j}}$ for $i=1, 
    \dots, N$, $j=1, \dots, M$ and we set
$Z = \sbm{ Z^{(1)} & & \\ & \ddots & \\ & & Z^{(N)}} \in \Omega_{n}$,
$W = \sbm{ W^{(1)} & & \\ & \ddots & \\ & & W^{(M)}}\in \Omega_{m}$ 
and  $P = [P_{ij}]_{1 \le i \le N, 1 \le j \le M} \in \cV_{1}^{n \times 
m}$ where we set $n=n_{1} + \cdots + n_{N}$, $m = m_{1} + \cdots + 
m_{M}$,  then}
\begin{equation} \label{kerds-iterated}
 K(Z,W)(P) = \left[ K(Z^{(i)},W^{(j)})(P_{ij}) \right]_{1 
\le i \le N, 1 \le j \le M}.
\end{equation}

\begin{example}  \label{E:globalker}
{\em One way to generate a global kernel is as follows (see \cite[Remark 
3.6 page 41]{KVV-book} for a more general setting). 
Let $\cV$, $\cV_{0}$, $\cX$ be vector spaces, let $H$ be a global function from 
$\Omega$ to $\cL(\cX, \cV_{0})_{\rm nc}$, let  $G$ be a global 
function from $\Omega$ to $\cL(\cV_{0}, \cX)_{\rm nc}$, and suppose 
that $\sigma \colon \cV_{1} \to \cX$ is a linear map.
Define $K$ by
\begin{equation}   \label{genker}
  K(Z,W)(P) = H(Z) ({\rm id}_{{\mathbb C}^{n \times m}} \otimes \sigma)(P) G(W)
\end{equation}
for $Z \in \Omega_{n}$, $W \in \Omega_{m}$,  $P \in (\cV_{1})^{n 
\times m}$
where we use the notation \eqref{id-n,m}.
Then one can check that $K$ so defined is a global kernel from 
$\Omega \times \Omega$ to $\cL(\cV_{1}, \cV_{0})_{\rm nc}$.}
\end{example}

For completeness we mention the following notion of what we shall call 
{\em affine noncommutative kernels} defined as follows.
Assume that the underlying set $\cS$ is replaced by a vector space $\cV$ 
with  $\Omega$ is a nc subset of $\cV_{\rm nc}$, and let 
$\cV_{1}$ and $\cX_{0}$ be vector spaces.  We say that the 
function $K$ from $\Omega \times \Omega$ to $\cL(\cV_{1}, 
\cL(\cX_{0}))_{\rm nc}$ 
is a \textbf{nc affine kernel} if
\begin{itemize} 
    \item $K$ is a \textbf{graded kernel} (see \eqref{kergraded}),
and 
\item $K$ \textbf{respects intertwinings} in the sense that
\begin{align} 
    & Z \in \Omega_{n},\, \widetilde Z \in \Omega_{\widetilde n}, \,
    \alpha \in {\mathbb C}^{\widetilde n \times n} \text{ such that } \alpha Z = 
    \widetilde Z \alpha,  \notag \\
    &  W \in \Omega_{m},\, \widetilde W \in \Omega_{\widetilde m}, \,
    \beta \in {\mathbb C}^{m \times \widetilde m} \text{ such that }  W \beta = 
    \beta \widetilde W, \notag \\
&  P \in \cV_{1}^{n \times m} \Rightarrow \alpha K(Z,W)(P) \beta = 
K(\widetilde Z, \widetilde W)(\alpha P \beta).
\label{kerintertwine}
\end{align}
\end{itemize}
An equivalent set of conditions is:
\begin{itemize}
    \item $K$ is \textbf{global kernel}, i.e., $K$ is graded 
    \eqref{kergraded} and respects direct sums \eqref{kerds}, and
    \item $K$ \textbf{respects similarities} in the following sense:
\begin{align}
    & Z, \widetilde Z \in \Omega_{n}, \, \alpha \in {\mathbb C}^{n \times 
    n} \text{ invertible  with } \widetilde Z = \alpha Z \alpha^{-1}, \notag \\
    & W, \widetilde W \in \Omega_{m}, \, \beta \in {\mathbb C}^{m \times 
    m} \text{ invertible  with } \widetilde W = \beta^{-1} W \beta,\notag  \\
   & P \in \cV_{1}^{n \times m} \Rightarrow K(\widetilde Z, \widetilde 
    W)(P) = \alpha \, K(Z,W)( \alpha^{-1} P \beta^{-1})\, \beta.
\label{kersim'}
\end{align}
\end{itemize}
One can check that the formula \eqref{genker} gives rise to an affine 
kernel in the case where $H$ and $G$ are  nc functions (rather than 
just global functions). These kernels arise in a natural way in the 
noncommutative differential-difference calculus worked out in 
\cite{KVV-book}.

Our main interest however will be in the following variant of affine kernels 
which we shall call simply {\em noncommutative (nc) kernels}.
We again assume that $\Omega$ is a nc subset of $\cV_{\rm nc}$ for a 
vector space $\cV$, and that $\cV_{1}$ and $\cV_{0}$ are vector 
spaces.   We then say that a function $K \colon \Omega \times \Omega \to 
\cL(\cV_{1}, \cV_{0})_{\rm nc}$ is a  \textbf{nc kernel} if 
the following  conditions hold:
\begin{itemize}
\item $K$ is \textbf{graded} (see \eqref{kergraded}), and
\item $K$ \textbf{respects intertwinings} in the following sense:
\begin{align}
   & Z \in \Omega_{n},\, \widetilde Z \in \Omega_{\widetilde n},\, \alpha \in 
    {\mathbb C}^{\widetilde n \times n} \text{ such that } \alpha Z = \widetilde Z 
    \alpha, \notag \\
 & W \in \Omega_{m},\, \widetilde W \in \Omega_{\widetilde m}, \, \beta \in {\mathbb 
 C}^{\widetilde m \times m} \text{ such that } \beta W = \widetilde W 
 \beta, \notag \\
 & P \in \cV_{1}^{n \times m} \Rightarrow  \alpha \, K(Z,W)(P) \, \beta^{*} = 
 K(\widetilde Z, \widetilde W) (\alpha P \beta^{*}).
 \label{kerHerintertwine}
 \end{align}
\end{itemize}

An equivalent set of conditions is: 
\begin{itemize}
    \item $K$ is \textbf{graded}, i.e., satisfies \eqref{kergraded}, 
    
    \item $K$ \textbf{respects direct sums} in the following sense:
\begin{align}
   & Z \in \Omega_{n}, \,\widetilde Z \in \Omega_{\widetilde n},\,
    W \in \Omega_{m}, \, W \in \Omega_{\widetilde m},  \,
    P = \sbm{ P_{11} & P_{12} \\ P_{21} & P_{22} } \in 
    \cV_{1}^{(n+m) \times (\widetilde n+ \widetilde m)} \Rightarrow 
    \notag \\
  &  K\left( \sbm{ Z & 0 \\ 0 & \widetilde Z}, \sbm {W & 0 \\ 0 & 
    \widetilde W} \right) \left( \sbm{ P_{11} & P_{12} \\ P_{21} & 
    P_{22} } \right) =
    \sbm {K(Z,W)(P_{11}) & K(Z, \widetilde W)(P_{12}) \\
	  K(\widetilde Z, W)(P_{21}) & K(\widetilde Z, \widetilde 
	  W)(P_{22})  } .
\label{kerdirsum} 
\end{align}

\item $K$ \textbf{respects similarities}:  
\begin{align}
  &  Z, \widetilde Z \in \Omega_{n}, \, \alpha \in {\mathbb C}^{n \times n} 
    \text{ invertible with } \widetilde Z =  \alpha  Z  \alpha^{-1}, \notag \\
  & W, \widetilde W \in \Omega_{m}, \beta \in {\mathbb C}^{m \times m} 
  \text{ invertible with } \widetilde W = \beta W \beta^{-1}, \notag \\
& P \in \cV_{1}^{n \times m} \Rightarrow  K(\widetilde Z, \widetilde W)(P) =
\alpha \, K(Z,W)(\alpha^{-1} P \beta^{-1*})\,  \beta^{*}.
\label{kersim}
\end{align}
\end{itemize}
We denote the class of all such nc kernels by $\widetilde 
\cT^{1}(\Omega; \cV_{0, {\rm nc}}, \cV_{1, {\rm nc}})$.

\subsection{Completely positive global/nc kernels}   \label{S:cpker}

We now assume that the vector spaces 
$\cV_{1}$ and $\cV_{0}$ are {\em operator systems} carrying an 
associated conjugate-linear involution $P \mapsto P^{*}$.
Then if $P$ is a square matrix over either $\cV_{1}$ or $\cV_{0}$ 
there is a well-defined notion of positivity $P \succeq 0$ arising 
from the operator-space structure (see Section \ref{S:ncsets} for 
some discussion).

We say that a nc kernel $K \in\widetilde \cT^{1}(\Omega; \cV_{0, {\rm nc}}, \cV_{1, 
{\rm nc}})$ is  a \textbf{completely positive noncommutative (cp nc) 
kernel}  if in addition, for all $n \in 
{\mathbb N}$ we have 
\begin{equation} \label{kercp}
    Z \in \Omega_{n},\,  P \succeq 0 \text{ in }
    \cV_{1}^{n \times n} \Rightarrow K(Z,Z)(P) \succeq 0 \text{ in } 
    \cV_{0}^{n \times n}.
\end{equation}
In case $\Omega \subset \cS_{\rm nc}$ where $\cS$ does not 
necessarily carry any vector space structure and $K \colon \Omega 
\times \Omega \to \cL(\cV_{1}, \cV_{0})_{\rm nc}$ is a global kernel,
we say that $K$ is a \textbf{completely positive global (cp global) 
kernel} if $K$ respects direct sums \eqref{kerdirsum} and $K$ 
satisfies \eqref{kercp}.

In case $\cV_{1} = \cA$ and $\cV_{0} = \cB$ are $C^{*}$-algebras, we 
have the following equivalent formulation of the complete positivity 
condition.  

\begin{proposition}   \label{P:cpkernel}
    The global kernel $K \colon \Omega \times \Omega \to \cL(\cA, 
    \cB)_{\rm nc}$ is cp if and only if: for any $Z^{(j)} \in 
    \Omega_{n_{j}}$, $P_{j} \in \cA^{N \times n_{j}}$,  
    $b_{j} \in \cB^{n_{j}}$ for $n_{j} \in {\mathbb N}$,   $j=1,2,\dots, N$, 
    and $N=1,2,\dots$, it 
    holds that
\begin{equation}  \label{cpker-expanded}
    \sum_{i,j=1}^{N} b_{i}^{*} K(Z^{(i)}, Z^{(j)})(P_{i}^{*}P_{j}) b_{j} 
    \succeq 0.
\end{equation}
\end{proposition}

\begin{proof}  For sufficiency, simply take $N=1$ in condition 
    \eqref{cpker-expanded} to recover condition \eqref{kercp}.
    
    For necessity, note that that the left hand side of 
    \eqref{cpker-expanded} can be rewritten as
\begin{equation}   \label{cpkereval-compact}
b^{*} K(Z,Z)(P^{*}P) b
\end{equation}
if we set
$$
b = \sbm{b_{1} \\ \vdots \\ b_{N}},\, Z = \sbm{ Z^{(1)} & & \\ & \ddots 
& \\ & & Z^{(N)}},\, P = \begin{bmatrix} P_{1} & \cdots & P_{N} 
\end{bmatrix}
$$
as a consequence of \eqref{kerds-iterated}.
The arbitrariness of $b \in \cB^{n}$ ($n = \sum_{j=1}^{N} n_{j}$) then implies 
that $K(Z,Z)(P^{*}P) \succeq 0$.
 \end{proof}
 
\begin{remark}  \label{R:cp-cb} 
    Let $\cV_{0}$ and $\cV_{1}$ be operator systems.  Recall (see 
    e.g.\ \cite{Paulsen}) that a 
    map $\varphi \colon \cV_{1} \to \cV_{0}$ between operator systems 
    $\cV_{1}$ and $\cV_{0}$ is said to be \textbf{completely 
    positive} (cp) if, for every $N \in {\mathbb N}$ (with notation 
    as in \eqref{id-n,m})
    \begin{align}
   & \left( {\rm  id}_{{\mathbb C}^{N \times N}} \otimes 
   \varphi\right)( P) \ge 0 \text{ in } (\cV_{0})^{N \times 
    N}  \notag \\
    &  \text{ whenever } P \ge  0 \text{ in } (\cV_{1})^{N \times N}.
    \label{cp}
  \end{align}
  Similarly, if $\varphi \colon \cO_{1} \to \cO_{0}$ is a linear map 
  between operator spaces $\cO_{1}$ and $\cO_{0}$, $\varphi$ is said to 
  be \textbf{completely bounded} if there is a constant $M < \infty$ 
  so that
  \begin{equation}   \label{cb}
  \| ({\rm id}_{{\mathbb C}^{N \times N}}\otimes \varphi)( P) \|_{\cO_{0}^{N \times N}} 
  \le M \| P \|_{\cO_{1}^{N \times N}} \text{ for all } P \in 
  \cO_{1}^{N \times N}.
\end{equation}
The smallest such constant $M$ is called the \textbf{completely 
bounded norm} of the map $\varphi$.  It is well known (see e.g. 
\cite{Paulsen}) that a completely positive map $\varphi \colon 
\cV_{1} \to \cV_{0}$ between operator 
systems is automatically completely bounded with completely bounded 
norm equal to $\| \varphi(1_{\cV_{1}}) \|$:
\begin{equation}  \label{cb=b}
    \| \varphi \| = \| \varphi \|_{\rm cb} = \| \varphi(1_{\cV_{1}}) \|
    \text{ for all } A \in \cO_{1}^{N \times N}.
\end{equation}

If $K$ is a cp global/nc kernel on the nc set $\Omega$, $n \in 
{\mathbb N}$ and $Z$ is any point in $\Omega_{n}$, the identity
$$
 \left[ K(Z,Z)\left( P_{ij} \right) \right]_{i,j=1, \dots, N} =
 K\left( \bigoplus_{1}^{N} Z,  \bigoplus_{1}^{N} Z \right)\left( [ 
 P_{ij} ]_{i,j=1, \dots, N} \right),
$$
a consequence of the ``respects direct sums'' condition 
\eqref{kerds-iterated} satisfied by $K$, shows that the map $K(Z,Z)$ 
is completely positive.  By the general fact \eqref{cb=b}, it 
follows that $K(Z,Z)$ is also completely bounded with completely 
bounded norm equal to its norm equal to $\|K(Z,Z)(I)\|$:
\begin{equation}   \label{cpkercb}
    \| K(Z,Z)\| = \| K(Z,Z)\|_{\rm cb} = \| K(Z,Z)(I) \|.
\end{equation}
\end{remark}
 
 \begin{remark}  \label{R:kermodel}
One can get cp global kernels and cp nc kernels by specializing Example \ref{E:globalker} 
as follows.  We assume that $\cV_{1}$ and $\cV_{0}$ are operator 
systems.   By definition (or by the Choi-Effros theorem if one views $\cV_{0}$ 
as an abstract operator system), there is no loss of generality in 
viewing $\cV_{0}$ as a subspace of a $C^{*}$-algebra $\cL(\cY)$ 
for some Hilbert space $\cY$.  Let us assume that $\cX$ is another 
Hilbert space, that $H$ is a global/nc 
function from $\Omega$ to $\cL(\cX,\cY)_{\rm nc}$, and that $\sigma$ is a 
cp map (as defined in \eqref{cp}) from $\cV_{1}$ into $\cL(\cX)$.  Then one can check that 
the function $K$ defined by
\begin{equation}   \label{hereditary}
K(Z,W)(P) = H(Z) \, ({\rm id}_{{\mathbb C}^{n \times m}} \otimes 
\sigma)(P) \, H(W)^{*}
\end{equation}
for $Z \in \Omega_{n}$, $W \in \Omega_{m}$, $P \in \cV_{1}^{n \times 
m}$ is a cp global/nc kernel $K$ from $\Omega \times \Omega$ to $\cL(\cV_{1}, 
\cL(\cY))_{\rm nc}$.  In case $\cV_{1}$ is also a $C^{*}$-algebra, 
then a consequence of Theorem \ref{T:cpker} below 
(specifically, of the equivalence (1) $\Leftrightarrow$ (3) there)
is that \eqref{hereditary} is the form for any cp global/nc kernel.
\end{remark}
 
  \begin{remark} \label{R:Arveson}
      Suppose that $\cV_{0}$ and $\cV_{1}$ are operator systems and  
      that the function $K$ from $\Omega \times \Omega$ into
      $\cL(\cV_{1}, \cV_{0})_{\rm nc}$ is a cp global/nc kernel.  As was mentioned in the 
      previous remark, there is no loss of generality in viewing $\cV_{0}$ 
      as a subspace of a full $C^{*}$-algebra $\cL(\cY)$, so there is no 
      loss of generality in assuming that $\cV_{0} = \cL(\cY)$ in the definition of cp 
      global/nc kernel.  An interesting open question is:  {\em given a cp 
      global/nc kernel from $\Omega \times \Omega$ to $\cL(\cV_{1}, \cL(\cY))_{\rm nc}$ 
      where $\cV_{1} \subset \cL(\cE)$ for some Hilbert space $\cE$,  is there a cp 
      global/nc kernel $\widetilde K$ from $\Omega \times \Omega$ to 
      $\cL(\cL(\cE), \cL(\cY))_{\rm nc}$ so that $\widetilde K(Z,W)(P) = 
      K(Z,W)(P)$ for all $Z \in \Omega_{n}$, $W \in \Omega_{m}$, $P 
      \in \cV_{1}^{n \times m}$?} If this question also has an 
      affirmative answer, then there is no loss of generality in 
      assuming that $\cV_{1} = \cL(\cE)$ is also a full 
      $C^{*}$-algebra in the definition of cp global/nc kernel.
      The special case of this question where one takes $\Omega$ to 
      be the nc envelope of a singleton set $\{\omega_{0}\}$ has a 
      positive answer by the Arveson-Wittstock Hahn-Banach extension 
      theorem (see \cite{ERbook, Paulsen}).  A positive answer to 
      this question for the general case would in turn guarantee that \eqref{hereditary} 
      is the form for a general cp global/nc kernel, even without the 
      assumption that $\cV_{1} = \cA$ is a $C^{*}$-algebra.
     \end{remark}

\section{Global/noncommutative reproducing kernel Hilbert spaces} 
\label{S:RKHS}

\subsection{Main result}   \label{S:main}

To formulate our main result, we assume without loss of generality 
(see Remark \ref{R:Arveson}) that the operator system 
$\cV_{0}$ is presented to us 
in concrete form as the set of bounded linear operators $\cL(\cY)$
on a Hilbert space $\cY$ (a \textbf{full} $C^{*}$-algebra).  Then we have the following 
characterization of cp global kernels and of cp nc kernels.
Part of the result is that the analogue of representation 
\eqref{genker} for cp kernels gives a complete characterization of cp 
global/nc kernels (see statement (3) in the statement of Theorem 
\ref{T:cpker}).

\begin{theorem}   \label{T:cpker}  
    Suppose that $\Omega$ is a nc subset of $\cS_{nc}$ for some set 
    $\cS$, $\cY$ is a Hilbert space, $\cA$ is a $C^{*}$-algebra and
    $K \colon \Omega \times \Omega \to \cL(\cA, \cL(\cY))_{\rm nc}$ is a 
    given function.  Then the following are equivalent.
\begin{enumerate}
   \item $K$ is a cp global kernel.
   
   \item There is a Hilbert space $\cH(K)$ whose elements are global 
   functions $f \colon \Omega \to \cL(\cA, \cY)_{\rm nc}$ such that:
   \begin{enumerate}
   \item For each $W \in \Omega_{m}$, $v \in \cA^{1 \times m}$, and $y \in \cY^{m}$, the function 
 $$
 K_{W,v,y} \colon \Omega_{n} \to \cL(\cA, \cY)^{n \times n} \cong 
 \cL(\cA^{n}, \cY^{n})
 $$
 defined by
 \begin{equation}   \label{kerel}
 K_{W,v,y}(Z) u = K(Z,W)(uv) y
 \end{equation}
 for $Z \in \Omega_{n}$, $u \in \cA^{n}$ belongs to $\cH(K)$.
 
 \item The kernel elements $K_{W,v,y}$ as in \eqref{kerel} have the reproducing property:  for 
 $f \in \cH(K)$, $W \in \Omega_{m}$, $v \in \cA^{1 \times m}$, 
 $y \in \cY^{m}$,
 \begin{equation}   \label{reprod}
  \langle f(W)(v^{*}), y \rangle_{\cY^{m}} = \langle f, K_{W,v,y} 
  \rangle_{\cH(K)}.
 \end{equation}
 
  \item $\cH(K)$ is equipped with a unital $*$-representation  
  $\sigma$ mapping $\cA$ to  $\cL(\cH(K))$  such that
 \begin{equation}  \label{rep1}
     \left( \sigma(a) f\right)(W)(v^{*}) = f(W)(v^{*}a)
 \end{equation}
 for $a \in \cA$, $W \in \Omega_{m}$, $v \in \cA^{1 \times m}$, with 
 action on kernel elements $K_{W,v,y}$ given by
 \begin{equation}   \label{kernelaction}
     \sigma(a) \colon K_{W,v,y} = K_{W, av, y}.
  \end{equation}
 \end{enumerate}
 
 \item $K$ has a Kolmogorov decomposition:  there is a Hilbert space 
 $\cX$ equipped with a unital $*$-representation $\sigma \colon \cA 
 \to \cL(\cX)$ together with a global function $H \colon \Omega \to 
 \cL(\cX, \cY)_{\rm nc}$ so that
 \begin{equation}   \label{Koldecom}
     K(Z,W) (P) = H(Z) ( {\rm id}_{{\mathbb C}^{n \times m}} \otimes \sigma) (P) H(W)^{*}
 \end{equation}
 for all $Z \in \Omega_{n}$, $W \in \Omega_{m}$, $P \in \cA^{n \times 
 m}$.
\end{enumerate}
    
Furthermore, statements (1), (2), and (3) remain equivalent if it is assumed that 
$\Omega \subset \cV_{\rm nc}$ for a vector spaces $\cV$, $K$ is taken 
to be a cp nc kernel in statement (1), the elements $f$ of $\cH(K)$ 
are taken to be nc functions from $\Omega$ to $\cL(\cA, \cY)_{\rm 
nc}$ in statement (2), and $H$ is taken to be a nc function from 
$\Omega$ to $\cL(\cX, \cY)_{\rm nc}$ in statement (3).
\end{theorem}

We prove (1) $\Rightarrow$ (2), (2) $\Rightarrow$ (3), and (3) 
$\Rightarrow$ (1).

\begin{proof}[Proof of (1) $\Rightarrow$ (2)]
    Assume first that $K$ is a cp global kernel and that we have 
    chosen a $Z \in \Omega_{n}$, $W \in \Omega_{m}$,$u \in \cA^{n}$, $v \in \cA^{1 \times m}$ and $y 
    \in \cY^{m}$.    The estimates
\begin{align*}
    \|K(Z,W)(uv) y \|_{\cY^{n}} & \le \|K(Z,W)(uv)\|_{\cL(\cY^{m}, 
    \cY^{n})} \| y  \|_{\cY^{m}} \\
    & \le \| K(Z,W)\|_{\cL(\cA^{n \times m}, \cL(\cY^{m}, \cY^{n}))} 
    \| u v \|_{\cA^{n \times m}} \| y \|_{\cY^{m}} \\
    & \le \left( \| K(Z,W)\|_{\cL(\cA^{n \times m}, \cL(\cY^{m}, 
    \cY^{n}))} \| v \|_{\cA^{1 \times m}} \| y \|_{\cY^{m}} \right)
    \| u \|_{\cA^{n}}
 \end{align*}
 show that the formula \eqref{kerel} defines a bounded operator 
 $K_{W,v,y}(Z)$ from $\cA^{n}$ to $\cY^{n}$.
 
  One can then use the 
    assumption that $K$ is a global kernel to check that each $K_{W, 
    v, y}$ is a global function as follows.  Given a $K_{W, v, y}$ 
    ($W \in \Omega_{m}$, $v \in \cA^{1 \times m}$, $y \in \cY^{m}$), 
    for $Z_{1} \in \Omega_{n_{1}}$, $Z_{2} \in \Omega_{n_{2}}$, 
    $u_{1} \in \cA^{n_{1}}$, $u_{2} \in \cA^{n_{2}}$, we have
    \begin{align*}
K_{W,v,y}\left( \sbm{Z_{1}& 0 \\ 0 & Z_{2}} \right) \begin{bmatrix} u_{1} \\ 
u_{2} \end{bmatrix} & =
K\left( \sbm{Z_{1} & 0 \\ 0 & Z_{2}}, W\right)\left( \begin{bmatrix} u_{1}v \\ 
u_{2} v \end{bmatrix} \right) y \\
& = \begin{bmatrix} K(Z_{1},W)(u_{1}v) y \\ K(Z_{2},W)(u_{2}v) y 
\end{bmatrix} \\
& = \begin{bmatrix} K_{W,v,y}(Z_{1}) & 0 \\ 0 & K_{W,v,y}(Z_{2}) 
\end{bmatrix} \begin{bmatrix} u_{1} \\ u_{2} \end{bmatrix}.
\end{align*}
and we conclude that $K_{W,v,y}$ respects direct sums.  

We define a 
linear space $\cH^{\circ}(K)$ as the span of such kernel elements
$$
\cH^{\circ}(K) = {\rm span}\{K_{W,v,y} \colon W \in \Omega_{m},\, v 
\in \cA^{1 \times m},\, y \in \cY^{m},\, m=1,2,\dots \}
$$
with inner product of two kernel elements given by
\begin{equation}   \label{kerinnerprod'}
\langle K_{W,v,y},\, K_{W',v',y'} \rangle_{\cH^{\circ}(K)} = 
\langle K(W',W)(v^{\prime *}v)y, y' \rangle_{\cY^{m'}}
\end{equation}
and then extend to any two elements of $\cH^{\circ}(K)$ by 
sesquilinearity.  The fact that $K$ satisfies the cp condition 
(expressed in the form \eqref{cpker-expanded}) implies that the inner 
product is positive semidefinite.  We may then take the completion to 
arrive at a pseudo-Hilbert space 
$\cH(K)$. Elements of the completion $f$ can again be viewed as 
global functions from $\Omega$ to $\cL(\cA, \cY)$ via the reproducing 
formula \eqref{reprod}; indeed one checks directly that 
\eqref{reprod} holds for $f = K_{W',v'y'}$ ($W' \in \Omega_{m'}$, $v' 
\in \cA^{1 \times m'}$, $y' \in \cY^{m'}$):
\begin{align}
    \langle f, K_{W,v,y} \rangle_{\cH(K)} & =
    \langle K_{W',v'y'}, K_{W,v,y} \rangle_{\cH(K)} \notag \\
    & = \langle K(W,W')(v^{*}v') y', y \rangle_{\cY^{m}} \notag  \\
    & = \langle K_{W',v',y'}(W)( v^{*}), y \rangle_{\cY^{m}} \notag \\
    & = \langle f(W) (v^{*}), y \rangle_{\cY^{m}}.
    \label{func}
 \end{align}
 We then justify the validity of the formula for the case that $f$ is a 
 finite linear combination of kernel elements by linearity; when $f$ 
 is the $\cH(K)$-limit of a sequence of finite linear combination of 
 kernel elements, we simply use the formula \eqref{reprod} to define 
 $f(W)$.  We note that the identification $f \mapsto ( W \mapsto f(W))$
 of $f \in \cH(K)$ with the function $W \mapsto f(W)$ defined by 
 \eqref{func} is well-defined:  if $ \langle f, f \rangle_{\cH(K)} = 
 0$, then a consequence of the Cauchy-Schwarz inequality is that $f$ 
 is orthogonal to $K_{W,v,y}$ for all $W,v,y$; consequently, for all 
 $W \in \Omega_{m}$ and $y \in \cY^{m}$ we have
 \begin{align*}
     \langle f(W)(v^{*}), y \rangle_{\cY^{m}} & =
     \langle f, K_{W,v,y} \rangle_{\cH(K)} = 0
 \end{align*}
 so that the associated function $W \mapsto f(W)$ is zero.
 Similarly, the correspondence $f \mapsto (W \mapsto f(W))$ is 
 injective: if $f \in \cH(K)$ and $f(W) = 0$ for all 
 $W$, then it follows that $\langle f(W) v^{*}, y \rangle_{\cY^{m}} = 
 0$ for all $W \in \Omega^{m}$, $v^{*} \in \cA^{1 \times m}$, $y \in 
 \cY^{m}$ for $m=1,2,\dots$ whence it follows from \eqref{reprod} that
 $f$ is orthogonal to all kernel elements $K_{W,v,y}$.  As such 
 kernel elements have dense span in $\cH(K)$ by construction, it 
 follows that $f$ is the zero element of $\cH(K)$. With this 
 identification, it follows that in fact $\cH(K)$ is a Hilbert space, 
 i.e., any element with zero self inner-product is the zero element 
 of the space.
 As we have already seen that kernel elements are global 
 functions, it follows by linearity and taking limits that each 
 element $f$ of $\cH(K)$ is also a global function from $\Omega$ to 
 $\cL(\cA, \cY)_{\rm nc}$.
 
 For $a \in \cA$, we define an action $\sigma(a)$ on kernel elements 
 $K_{W,v,y}$ by
 $$
 \sigma(a) \colon K_{W,v,y} \mapsto K_{W, a v, y}.
 $$
 It is easily checked that $\sigma$ is additive, multiplicative and 
 unital:
 $$
 \sigma(a_{1} + a_{2}) = \sigma(a_{1}) + \sigma(a_{2}),\quad
 \sigma(a_{1} a_{2}) = \sigma(a_{1}) \circ \sigma(a_{2}), \quad
 \sigma(1_{\cA}) = I_{\cH(K)}.
 $$
 We check that $\sigma$ respects adjoints:
 \begin{align*}
     \langle \sigma(a) K_{W,v,y},\, K_{W',v',y'} \rangle_{\cH(K)} & =
     \langle K_{W, av, y}, \, K_{W',v',y'} \rangle_{\cH(K)} \\
     & = \langle K(W',W)(v^{\prime *} (av)) y, \, y' 
     \rangle_{\cY^{m'}} \\
     & = \langle K(W',W)( (a^{*}v^{\prime})^{*} v) y, \, y' 
     \rangle_{\cY^{m'}} \\
     & = K_{W,v,y},\, K_{W',a^{*}v', y'} \rangle_{\cH(K)} \\
     & = \langle K_{W,v,y}, \, \sigma(a^{*}) K_{W',v',y'} 
     \rangle_{\cH(K)}
 \end{align*}
 and it follows that $\sigma(a)^{*} = \sigma(a^{*})$.  
 
 Given $a \in \cA$, we extend the action $\sigma(a)$ to the span 
 $\cH^{\circ}(K)$ by linearity.  We claim that $\sigma(a)$ is bounded 
 and therefore extends to a bounded operator on all of $\cH(K)$.  
 Indeed, we note that, for $f = \sum_{j=1}^{N} K_{W_{j}, v_{j}, 
 y_{j}} \in \cH^{\circ}(K)$ where say $W_{j} \in \Omega_{m_{j}}$, 
 $v_{j} \in \cA^{1 \times m_{j}}$, 
 \begin{align*}
   & \| a \|^{2} \|f\|_{\cH(K)}^{2} -  \| \sigma(a) f \|^{2}_{\cH(K)} = \\ 
   &  \| a \|^{2}  \left\langle  \sum_{j=1}^{N} K_{W_{j}, v_{j}, y_{j}},  
     \sum_{i=1}^{N} K_{W_{i}, v_{i}, y_{i}} \right\rangle_{\cH(K)} -
     \left\langle  \sum_{j=1}^{N} K_{W_{j}, av_{j}, y_{j}},  
     \sum_{i=1}^{N} K_{W_{i}, av_{i}, y_{i}} \right\rangle_{\cH(K)} \\
& = \sum_{i,j=1}^{N} \langle K(W_{i}, W_{j}) (v_{i}^{*}
( \| a \|^{2} 1_{\cA} - a^{*} a) v_{j}) y_{j}, y_{i}\rangle_{\cY^{m_{i}}} \ge 0
     \end{align*}
 since $\|a \|^{2} 1_{\cA} - a^{*} a$ is a positive element of $\cA$ and the 
 kernel $K$ satisfies the cp condition \eqref{cpker-expanded}.  We 
 conclude that $\| \sigma(a) \| \le \| a\|$ and hence $\sigma$ 
 extends to a unital $*$-representation $\sigma \colon \cA \mapsto 
 \cL(\cH(K))$.
 
 We next compute
 \begin{align*}
     \langle \left( \sigma(a) f\right)(W)(v^{*}), y \rangle_{\cY^{m}} 
     & = \langle \sigma(a) f, K_{W,v,y} \rangle_{\cH(K)} \text{ (by 
     \eqref{reprod})} \\
     & = \langle f, K_{W, a^{*}v,y} \rangle_{\cH(K)} \\
     & = \langle f(W)(v^{*}a), y \rangle_{\cY^{m}}
\end{align*}
and the formula \eqref{rep1} follows.

In case $\Omega \subset \cV_{nc}$ for a vector space $\cV$ and $K$ 
is a cp nc kernel, we verify that the kernel elements $K_{W,v,y}$ are 
actually nc functions as follows.  Suppose first that $f = K_{W,v,y}$ 
(with $W \in \Omega_{m}$, $v \in \cA^{1 \times m}$, $y \in \cY^{m}$) 
is a kernel element, and suppose that we are given $W' \in 
\Omega_{n}$, $\widetilde W' \in \Omega_{\widetilde n}$, and an $\alpha \in 
{\mathbb C}^{\widetilde n \times n}$ with $\alpha W' = \widetilde W' \alpha$.  
We then use the assumed intertwining property for $K$ 
\eqref{kerintertwine} to deduce the intertwining property 
\eqref{funcintertwine} for $f$: for $y' \in \cY^{\widetilde n}$, we 
have
\begin{align}  \
   \langle \alpha f(W')  v^{\prime *}, y' \rangle_{\cY^{\widetilde n}} &
   = \langle \alpha K_{W,v, y}(W') v^{\prime *}, y' 
   \rangle_{\cY^{\widetilde n}} \notag \\
&  = \langle \alpha K(W',W)(v^{\prime *}v) y, y' \rangle_{\cY^{\widetilde n}} 
\notag \\
 & = \langle K(\widetilde W',W)(\alpha v^{\prime *}v) y, y' 
 \rangle_{\cY^{\widetilde n}} \notag \\
 & = \langle K_{W,v,y}(\widetilde W')(\alpha v^{\prime *}), y' 
 \rangle_{\cY^{\widetilde n}} \notag \\
 & = \langle f(\widetilde W')(\alpha v^{\prime *}), y' 
 \rangle_{\cY^{\widetilde n}}
 \label{Kfnc}
 \end{align}
 The general case now follows by linearity and taking limits.

This completes the verification of statement (2) in the Theorem.
\end{proof}

\begin{proof}[Proof of (2) $\Rightarrow$ (3)]
    Given a function $K \colon \Omega \times \Omega \to 
    \cL(\cA, \cL(\cY))_{\rm nc}$ together with a Hilbert space of 
    global functions $\cH(K)$ for which properties (a), (b), (c) in 
    statement (2) of the Theorem hold, we must construct a Hilbert space $\cX$ equipped with a unital 
$*$-representation $\sigma \colon \cA \to \cL(\cX)$ together with a 
global function $H \colon \Omega \to \cL(\cX, \cY)_{\rm nc}$ so that 
we recover the given kernel $K$ from the Kolmogorov decomposition 
formula \eqref{Koldecom}.   To this end it is natural to choose 
$\cX = \cH(K)$ which is already equipped with the unital 
$*$-representation $\sigma$ given by \eqref{rep1} or 
\eqref{kernelaction}.  It remains to construct $H$.  

For this purpose it is convenient introduce some general notation.  
For $W \in \Omega_{m}$ and $u \in 
\cA^{m}$, we define the \textbf{directional point-evaluation operator}
$\bev_{W,u} \colon \cH(K) \to \cY^{m}$ by
\begin{equation}   \label{evWu}
    \bev_{W,u}(f) = f(W) u.
\end{equation}
More generally, for $W \in \Omega_{m}$ and $U = \begin{bmatrix} u_{1} 
& \cdots & u_{N}\end{bmatrix} \in \cA^{m \times N}$, we define 
$\bev_{W,U} \colon \cH(K)^{N} \to \cY^{m}$ by
\begin {equation}   \label{evWU}
   \bev_{W,U} = \bev_{W, \sbm{u_{1} & \cdots & u_{N}}} : =
   \begin{bmatrix} \bev_{W,u_{1}} & \cdots & \bev_{W, u_{N}} 
   \end{bmatrix},
\end{equation}
i.e.,  for $f = \sbm{f_{1} \\ \vdots \\ f_{N}} \in \cH(K)^{N}$, $W 
\in \Omega_{m}$, and $U = \begin{bmatrix} u_{1} & \cdots & u_{N} \end{bmatrix} \in 
\cA^{m \times N}$ we define
$$
\bev_{W,U} (f) = \sum_{i=1}^{N} \bev_{W,u_{i}} f_{i}  =  \sum_{i=1}^{N} f_{i}(W) u_{i}.
$$

For $W \in \Omega_{m}$ and $V = \sbm{ v_{1} \\ \vdots \\ v_{N}} \in 
\cA^{N \times m}$ (so each $v_{j} \in \cA^{1 \times m}$ for $j=1, 
\dots, N$), and $y \in \cY^{m}$, let us define $K_{W,V,y} \in 
\cH(K)^{N}$ by
$$
   K_{W,V,y} = \begin{bmatrix} K_{W,v_{1},y} \\ \vdots \\ K_{W, 
   v_{N},y}\end{bmatrix}.
$$
Then, for $f = \sbm{ f_{1} \\ \vdots \\ f_{N}} \in \cH(K)^{N}$, $V = 
\sbm{v_{1} \\ \vdots \\  v_{N} } \in \cA^{N \times m}$, $W \in 
\Omega_{m}$, and $y \in \cY^{m}$, we compute
$$
    \langle \bev_{W,V^{*}}f, y \rangle_{\cY^{m}}  = \sum_{i=1}^{N} 
    \langle f_{i}(W) v_{i}^{*}, y \rangle_{\cY^{m}} =
    \sum_{i=1}^{N} \langle f_{i}, K_{W, v_{i},y} \rangle_{\cH(K)}
    = \langle f, K_{W,V,y} \rangle_{\cH(K)^{N}}
$$
and we conclude that
\begin{equation}   \label{bevWV**}
    (\bev_{W,V^{*}})^{*} y = K_{W,V,y}
\end{equation}
Thus, for $Z \in \Omega_{n}$, $U = \sbm{ u_{1} \\ \vdots \\ u_{N}} 
\in \cA^{N \times n}$, $y' \in \cY^{n}$, $W \in 
\Omega_{m}$, $V = \sbm{ v_{1} \\ \vdots \\ v_{N}} \in \cA^{N 
\times m}$, and $y \in \cY^{m}$, we have
\begin{align}
\langle K_{W,V,y}, K_{Z,U,y'} \rangle_{\cH(K)^{N}} & =
\langle \bev_{Z,U^{*}}(K_{W,V,y}), \, y' \rangle_{\cY^{n}} \notag \\
& = \sum_{i=1}^{N} \langle \bev_{Z,u_{i}^{*}}(K_{W,v_{i},y}),\, y' 
\rangle_{\cY^{n}} \notag \\
& = \sum_{i=1}^{N} \langle K(Z,W) (u_{i}^{*} v_{i}) y, y' 
\rangle_{\cY^{n}} \notag \\
& = \langle K(Z,W) \left( \sum_{i=1}^{N} u_{i}^{*} v_{i}\right) y, y' 
\rangle_{\cY^{n}} \notag \\
& = \langle K(Z,W) (U^{*} V) y, y' \rangle_{\cY^{n}},
\label{kerinnerprod''}
\end{align}
the higher-rank generalization of the formula \eqref{kerinnerprod'}.

Next, let us note how the formula \eqref{kernelaction} for the action 
of $\sigma$ extends to the higher-rank case: for $W \in 
\Omega_{m}$, $V \in \cA^{N \times m}$, $y \in {\mathbb C}^{m}$, $P 
= [P_{ij}]_{1 \le i \le N', 1 \le j \le N} \in \cA^{N' \times N}$, 
and $y \in \cY^{m}$, we have
\begin{equation}   \label{kernelaction'}
    ({\rm id}_{{\mathbb C}^{N' \times N}} \otimes \sigma)(P) 
    K_{W,V,y} = K_{W, PV, y}.
\end{equation}

For $Z \in \Omega_{n}$ and $f \in \cH(K)^{n}$, we define $H(Z) \colon 
\cH(K)^{n} \to \cY^{n}$ simply as
\begin{equation} \label{defH}
    H(Z) = \bev_{Z,\,  1_{\cA^{n \times n}}} \colon \cH(K)^{n} \to 
    \cY^{n}.
\end{equation}
Then, for $Z \in \Omega_{n}$, $P  \in \cA^{n \times m}$, and $W \in 
\Omega_{m}$, we compute
\begin{align}
  &  H(Z) \left( {\rm id}_{{\mathbb C}^{n \times m}} \otimes 
    \sigma \right)(P) H(W)^{*} y = H(Z) \left( {\rm id}_{{\mathbb C}^{n 
    \times m}} \otimes \sigma \right)(P) K_{W,\, 1_{\cA^{m \times 
    m}}, y} \notag \\
   & \quad \quad \quad \quad \text{ (by \eqref{defH} and 
   \eqref{bevWV**}} \notag \\
& \quad = H(Z) K_{W, P, y} \text{ (by \eqref{kernelaction'})} \notag \\
& \quad = K(Z,W) \left( 1_{\cA^{n \times n}} \cdot P \right) y \text{ (by 
\eqref{defH} and \eqref{kerinnerprod''})} \notag  \\
& \quad = K(Z,W)(P)y
\label{modelKoldecom}
\end{align}
and it follows that $H$ defined as in \eqref{defH} provides the 
sought-after Kolmogorov decomposition \eqref{Koldecom} for $K$.

 It still remains to check that $H$ is a global function given that 
 each $f$ in $\cH(K)$ is a global function. In general let us use the 
 notation 
 \begin{equation}   \label{notation}
  E^{(n)}_{j} = j\text{-th column of } I_{n}
  \end{equation}
  where $I_{n}$ is the $n \times n$ identity matrix over ${\mathbb C}$.
 For $Z \in \Omega_{n}$ 
 $\widetilde Z \in \Omega_{\widetilde n}$, $f \in \cH(K)^{n}$, and 
 $\widetilde f \in \cH(K)^{\widetilde n}$, we compute
 \begin{align}
  &   H\left( \sbm{ Z & 0 \\ 0 & \widetilde Z} \right) \begin{bmatrix} 
     f \\ \widetilde f \end{bmatrix} =
 \left( \bev_{\sbm{ Z & 0 \\ 0 & \widetilde Z},\, 1_{\cA^{(n + 
 \widetilde n) \times (n + \widetilde n)}}} \right)  \begin{bmatrix} 
     f \\ \widetilde f \end{bmatrix}  \notag \\
     & \quad = \sum_{i=1}^{n} f_{i}\left(\sbm{ Z & 0 \\ 0 & 
     \widetilde Z} \right) \left(E^{(n + \widetilde n)}_{i} \otimes 
     1_{\cA}\right)
 + \sum_{j=1}^{\widetilde n} \widetilde f_{j}\left( \sbm{ Z & 0 \\ 0 
 & \widetilde Z} \right) \left( E^{(n + \widetilde n)}_{n+j} \otimes 
 1_{\cA} \right)  \notag \\
 & \quad  = \sum_{i=1}^{n} \begin{bmatrix} f_{i}(Z) & 0 \\ 0 & 
 f_{i}(\widetilde Z) \end{bmatrix} \left( E^{(n + \widetilde n)}_{i} 
  \otimes 1_{\cA} \right) + \sum_{j=1}^{\widetilde n} \begin{bmatrix} \widetilde 
 f_{j}(Z) & 0 \\ 0 & \widetilde f_{j}(\widetilde Z) \end{bmatrix}  
 \left( E^{(n + \widetilde n)}_{n+j} \otimes 1_{\cA} \right)  \notag \\
 &  \quad = \begin{bmatrix} \sum_{i=1}^{n} f_{i}(Z) (E^{(n)}_{i} 
 \otimes 1_{\cA}) \\
   \sum_{j=1}^{\widetilde n} \widetilde f_{j}(\widetilde Z) 
   (E^{(\widetilde n)}_{j} \otimes 1_{\cA}) \end{bmatrix}  =
   \begin{bmatrix} H(Z) & 0 \\ 0 & H(\widetilde Z) \end{bmatrix}
       \begin{bmatrix} f \\ \widetilde f \end{bmatrix}
\label{compute1}
 \end{align}
 where we used that each $f_{i}$ and $\widetilde f_{j}$ are global 
 functions in the third line of the computation. 
 It follows that $H$ is a global function as wanted.

Finally, we now add the assumption that $\Omega \subset \cV_{\rm nc}$ 
where $\cV$ is a vector space and assume that each $f \in \cH(K)$ is 
a nc function.  The goal is to show that then $H$ is a nc function.
Toward this end, we suppose that $Z \in \Omega_{N}$, $\widetilde Z 
\in \Omega_{M}$ and $\alpha \in {\mathbb C}^{M \times N}$ are such that 
$\alpha Z = \widetilde Z \alpha$, and that $f = \sbm{ f_{1} \\ \vdots \\ f_{N}} 
\in \cH(K)^{N}$.  The following computation verifies the 
desired result:
\begin{align}
  &  \alpha H(Z)(f)  = \alpha \left( \sum_{j=1}^{N} f_{j}(Z) (E^{(N)}_{j} 
    \otimes 1_{\cA}) \right)   
     = \sum_{j=1}^{N} f_{j}(\widetilde Z)(\alpha E^{(N)}_{j} \otimes 
    1_{\cA})  \notag \\
    & \quad  = \sum_{j=1}^{N} f_{j}(\widetilde Z) \left(\sum_{i=1}^{M} \alpha_{ij} 
    E^{(M)}_{i} \otimes 1_{\cA} \right) 
     = \sum_{i=1}^{M} \sum_{j=1}^{N} f_{j}(\widetilde Z)(\alpha_{ij} 
    E^{(M)}_{i} \otimes 1_{\cA}) \notag \\
    &\quad  = \sum_{i=1}^{M} \left( \sum_{j=1}^{N} \alpha_{ij} f_{j}(\widetilde 
    Z) (E^{(M)}_{i} \otimes 1_{\cA}) \right)  = H(\widetilde Z) (\alpha f).
    \label{compute2}
\end{align}
This completes the proof of (2) $\Rightarrow$ (3) in Theorem 
\ref{T:cpker}.
\end{proof}

\begin{proof}[Proof of (3) $\Rightarrow$ (1)]
    We suppose that the kernel $K(Z,W)$ has a Kolmogorov 
    decomposition \eqref{Koldecom}.  
For $P \succeq 0$ in $\cA^{n \times n}$, write $P$ in factored form 
as $P = R^{*}R$ for some $R \in \cA^{n \times n}$.
Then the computation
\begin{align*}
    K(Z,Z)(P) & =   H(Z) ({\rm id}_{{\mathbb C}^{n \times n}} \otimes 
    \sigma)(R^{*}R) H(Z)^{*} \\
 & = H(Z)\left( ({\rm id}_{{\mathbb C}^{n \times n}} \otimes 
 \sigma)(R) \right)^{*} ({\rm id}_{{\mathbb C}^{n \times n}} \otimes 
 \sigma)(R)  H(Z)^{*}
 \end{align*}
shows that $K(Z,Z)(P)$ is positive.

We check that $K$ is a global kernel if $H$ is a global function. 
Indeed, for $Z^{(i)} \in \Omega_{n_{i}}$, $W^{(j)} \in \Omega_{m_{j}}$ and 
$P_{ij} \in \cA^{n_{i} \times m_{j}}$ for $i,j = 1,2$,  we have
\begin{align*}
    & K\left( \sbm{ Z^{(1)} & 0 \\ 0 & Z^{(2)}}, \sbm{ W^{(1)} & 0 \\ 0 & 
    W^{(2)}} \right)\left( \sbm{ P_{11} & P_{12} \\ P_{21} & P_{22}} 
    \right) \\ 
    &  = 
    H\left( \sbm{Z^{(1)} & 0 \\ 0 & Z^{(2)}} \right) \left( {\rm 
    id}_{{\mathbb C}^{(n_{1} + n_{2}) \times (m_{1} \times m_{2})}}
    \otimes \sigma \right) \left( \sbm{ P_{11} & P_{12} \\ P_{21} & P_{22} } 
    \right)  H \left( \sbm{ W^{(1)}& 0 \\ 0 & W^{(2)}} \right)^{*} 
    \\
  &  = \begin{bmatrix} H(Z^{(1)}) & 0 \\ 0 & H(Z^{(2)}) 
\end{bmatrix}  \begin{bmatrix} ({\rm id} \otimes \sigma)(P_{11}) & 
({\rm id} \otimes \sigma)(P_{12}) \\ 
({\rm id}  \otimes \sigma)(P_{21}) & ({\rm id} \otimes \sigma)(P_{22}) 
\end{bmatrix}  \begin{bmatrix} H(W^{(1)})^{*} & 0 \\ 0 & 
H(W^{(2)})^{*} \end{bmatrix} \\
& = \begin{bmatrix} K(Z^{(1)}, W^{(1)})(P_{11}) &
 K(Z^{(1)}, W^{(2)})(P_{12}) \\  K(Z^{(2)}, W^{(1)})(P_{21}) &
  K(Z^{(2)}, W^{(2)})(P_{22}) \end{bmatrix}
\end{align*}
as required.

Finally, we suppose that $\Omega \subset \cV_{\rm nc}$ for a vector 
space $\cV$ and that $H$ is a nc function.  The following calculation 
shows that then $K$ is a nc kernel.  We suppose that we are given
$Z \in \Omega_{n}$, $\widetilde Z \in \Omega_{\widetilde n}$ and $\alpha 
\in {\mathbb C}^{\widetilde n \times n}$ such that $\alpha Z = \widetilde 
Z \alpha$, along with $W \in \Omega_{m}$, $\widetilde W \in 
\Omega_{\widetilde m}$ and $\beta \in {\mathbb C}^{\widetilde m \times 
m}$ with $\beta W = \widetilde W \beta$.  We then use the Kolmogorov 
decomposition to compute
\begin{align*}
    \alpha \, K(Z,W)(P) \, \beta^{*} & = \alpha H(Z) ( {\rm id}_{{\mathbb C}^{n 
    \times m}} \otimes 
    \sigma)(P) H(W)^{*} \beta^{*} \\
& = H(\widetilde Z) \alpha \left( {\rm id}_{{\mathbb C}^{n \times m}} \otimes \sigma)(P) \right) 
\beta^{*} H(\widetilde W)^{*} \\
& = H(\widetilde Z) ( {\rm id}_{{\mathbb C}^{\widetilde n \times 
\widetilde m}} \otimes \sigma)(\alpha P \beta^{*})  H(\widetilde W)^{*} \\
& = K(\widetilde Z, \widetilde W)(\alpha P \beta^{*}).
\end{align*}
This completes the proof of (3) $\Rightarrow$ (1) in Theorem 
\ref{T:cpker}.
\end{proof}

\begin{remark}  \label{R:modelKol}
     The proof of (2) $\Rightarrow$ (3) in Theorem \ref{T:cpker} 
    implies that we may take a canonical form for the nc Kolmogorov 
    decomposition \eqref{Koldecom}, namely:  we may take the state 
    space $\cX$ in \eqref{Koldecom} to be the nc reproducing kernel 
    Hilbert space $\cH(K)$, the representation $\sigma_{\cX}$ to be the 
    canonical representation $\sigma$ \eqref{rep1} on $\cH(K)$, and 
    the nc function $H \colon \Omega \to \cL(\cH(K), \cY)_{\rm nc}$  
    to have the concrete form \eqref{defH}.
\end{remark}

To complement the understanding of global/nc kernels, we present the 
following converse to statement (2) in Theorem \ref{T:cpker}.

\begin{theorem}  \label{T:RKHS}
Suppose that $\cH$ is a Hilbert space whose elements consist of 
global functions $f$ from the nc set $\Omega  \subset \cS_{\rm nc}$ to
$\cL(\cA, \cY)$ (where $\cA$ is a $C^{*}$-algebra and $\cY$ is a 
coefficient Hilbert space) such that
\begin{enumerate}
    \item for each $W \in \Omega_{m}$, the map $f \mapsto f(W)$ is 
    bounded as an operator from $\cH$ to $\cL(\cA, \cY)^{m \times m} \cong 
    \cL(\cA^{m}, \cY^{m})$, and
     \item the mapping $\sigma \colon \cA \to \cL(\cH)$ given by
\begin{equation}   \label{rep2}
    (\sigma(a) f)(W)(u) = f(W)(ua)
\end{equation}
(for $f \in \cH$, $W \in \Omega_{m}$, $u \in \cA^{m}$)
defines a unital $*$-representation of $\cA$.
\end{enumerate}
Then there is a cp global kernel $K$ so that $\cH$ is isometrically 
equal to the global reproducing kernel Hilbert space $\cH(K)$ defined 
as in statement (2) of Theorem \ref{T:cpker}.  Furthermore, if 
$\Omega \subset \cV_{\rm nc}$ for a vector space $\cV$ and the 
elements of $\cH$ are nc functions, then $K$ is a cp nc kernel.
\end{theorem}

\begin{remark} \label{R:RKHS}  Part of the assumption in condition 
    (2) in Theorem \ref{T:RKHS} is that the space $\cH$ is invariant 
    under the $\cA$ action $f \mapsto \sigma(a) f$ given by formula 
    \eqref{rep2}.  If one assumes only this invariance condition, one 
    gets as a consequence of the closed graph theorem that each 
    $\sigma(a)$ is a bounded linear operator on $\cH$.  It is also 
    immediate that $\sigma$ is multiplicative:  $\sigma(a_{1} a_{2}) 
    = \sigma(a_{1}) \sigma(a_{2})$.  However in general it need not 
    be the case that $\sigma$ preserve adjoints ($\sigma(a)^{*} = 
    \sigma(a^{*})$).  Part of the content of the hypothesis in
    statement (2) in Theorem \ref{T:RKHS} is that this is the case. 
    \end{remark}
    
    \begin{proof}[Proof of Theorem \ref{T:RKHS}]
Choose $W \in \Omega_{m}$, $v \in \cA^{1 \times m}$, and $y \in \cY^{m}$.  Then 
$f(W) \in \cL(\cA, \cY)^{m \times m} \cong \cL(\cA^{m}, \cY^{m})$ so
$f(W)(v^{*}) \in \cY^{m}$.  As $f \to f(W)$ is bounded as a linear 
operator from $\cH$ to $\cL(\cA^{m}, \cY^{m})$, it follows that
$f \mapsto \langle f(W) v^{*}, y \rangle_{\cH}$ is a bounded linear 
functional on $\cH$.  By the Riesz-Frechet theorem there therefore is 
a $K_{W,v,y} \in \cH$ so that
\begin{equation}   \label{reprod3}
    \langle f(W)v^{*}, y \rangle_{\cY^{m}} = \langle f, K_{W,v,y} 
    \rangle_{\cH}.
\end{equation}
We call these special functions $K_{W,v,y}$ in $\cH$ \textbf{kernel 
elements}.  It is now a simple matter to compute the action of 
$\sigma(a)$ on a kernel element (for any $a \in \cA$):
\begin{align*}
    \langle \sigma(a) f, K_{W, v, y} \rangle_{\cH} & = 
    \langle (\sigma(a) f)(W)(v^{*}), y \rangle_{\cY^{m}}  = 
 \langle f(W)(v^{*} a), y \rangle_{\cY^{m}} \\
 & = \langle f,\, K_{W, a^{*} v, y} \rangle_{\cH}
 \end{align*}
 from which we conclude that
\begin{equation}  \label{*rep}
\sigma(a)^{*} \colon K_{W,v,y} \mapsto K_{W, a^{*}v, y}.
\end{equation}
By hypothesis, $\sigma(a)^{*} = \sigma(a^{*})$.  By replacing $a$ by 
$a^{*}$ we see that
\begin{equation}  \label{ker-rep}
  \sigma(a) \colon K_{W,v,y} \mapsto K_{W, av, y}.
\end{equation}

For $V = \sbm{ v_{1} \\ \vdots \\ v_{N}} \in \cA^{N \times m}$, $W \in 
\Omega_{m}$ and $y \in \cY^{m}$, we define $K_{W,V,y} \in \cH^{N} = 
\bigoplus_{j=1}^{N} \cH$ by
\begin{equation}   \label{ker-enlarged}
  K_{W,V,y} = \begin{bmatrix}  K_{W,v_{1},y} \\ \vdots \\ K_{W, 
  v_{N},y} \end{bmatrix}.
\end{equation}
We extend the representation $\sigma$ to elements of $\cA^{M \times 
N}$ in the entrywise way:
$$
  ({\rm id}_{{\mathbb C}^{M \times 
  N}} \otimes \sigma)(P) : = [ \sigma(P_{ij})] \in \cL(\cH)^{M \times 
  N}.
$$
Then the formula extends to the matricial form
\begin{equation}   \label{ker-matrep}
    \sigma(U) K_{W,V,y} = K_{W,\, U V, \, y} \in \cH^{N'}
\end{equation}
for $W \in \Omega_{m}$, $V \in \cA^{N \times m}$, $y \in \cY^{m}$, $U 
\in \cA^{N' \times N}$.
Furthermore, non-square versions of the unital $*$-representation properties are 
preserved:
\begin{align*}
   & ({\rm id}_{{\mathbb C}^{M \times N}} \otimes \sigma)(P)^{*}  = 
  ( {\rm id}_{{\mathbb C}^{N \times M}} \otimes 
    \sigma)(P^{*}),\\
 & ({\rm id}_{{\mathbb C}^{M \times N}} \otimes \sigma)(P_{1}P_{2}) = 
 ({\rm id}_{{\mathbb C}^{M \times K}} \otimes \sigma)(P_{1}) ({\rm 
 id}_{{\mathbb C}^{K \times N}} \otimes \sigma)(P_{2}) \\
 & \quad 
\text{ if } P_{1} \in \cA^{M \times K},\, P_{2} \in \cA^{K \times N}, \\
& ({\rm id}_{{\mathbb C}^{N \times N}} \otimes \sigma) (I_{N} \otimes 1_{\cA}) = I_{\cH^{N}}.
\end{align*}

For 
$$ W \in \Omega_{m}, \quad  v \in \cA^{1 \times m}, \quad y \in 
\cY^{m},  \quad
W' \in \Omega_{m'}, \quad  v \in \cA^{1 \times m'}, \quad y \in 
\cY^{m'},
$$
we can now compute the inner product of the associated kernel 
elements of $\cH$ by using these matricial versions of the 
representation $\sigma$ as follows:
\begin{align}
    \langle K_{W,v,y}, K_{W', v', y'} \rangle_{\cH} & =
    \langle \sigma(v) K_{W, I_{m} \otimes 1_{\cA}, y}, \,
       \sigma(v') K_{W', I_{m'} \otimes 1_{\cA}, y'} \rangle_{\cH} 
       \notag \\
 & = \langle \sigma(v')^{*} \sigma(v) K_{W, I_{m} \otimes 1_{\cA}, y}, \,
  K_{W', I_{m'} \otimes 1_{\cA}, y'} \rangle_{\cH^{m'}} \notag  \\
  & =  \langle \sigma(v^{\prime *} v) K_{W, I_{m} \otimes 1_{\cA}, y}, \,
  K_{W', I_{m'} \otimes 1_{\cA}, y'} \rangle_{\cH^{m'}}.
  \label{vprime*v}
\end{align}
We conclude that the result depends on $v$ and $v'$ only through the 
product $v^{\prime *} v$.

For $W \in \Omega_{m}$ and $u \in \cA^{m}$, let us define 
\textbf{directional point-evaluation operator}  by the same formula 
\eqref{evWu} as used for the case where $\cH = \cH(K)$, namely 
 define $\boldsymbol{\rm ev}_{W,u} \in \cL(\cH, \cY^{m})$ by 
\begin{equation}   \label{LWv}
   \bev_{W,u} \colon f \mapsto f(W)(u).
\end{equation}
From the formula \eqref{reprod3} we see that  
$K_{W,v,y} = (\bev_{W,v^{*}})^{*} y$ and hence
\begin{align*}
&  \langle K_{W,v,y}, K_{W', v',y'} \rangle_{\cH}: = \langle  
  (\bev_{W,v^{*}})^{*} y,  \, (\bev_{W', v^{\prime *}})^{*} y' 
  \rangle_{\cH}  \\
  & \quad = \langle 
  (\bev_{W',v^{\prime *}})(\bev_{W,v^{*}})^{*}y, y' \rangle_{\cY^{m'}}.
\end{align*}
	
As we have already observed from \eqref{vprime*v} that the dependence of the 
inner product $\langle K_{W,v,y}, K_{W',v',y'} \rangle_{\cH}$ on $v, v'$ is only 
through the product $v^{\prime *} v \in \cA^{m' \times m}$, we deduce that there is an 
$\cL(\cY^{m}, \cY^{m'})$-valued function $K^{\circ}$ of three arguments 
$(W',W, v^{\prime*}v)$ defined by
\begin{equation} \label{graded1}
 K^{\circ}(W',W, v^{\prime *}v) =  (\bev_{W',v^{\prime 
 *}})(\bev_{W,v^{*}})^{*} \in 
 \cL(\cY^{m}, \cY^{m'}) \cong \cL(\cY)^{m' \times m}.
\end{equation}
For a given $W' \in \Omega_{m'}$ and $W \in \Omega_{m}$, the function 
$K^{\circ}(W',W, P)$ is defined only for $P \in \cA^{m' \times m}$ 
having a column-row vector factorization $P = v^{\prime *} v$ ($v' \in 
\cA^{1 \times m'}$ and $v \in \cA^{1 \times m}$). 
We 
extend $K(W',W, \cdot)$ so as to be defined on all of $\cA^{m' \times 
m}$ by making use of the  higher-rank kernel elements $K_{W,V,y}$ 
with $V \in \cA^{N \times m}$ \eqref{ker-enlarged}.  
Given two such kernel elements $K_{W,V,y}$ and $K_{W',V',y'}$ (where 
$V \in \cA^{N \times m}$ and $V' \in \cA^{N \times m'}$), the 
$\cH^{N}$-inner product works out to be
\begin{align}
&\langle K_{W,V,y},\, K_{W',V',y'} \rangle_{\cH^{N}}  =
\sum_{i=1}^{N} \langle K_{W,v_{i},y},\, K_{W', v'_{i},y'} 
\rangle_{\cH} \notag \\
& \quad = \sum_{i=1}^{N} \langle K^{\circ}(W',W, v^{\prime *}_{i} 
v_{i})y, y'\rangle_{\cY^{m'}} =: \langle K^{\circ}(W', W, V^{\prime 
*}V) y, y' \rangle_{\cY^{m'}}.
\label{Vprime*V}
\end{align}
We conclude that the inner product $\langle K_{W,V,y}, 
K_{W',V',y'} \rangle_{\cH^{N}}$ has $(V,V')$ dependence only as a function 
of the matrix product $V^{\prime *} V$. We then use linearity to extend the function 
$K^{\circ}$ (still denoted as $K^{\circ}$) 
defined above to allow the third argument to have any 
factorization $V^{\prime *} V$ with $V \in \cA^{N \times m}$ and $V' 
\in \cA^{N \times m'}$ for some common $N \in {\mathbb N}$
if $W \in \Omega_{m}$ and $W' \in \Omega_{m'}$ with the result that 
the identity \eqref{Vprime*V} holds.

We wish to check next that $K^{\circ}$ is linear in its third argument $P = 
V^{\prime *} V \in \cA^{m' \times m}$.
Toward this end, fix $W \in 
\Omega_{m}$ and $W' \in \Omega_{m'}$ and suppose that we have 
$\cA$-matrices
$$
V_{1} \in \cA^{N_{1} \times m}, \quad
V_{1}' \in \cA^{N_{1} \times m'}, \quad 
V_{2} \in \cA^{N_{2} \times  m}, \quad
V_{2}' \in \cA^{N_{2} \times m'}.
$$
We then compute, for $y \in \cY^{m}$ and $y' \in \cY^{m'}$,
\begin{align*}
   & \langle K(W',W, V_{1}^{\prime *} V_{1} + V_{2}^{\prime *} V_{2}) 
    y, y' \rangle_{\cY^{m'}}  =
  \left\langle K\left(W', W, \left( \sbm{ V_{1}' 
 \\ V_{2}' }\right)^{*} \sbm{V_{1} \\ V_{2} } \right) y, y' 
 \right\rangle_{\cY^{m'}}  \\
 & \quad = \left\langle K_{W, \sbm{ V_{1} \\ V_{2}}, y}, \,
 K_{W', \sbm{ V_{1}' \\ V_{2}'}, y'} \right\rangle_{\cH^{N_{1} + 
 N_{2}}} \\
 & \quad = \langle K_{W, V_{1},y}, \, K_{W',V_{1}',y'} 
 \rangle_{\cH^{N_{1}}} + \langle K_{W,V_{2},y}, \, K_{W', V_{2}', y'} 
 \rangle_{\cH^{N_{2}}}  \\
 & \quad = \langle \left( K(W',W, V_{1}^{\prime *} V_{1}) + 
 K(W', W,  V_{2}^{\prime *} V_{2}) \right) y, y' \rangle_{\cY^{m'}}.
\end{align*}
thereby proving the additivity property
$$
  K(W',W, P_{1} + P_{2}) = K(W',W, 
  P_{1}) + K(W',W, P_{2}) \text{ for }
  P_{1}, P_{2} \in \cA^{m' \times m}.
$$
We therefore write
\begin{equation}   \label{graded}
K(W',W)(P) = K^{\circ}(W',W,P) \text{ where } K(W',W) \in \cL(\cA^{m' \times 
m}, \cL(\cY)^{m' \times m})
\end{equation}
in case $W' \in \Omega_{m'}$, $W \in \Omega_{m}$, the generalization 
of \eqref{graded1} to the case where $P$ has higher rank over $\cA$.
One can check that $K(W',W)$ is indeed a bounded operator by using
the identity \eqref{Vprime*V} and
the fact that the map $f \in \cH$ to $f(W) \in \cL(\cA^{m}, \cY^{m})$ is bounded 
(for $W \in \Omega_{m}$). 
We now rewrite the formula \eqref{Vprime*V} as
\begin{equation}  \label{Vprime*V'}
    \langle K_{W,V,y}, K_{W',V',y'} \rangle_{\cH^{N}} = 
    \langle K(W',W)(V^{\prime *}V) y, y' \rangle_{\cY^{m'}}
\end{equation}
for $W \in \Omega_{m}$, $V \in \cA^{N \times m}$, $y \in \cY^{m}$, 
$W' \in \Omega_{m'}$, $V' \in \cA^{N \times m'}$, $y' \in \cY^{m'}$.  
If we introduce the higher-rank version of the directional 
point-evaluation operator $\bev_{W,u}$ 
defined in \eqref{LWv}, namely, the operator $\bev_{W,U} \in 
\cL(\cH^{N}, \cY^{m})$ given by
\begin{equation}   \label{LWV}
    \bev_{W,U} = \begin{bmatrix} \bev_{W,u_{1}} & \cdots & 
    \bev_{W,u_{N}} \end{bmatrix} \colon f \mapsto \sum_{i=1}^{N} 
    f_{i}(W)(u_{i}) \text{ if } U = \begin{bmatrix} u_{1} & \cdots 
    & u_{n} \end{bmatrix}
\end{equation}
then we have a succinct formula for the kernel $K$:
\begin{equation}   \label{graded2}
    K(W',W)(V^{\prime *} V) = ( \bev_{W',V^{\prime *}}) (\bev_{W,V^{*}})^{*}.
\end{equation}
From this formula we immediately see that $K$ is Hermitian, i.e.
\begin{equation}   \label{Hermitian}
    K(Z,W)(P)^{*} = K(W,Z)(P^{*}).
\end{equation}

To verify that $K$ is a cp kernel, simply note that
$$
  \langle K(W,W)(V^{*}V) y, y \rangle_{\cY^{m}}
  = \| K_{W,V,y} \|^{2}_{\cH^{N}} \ge 0
$$
for $W \in \Omega_{m}$, $V \in \cA^{N \times m}$, $y \in \cY^{m}$.
It is also possible to verify the expanded cp condition 
\eqref{cpker-expanded} by considering the norm-squared of a linear 
combination of kernel elements in $\cH^{N}$, but, as explained in 
Proposition \ref{P:cpkernel}, this follows automatically once we 
verify that $K$ is a global kernel.

The graded property of  $K$ (property \eqref{kergraded}) was already 
noted above (see \eqref{graded}).  Therefore to check that $K$ is a 
global kernel, it remains only to check  the ``respects direct sums'' condition 
\eqref{kerds}.   Since we have already noted that $K$ is Hermitian 
(see \eqref{Hermitian}), it suffices to check that $K$ respects 
direct sums in the first argument, i.e.:
\begin{equation}  \label{toshow1}
    K\left( \sbm{ Z & 0 \\ 0 & \widetilde Z }, W \right)\left( \sbm{ P 
     \\ \widetilde P} \right) =
    \begin{bmatrix} K(Z,W)(P) \\ K(\widetilde Z, W) (\widetilde P)  
	\end{bmatrix}.
\end{equation}
for $Z \in \Omega_{n}$, $\widetilde Z \in \Omega_{\widetilde n}$, $W 
\in \Omega_{m}$, $P \in \cA^{n \times m}$, $\widetilde P \in 
\cA^{\widetilde n \times m}$.

Toward this end, we choose any factorization
$$
   \begin{bmatrix} P \\ \widetilde P \end{bmatrix} = \begin{bmatrix} 
   U^{*} \\ \widetilde U^{*} \end{bmatrix} V
$$
with $U \in \cA^{N \times n}$, $\widetilde U \in \cA^{N \times 
\widetilde n}$, $V \in \cA^{N \times m}$.  Making use of the formula 
\eqref{graded2}, we see that the  desired identity \eqref{toshow1} 
comes down to
$$
    \left( \bev_{\sbm{Z & 0 \\ 0 & \widetilde Z}, \sbm{ U^{*} \\ 
    \widetilde U^{*}}} \right) \left( \bev_{W,V^{*}} \right)^{*}
= \begin{bmatrix} \left( \bev_{Z,U} \right) \left( 
\bev_{W,V^{*}}\right)^{*} \\ ( \bev_{\widetilde Z, \widetilde U^{*}}) 
( \bev_{W, V^{*}} )^{*} \end{bmatrix}
= \begin{bmatrix}  \bev_{Z,U^{*}} \\ \bev_{\widetilde Z, \widetilde 
U^{*}} \end{bmatrix} (\bev_{W,V^{*}})^{*}
$$
Canceling off the common right factor leaves us with
\begin{equation}   \label{toshow2}
\bev_{\sbm{ Z & 0 \\ 0 & \widetilde Z}, \sbm{ U^{*} \\ \widetilde U^{*}}} =
\begin{bmatrix} \bev_{Z,U^{*}} \\ \bev_{\widetilde Z, \widetilde U^{*}} 
\end{bmatrix}.
\end{equation}
To verify \eqref{toshow2}, it suffices to show that it holds when 
applied to a generic element $f$ of $\cH^{N}$, namely, we wish to 
verify
\begin{equation}   \label{toshow3}
\bev_{\sbm{ Z & 0 \\ 0 & \widetilde Z}, \sbm{ U^{*}\\ \widetilde U^{*}}} f = 
\begin{bmatrix} \bev_{Z,U^{*}} \\ \bev_{\widetilde Z, \widetilde U^{*}} \end{bmatrix}f
\end{equation} 
for all $f \in \cH^{N}$.  Let us write out the rows of $U^{*}$ and of 
$\widetilde U^{*}$ as $u_{1}^{*}, \dots, u_{n}^{*}$ and $\widetilde 
u_{1}^{*}, \dots, \widetilde u_{\widetilde n}^{*}$ os that
$$
   U^{*} = \sbm{ u_{1}^{*} \\ \vdots \\ u_{n}^{*}}, \quad
\widetilde U = \sbm{\widetilde u_{1}^{*} \\ \vdots \\ \widetilde u_{\widetilde n}^{*}}.
$$
Then verification of \eqref{toshow3} amounts to a mild generalization 
of the  computation \eqref{compute1} where $1_{\cA^{(n + \widetilde n) \times (n + 
\widetilde n)}}$ is replaced by $\sbm{ U^{*} \\ \widetilde U^{*}}$ as 
well as $E_{i}^{(n)} \otimes 1_{\cA}$ replaced by $u_{i}^{*}$ and $E_{j}^{(\widetilde 
n)} \otimes 1_{\cA}$ replaced by $\widetilde u_{j}^{*}$.

Let us now assume that each element $f$ of $\cH$ respects 
intertwinings \eqref{funcintertwine}.  We claim that the kernel $K$ respects 
intertwinings in the first argument, i.e.,
\begin{align}
 &   Z \in \Omega_{n},\, \widetilde Z \in \Omega_{\widetilde n}, \, \alpha 
    \in {\mathbb C}^{\widetilde n \times n} \text{ such that } \alpha Z = 
    \widetilde Z \alpha, \, W \in \Omega_{m}, \, P \in \cA^{n \times m} 
    \notag \\
    & \quad 
    \Rightarrow \alpha \, K(Z,W)(P)  = K(\widetilde Z, W)(\alpha P).
    \label{toshow-a}
\end{align}
From the formula \eqref{graded2}, we see that the conclusion of 
\eqref{toshow-a}
is equivalent to 
$$
  \alpha \left( \bev_{Z, V'^{*}}\right) \left( \bev_{W, V^{*}} \right)^{*} 
  = \left( \bev_{\widetilde Z, \alpha V'^{*}} \right) \left( \bev_{W, 
  V^{*}}\right)^{*}.
$$
Canceling off the common right factor converts this to
$$  \alpha \left( \bev_{Z,V'^{*}} \right) = \bev_{\widetilde Z, \alpha 
V^{\prime *}},
$$
i.e., 
\begin{equation}  \label{toshow-b}
    \alpha \left( \bev_{Z, V^{\prime *}} \right) f = \left( 
    \bev_{\widetilde Z, \alpha V^{\prime *}} \right) f
\end{equation}
for a general $f = \sbm{f_{1} \\ \vdots \\ f_{N}} \in \cH^{N}$ ($N$ equal to the number of rows in 
$V^{\prime *}$).  From the definition \eqref{LWV} of $\bev_{Z, 
V^{\prime *}}$, we see that \eqref{toshow-b} amounts to
\begin{equation}  \label{toshow-c}
    \alpha \, \sum_{i=1}^{N} f_{i}(Z) v_{i}^{\prime *} = \sum_{i=1}^{N} 
    f_{i}(\widetilde Z) \alpha v_{i}^{\prime *}.
\end{equation}
We now recall the hypothesis in \eqref{toshow-a}, namely, that $Z, \, 
\widetilde Z$ are in $\Omega$ with $\alpha Z =\widetilde Z \alpha$. As 
each $f_{i} \in \cH$ as a nc function, each $f_{i}$ in particular 
respects intertwinings \eqref{funcintertwine}. Hence
$$
\alpha \sum_{i=1}^{N} f_{i}(Z) v_{i}^{\prime *} = \sum_{i=1}^{N} \left( \alpha 
f_{i}(Z) \right)v_{i}^{\prime *} = \sum_{i=1}^{N} \left( f_{i}(\widetilde Z) \alpha \right)v_{i}^{\prime *} =
\sum_{i=1}^{N} f_{i}(\widetilde Z) (\alpha v_{i}^{\prime *} )
$$
and \eqref{toshow-c} (and then also \eqref{toshow-a}) follows as 
claimed. 
The Hermitian property \eqref{Hermitian} of $K$ then implies that $K$ 
has the full kernel ``respects intertwining'' property 
\eqref{kerHerintertwine}, and hence $K$ is a nc kernel.

From the formulas \eqref{Vprime*V'}, \eqref{reprod3}, \eqref{*rep}, 
we see that $K$ meets all the conditions in part (2) of Theorem 
\ref{T:cpker} to serve as the reproducing kernel for the functional 
Hilbert space $\cH$, so $\cH = \cH(K)$ identically and isometrically.
\end{proof}

Theorem \ref{T:RKHS} gives the existence of a reproducing kernel for 
a space satisfying the hypotheses of the theorem but does not give 
much information on how to actually compute $K$.  The next two 
results fill in this gap.

In the next result we use the following convention.  If ${\mathbf y}$ 
is a vector in $\cY^{m}$, we can view ${\mathbf y}$ as an operator 
from ${\mathbb C}$ to $\cY^{m}$ (the {\em column} operator space 
structure for the Hilbert space $\cY$).  The adjoint operator 
${\mathbf y}^{*} \colon \cY^{m} \to {\mathbb C}$ is then given by 
\begin{equation}   \label{conv}
{\mathbf y}^{*} \colon y \mapsto 
\langle y, {\mathbf y} \rangle_{\cY^{m}}.
\end{equation}
We apply this notion in 
particular to the case where ${\mathbf y} = f_{i}(W)(v_{j}^{*})$ where 
$f_{i}$ is a global/nc function on $\Omega$, $W \in \Omega_{m}$, 
and $v_{j}^{*}$ is a vector in $\cA^{m}$ (so $f_{i}(W) \in \cL(\cA^{m}, \cY^{m})$ 
and $f_{i}(W)(v_{j}^{*}) \in \cY^{m}$).  Finally if $V = \sbm{ 
v_{1} \\ \vdots \\ v_{N}}$ is an $N \times m$ matrix over $\cA$, we 
let $f_{i}(W)(V^{*})$ denote the block row matrix
$$
  f_{i}(W)(V^{*}) = \begin{bmatrix} f_{i}(W)(v^{*}_{1}) & 
  \cdots & f_{i}(W)(v^{*}_{N})\end{bmatrix} \in \cY^{m \times 
  N}.
$$

\begin{theorem}  \label{T:Bergmann}
    Suppose that $\cH$ is a Hilbert space of global/nc functions from 
    $\Omega$ to $\cL(\cA, \cY)_{\rm nc}$ equipped with a unital 
    $*$-representation $\sigma$ \eqref{rep2} as in Theorem \ref{T:RKHS}.
    Let $\{f_{i}\}_{i \in \cI}$ be an orthonormal basis for $\cH$.
    Then the kernel function $K(Z,W)(P)$ for $\cH$ is given by
 \begin{equation}  \label{Bergman-ker}
     K(W',W)(V^{\prime *} V) = \sum_{i \in \cI} \left(f_{i}(W') 
     (V^{\prime *}) \right) \left(f_{i}(W)(V^{*}) \right)^{*} \in 
     \cL(\cY^{m}, \cY^{m'})
  \end{equation}
  for $W' \in \Omega_{m'}$,  $W \in \Omega_{m}$, $V' \in \cA^{N \times 
  m' }$, $V \in \cA^{N \times m}$, with the series converging in the 
  weak operator topology. 
 \end{theorem}

 \begin{proof}  As we have seen in \eqref{Vprime*V}, for $y \in 
     \cY^{m}$, $W \in \Omega_{m}$, $V \in \cA^{N \times m}$, $y' \in \cY^{m'}$, $W' \in 
     \Omega_{m'}$, and $V' \in \cA^{N \times m'}$ we have
  \begin{equation}   \label{KerInProd}
  \langle K(W',W)(V^{\prime *} V) y, y' \rangle_{\cY^{m'}} =
  \langle K_{W,V,y},\, K_{W',V',y'} \rangle_{\cH(K)^{N}}
  \end{equation}
  where $K_{W,V,y}$ is the element of $\cH(K)^{N}$ with the 
  reproducing property
  \begin{equation}        \label{reprod5}
  \left\langle \sbm{h_{1} \\ \vdots \\ h_{n}}, \, K_{W,V,y} 
  \right\rangle_{\cH(K)^{N}} = \sum_{j=1}^{N} \langle h_{j}(W) v_{j}^{*}, 
  y\rangle_{\cY^{m}}.
  \end{equation}  
  if we write out $V \in \cA^{N \times m}$ in terms of its rows as
  $V = \sbm{v_{1} \\ \vdots \\ v_{N}}$. Recalling the notation 
  \eqref{notation}, we see that  
  $\{f_{i} \otimes E^{(N)}_{j} \colon i \in \cI, 1 \le j \le N\}$ is an orthonormal basis for 
  $\cH(K)^{N}$.   By the Parseval equality,
 \begin{align*}
   &  \langle K_{W,V,y}, \, K_{W',V',y'} \rangle_{\cH(K)^{N}} \\
     & \quad  = \sum_{i \in \cI} \sum_{j=1}^{N} 
  \langle f_{i}\otimes E^{(N)}_{j}, K_{W',V',y}\rangle \cdot
  \langle K_{W,V,y}, f_{i}\otimes E^{(N)}_{j} \rangle.
  \end{align*}
  By the reproducing property \eqref{reprod5},
  \begin{align*}
      \langle f_{i} \otimes E^{(N)}_{j}, K_{W',V',y'} 
      \rangle_{\cH(K)^{N}} & = 
   \langle f_{i}, K_{W',v'_{j},y'} \rangle_{\cH(K)} \\
   & = \langle f_{i}(W') v^{\prime *}_{j}, y' \rangle_{\cY^{m'}}.
  \end{align*}
  Similarly
  $$
  \langle K_{W,V,y}, f_{i} \otimes E^{(N)}_{j} \rangle_{\cH(K)^{N}} =
  \langle y, f_{i}(W) v_{j}^{*} \rangle_{\cY^{m}}.
 $$
 Hence, making use of the convention \eqref{conv} applied to the 
 vectors $f_{i}(W) v_{j}^{*}$, we have 
 \begin{align*}
    & \langle K_{W,V,y}, K_{W',V',y'} \rangle_{\cH(K)^{N}} =
     \sum_{i \in \cI} \sum_{j=1}^{N} \langle f_{i}(W') v^{\prime 
     *}_{j}, y' \rangle_{\cY^{m'}}  \langle y, f_{i}(W) v_{j}^{*} 
     \rangle_{\cY^{m}} \\
     & = \sum_{i \in \cI} \sum_{j=1}^{N} \langle \left( 
     f_{i}(W')(v_{j}^{\prime *}) \right) \left( f_{i}(W)(v_{j}^{*}) 
     \right)^{*} y, y' \rangle_{\cY^{m'}} \\
     & = \sum_{i \in \cI} \langle \left( f_{i}(W')(V^{\prime *}) 
     \right)  \left( f_{i}(W)(V^{*}) \right)^{*} y, y' 
     \rangle_{\cY^{m'}}
  \end{align*}
  Recalling the formula \eqref{KerInProd} now leads to the expression 
  \eqref{Bergman-ker} for the kernel function $K(W',W)(V^{\prime *} V)$.  The 
  arbitrariness of the vectors $y$ and $y'$  in the preceding 
  analysis leads to the conclusion 
  that the series in \eqref{Bergman-ker} converges in the weak 
  topology.
     \end{proof}
     
  In case $\cH$ is finite-dimensional, one can get explicit formulas 
  for the kernel function from an arbitrary basis (not necessarily 
  orthonormal).
  
  \begin{theorem} \label{T:AlpayDym}
      Let $\cH$ be a finite-dimensional Hilbert space consisting of 
      $\cL(\cA, \cY)$-valued global/nc functions and equipped with a unital 
      $*$-representa\-tion $\sigma$ as in \eqref{rep2}, and let 
      $\{f_{1}, \dots, f_{S}\}$ be a basis (not necessarily 
      orthonormal or orthogonal) for $\cH$.  Let us introduce the 
      gramian matrix
      $$
          G = [ \langle f_{j}, f_{i} \rangle ]_{i,j = 1, \dots, S}
$$
for the basis $\{f_{i}\}_{i=1, \dots, S}$.  Then, with the same 
conventions as used in the statement of Theorem \ref{T:Bergmann}, the kernel function 
$K$ for $\cH$ (existence of which is guaranteed by Theorem 
\ref{T:RKHS}) is given by
\begin{equation}   \label{Ker-FinDim}
    K(W',W)(V^{\prime *} V) = \sum_{i,j=1}^{S} \left( 
    f_{i}(W')(V^{\prime *}) \right) \left( G^{-1} \right)_{ij} \left( 
    f_{j}(W)(V^{*}) \right)^{*}.
\end{equation}
  \end{theorem}
  
  \begin{proof}
     Any $f \in \cH$ has an expansion $f = \sum_{i=1}^{S} \alpha_{i} 
     f_{i}$ in terms of the basis $\{f_{i}\}_{i=1}^{S}$.  We set up 
     a system of equations in order to solve for the coefficients 
     $\{\alpha_{i} \in \bbC \colon i=1, \dots, S\}$:
 $$
     \langle f, f_{i} \rangle_{\cH}  =  \sum_{j=1}^{S} \alpha_{j} 
     \langle f_{j}, f_{i} \rangle_{\cH} 
      = \sum_{j=1}^{S} G_{ij} \alpha_{j}.
 $$
 Solving for the $\alpha_{j}$'s gives
 $$
   \begin{bmatrix} \alpha_{1} \\ \vdots \\ \alpha_{S} \end{bmatrix} =
       G^{-1} \begin{bmatrix} \langle f, f_{1}\rangle \\ \vdots \\ 
       \langle f, f_{S} \rangle \end{bmatrix}.
 $$
 or 
 $$
      \alpha_{i} = \sum_{j=1}^{S} \left( G^{-1} \right)_{ij} \langle 
      f, f_{j} \rangle.
 $$
 Thus $f = \sum_{i=1}^{S} \alpha_{i} f_{i}$ is given by
 $$
   f = \sum_{j=1}^{S} \alpha_{i} f_{i} = \sum_{i,j=1}^{S} \left( 
   G^{-1} \right)_{ij} \langle f, f_{j}\rangle_{\cH} f_{i}.
 $$
 Hence, for $W \in \Omega_{m}$ and $v \in \cA^{1 \times m}$, 
 \begin{align*}
 \langle f(W) v^{*}, y \rangle_{\cY^{m}} & = \sum_{i,j=1}^{S} \left( 
 G^{-1} \right)_{ij} \langle f, f_{j} \rangle_{\cH} \langle f_{i}(W) 
 v^{*}, y \rangle_{\cY^{m}}  \\
 & = \langle f, \sum_{i,j = 1}^{S} \left( G^{-1} \right)_{ji} \langle 
 y, f_{i}(W)(v^{*}) \rangle_{\cY^{m}} f_{j} \rangle_{\cH} \\
 & = \langle f, K_{W,v,y} \rangle_{\cH}
 \end{align*}
 with $K_{W,v,y}$ given by
 $$
 K_{w,v,y} = \sum_{i,j = 1}^{S} \left( G^{-1} \right)_{ji} f_{j} 
 \left( f_{i}(W) (v^{*}) \right)^{*} y
 $$
 where we again make use of the convention \eqref{conv}.
 Then 
 \begin{align*}
   &  \langle K(W',W)(v^{\prime *} v) y, y' \rangle_{{\mathbb C}^{m'}} 
    = \langle K_{W,u,y}, K_{W',v',y'} \rangle_{\cH(K)} \\
     & =\left \langle \sum_{i,j=1}^{S} \left( G^{-1} \right)_{ji} f_{j} 
     \left( f_{i}(W)(v^{*}) \right)^{*}y,
     \sum_{i',j' = 1}^{S} \left( G^{-1} \right)_{j' i'} f_{j'} \left( 
     f_{i'}(W') (v^{\prime *}) \right)^{*} y' \right \rangle  \\
     & = \left\langle \sum_{i,j,i',j' = 1}^{S} \left( G^{-1} \right)_{ji} G_{j'j}
     \left(G^{-1} \right)_{i' j'} \left( (f_{i'}(W') (v^{\prime *}) 
     \right) \left( f_{i}(W) (v^{*}) \right)^{*}y, y' \right\rangle \\
     & = \left\langle \sum_{i',i = 1}^{S} \left( G^{-1} \right)_{i' i}
     \left( f_{i'}(W')(v^{\prime *}) \right)^{*} \left( f_{i}(W)(v') 
     \right)^{*} y, y' \right\rangle
 \end{align*} 
 which after some minor rearrangement agrees with \eqref{Ker-FinDim}.
 \end{proof}

\subsection{Lifted norm spaces}   \label{S:LiftedNorm}

 In this section we present a different way of viewing our global/nc 
 reproducing kernel Hilbert spaces.
 
 The ingredients for the construction are as follows:
 \begin{itemize}
     \item 
 $\Omega$ is a nc subset of $\cS_{\rm nc}$ where $\cS$ is a set, 
 
  \item $\cA$ is a $C^{*}$-algebra,
 
 \item $\cH$ is a Hilbert space equipped with a unital 
 $*$-representation $\sigma_{\cH}$ mapping $\cA$ to $\cL(\cH)$,
 
 \item $\cY$ is a coefficient Hilbert space, 
 
 \item $H \colon \Omega \to \cL(\cH, \cY)_{\rm nc}$ is a global 
 function.  In case $\cS$ is taken to be a vector space $\cV$, we 
 shall also consider the case where $H$ is a nc function.
 \end{itemize}
 
Given these ingredients we define a Hilbert space $\cH_{\ell} = 
\cH_{\ell}(H, \sigma_{\cH})$ (the \textbf{lifted norm space associated 
with $\cH$ and $\sigma_{\cH}$}) by
\begin{equation}  \label{Hl}
\cH_{\ell} = \{ f \colon \Omega \to \cL(\cA, \cY)_{\rm nc} \colon
f = f_{h} \text{ for some } h \in \cH\}
\end{equation}
where the function $f_{h} \colon \Omega \to \cL(\cA, \cY)_{\rm nc}$ is 
specified as follows:  given $Z \in \Omega_{n}$, $u \in \cA^{n}$, 
\begin{equation}   \label{fh}
    f_{h}(Z) u = H(Z) \left({\rm id}_{{\mathbb C}^{n}} \otimes 
    \sigma_{\cH} \right)(u) h
\end{equation}
with $\cH_{\ell}$-norm given by
$$
   \| f \|^{2}_{\cH_{\ell}} = \min \{ \| h \|^{2} \colon h \in \cH 
   \text{ with } f = f_{h} \}.
$$
Then we have the following result.

\begin{theorem}  \label{T:lifted}
    Suppose that the Hilbert space $\cH_{\ell}$ is defined as in 
    \eqref{Hl}.  
    Then $\cH_{\ell}$ is a global reproducing kernel Hilbert space 
    with reproducing kernel $K$ given by
 \begin{equation}   \label{Hsubl-K}
  K(Z,W)(P) = H(Z) ({\rm id}_{{\mathbb C}^{n \times m}} \otimes 
  \sigma_{\cH})(P) H(W)^{*}
 \end{equation}
 for $Z \in \Omega_{n}$, $W \in \Omega_{m}$, $P \in \cA^{n \times m}$.
 If $\cS$ is taken to be a vector space $\cV$ and $H$ is assumed to 
 be a nc function, then $K$ given by \eqref{Hsubl-K} is a cp nc 
 kernel.
 \end{theorem}
 
 \begin{proof} 
     We first verify that any function of the form $f_{h}$ as in 
     \eqref{fh} is a global function from $\Omega$ to $\cL(\cA, 
     \cY)_{\rm nc}$ given that $H$ is a global function from $\Omega$ 
     to $\cL(\cH, \cY)_{\rm nc}$.  Indeed, from \eqref{fh} we read off
     \begin{align*}
& f_{h} \left( \sbm{Z_{1} & 0 \\ 0 & Z_{2}} \right) \begin{bmatrix} u_{1} \\ 
u_{2} \end{bmatrix}    =
H\left( \sbm{Z_{1} & 0 \\ 0 & Z_{2}} \right) \begin{bmatrix} 
({\rm id}_{{\mathbb C}^{n_{1}}} \otimes \sigma_{\cH})(u_{1}) \\ 
({\rm id}_{{\mathbb C}^{n_{2}}} \otimes \sigma_{\cH})(u_{2}) \end{bmatrix} h \\
& \quad = \begin{bmatrix} H(Z_{1}) & 0 \\ 0 & H(Z_{2}) \end{bmatrix} 
\begin{bmatrix} ({\rm id}_{{\mathbb C}^{n_{1}}} \otimes \sigma_{\cH})(u_{1}) h
    \\ ({\rm id}_{{\mathbb C}^{n_{2}}} \otimes \sigma_{\cH})(u_{2}) h \end{bmatrix}   
     = \begin{bmatrix} f_{h}(Z_{1}) & 0 \\ 0 & f_{h}(Z_{2}) 
\end{bmatrix} \begin{bmatrix} u_{1} \\ u_{2} \end{bmatrix}.
\end{align*}

Similarly, if $\cS$ is a vector space $\cV$ and $H$ is a nc function, 
we check that each $f_{h}$ is a nc function as follows.
Suppose that $Z \in \Omega_{n}$, $\widetilde Z \in \Omega_{m}$, $\alpha 
\in {\mathbb C}^{m \times n}$ are such that $\alpha Z = \widetilde Z \alpha$.
Then the computation
\begin{align*}
    \alpha f_{h}(Z)(u) & = \alpha H(Z) ({\rm id}_{{\mathbb C}^{n}} \otimes 
    \sigma_{\cH})(u) h = H(\widetilde Z) \alpha ({\rm id}_{{\mathbb C}^{n}} 
    \otimes \sigma_{\cH})(u) h \\
    & = H(\widetilde Z) \alpha \sbm{ \sigma_{\cH}(u_{1}) \\ \vdots \\ 
    \sigma_{\cH}(u_{n})} h = H(\widetilde Z) \sbm{ \sigma_{\cH}(\sum_{j=1}^{n} 
    \alpha_{1j} u_{j}) \\ \vdots \\ \sigma_{\cH}(\sum_{j=1}^{n} \alpha_{mj} u_{j}) } 
    h \\
    & = H(\widetilde Z) \left( {\rm id}_{{\mathbb C}^{m}} \otimes 
    \sigma_{\cH} \right)(\alpha u) h 
    = f_{h}(\widetilde Z)(\alpha u)
\end{align*}
verifies that $f_{h}$ is a nc function.
 
Let us check next that, for each $Z \in \Omega_{n}$, the map $f 
\mapsto f(Z)$ is bounded from $\cH_{\ell}$ to $\cL(\cA^{n}, \cY^{n})$.
Indeed, if $h \in \cH$ is any choice of vector in $\cH$ such that $f 
= f_{h}$ and if $u \in \cA^{n}$, then
$$
    \| f(Z) u \|_{\cY^{n}}  = \| H(Z) ({\rm id}_{{\mathbb C}^{n}} 
    \otimes \sigma_{\cH})(u) h \| \le \| H(Z)\|_{\cL(\cH^{n}, \cY^{n})} \| u 
    \|_{\cA^{n}} \| h \|_{\cH}.
$$
Minimizing over all $h \in \cH$ for which $f = f_{h}$ then gives
$$
 \| f(Z) u \|_{\cY^{n}} \le \| H(Z)\|_{\cL(\cH^{n}, \cY^{n})} 
 \|u\|_{\cA^{n}} \| f \|_{\cH_{\ell}}
$$
and the boundedness claim follows as wanted.  We shall see below a 
more explicit verification of the boundedness of these point 
evaluations.

For $a \in \cA$ and $f \in \cH$, let us consider the new function 
$\sigma(a) f$ given by the formula \eqref{rep2} in Theorem 
\ref{T:RKHS}: for $Z \in \Omega_{n}$ and $u \in \cA^{n}$ define 
$\sigma(a) f \colon \Omega_{n} \to \cL(\cA^{n}, \cY^{n})$ so that
\begin{equation}   \label{sigma-concrete}
(\sigma(a) f)(Z)(u) = f(Z)(ua).
\end{equation}
Choose $h \in \cH$ so that $f = f_{h}$.  Then we compute
\begin{align*}
    \left( \sigma(a) f_{h} \right) (Z)(u) & = f_{h}(Z)(ua) = H(Z) 
    ({\rm id}_{{\mathbb C}^{n}} \otimes \sigma_{\cH})(ua) h \\
    & = H(Z) ({\rm id}_{{\mathbb C}^{n}} \otimes \sigma_{\cH})(u) 
    \sigma_{\cH}(a) h = f_{\sigma_{\cH}(a) h}(Z)(u)
\end{align*}
i.e., we have verified
\begin{equation} \label{sigmaH}
    \sigma(a) f_{h} = f_{\sigma_{\cH}(a) h}.
\end{equation}
In particular $\cH_{\ell}$ is invariant under the action of $\cA$ 
defined by $\sigma$. 
By Remark \ref{R:RKHS} we see that $\sigma \colon \cA \to 
\cL(\cH_{\ell})$ is a unital 
representation of $\cA$ on $\cH_{\ell}$.  We shall see below (after a 
little more work) that in fact $\sigma$ is a $*$-representation of 
$\cA$.  One could then apply Theorem \ref{T:RKHS} to see 
that there is a global/nc kernel $K$ so that $\cH_{\ell} = \cH(K)$.  
Rather than applying the existence result from Theorem \ref{T:RKHS}, 
we shall show directly that the kernel $K$ given by  
\eqref{Hsubl-K} satisfies all the properties required to be the 
reproducing kernel for the space $\cH_{\ell}$.

By definition the map $h \mapsto f_{h}$ is a coisometry of 
$\cH$ onto $\cH_{\ell}$.  Moreover, given $f \in \cH_{\ell}$, 
there is always a unique $h_{0} 
\in \cH$ so that $f = f_{h_{0}}$ and we have the equality of norms:  
$\| f \|_{\cH_{\ell}} = \| h_{0} \|_{\cH}$: simply take $h_{0} = 
P_{\cN^{\perp}} h_{1}$ where $h_{1}$ is any choice of vector in $\cH$ 
with $f = f_{h_{1}}$ and where $\cN = \{ h \in \cH \colon f_{h} = 0 
\}$.  The space $\cN$ can be characterized explicitly as 
\begin{equation}  \label{KernelSpace}
  \cN = \left({\rm span} \{ ({\rm id}_{{\mathbb C}^{1 \times m}} 
  \otimes\sigma_{\cH})(v) {\rm Ran} \, H(W)^{*} \colon W \in \Omega_{m}, 
  v \in \cA^{1 \times m},\, m \in {\mathbb N} \}\right)^{\perp}.
\end{equation}

Given $W \in \Omega_{m}$, $v \in \cA^{1 \times m}$, and $y \in 
\cY^{m}$, we have
\begin{align*}
    \langle f_{h}(W)(v^{*}), y \rangle_{\cY^{m}} & =
    \langle H(W) ({\rm id}_{{\mathbb C}^{m}} \otimes 
    \sigma_{\cH})(v^{*}) h, y \rangle_{\cY^{m}} \\
    & = \langle h, \, ({\rm id}_{{\mathbb C}^{1 \times m}} \otimes 
    \sigma_{\cH})(v) H(W)^{*} y \rangle_{\cH}.
\end{align*}
From the characterization \eqref{KernelSpace} of the space $\cN$, one can see 
that the vector $({\rm id}_{{\mathbb C}^{1 \times m}} \otimes 
    \sigma_{\cH})(v) H(W)^{*} y $ is in the initial space of the 
    coisometry $h \mapsto f_{h}$.  Alternatively we can always choose 
    $h$ in the initial space of $h \mapsto f_{h}$ so that we still 
    have $f  = f_{h}$.  In any case we may apply the map $h \mapsto 
    f_{h}$ to each vector in the inner product in the last 
    expression above and preserve the value of the inner product. 
 Thus we get
 \begin{align}
     \langle f_{h}(W)(v^{*}), y \rangle_{\cY^{m}} & = 
  \langle f_{h}, f_{({\rm id}_{{\mathbb C}^{1 \times m}} \otimes 
  \sigma_{\cH})(v) H(W)^{*}y} \rangle_{\cH_{\ell}} \notag \\
  & = \langle f_{h}, K_{W,v,y} \rangle_{\cH_{\ell}}
  \label{reprod4}
 \end{align}
 where we have set
 \begin{equation}   \label{KerEl}
    K_{W,v,y} = f_{({\rm id}_{{\mathbb C}^{1 \times m}} \otimes 
  \sigma_{\cH})(v) H(W)^{*}y}.
 \end{equation}
 As $K_{W,v,y}$ is clearly an element of $\cH_{\ell}$, the formula 
 \eqref{reprod4} exhibits the fact that the point evaluation $f_{h} 
 \mapsto f_{h}(W)$ is bounded as an operator from $\cH_{\ell}$ to 
 $\cL(\cA^{m}, \cY^{m})$, i.e., we have arrived at the promised 
 second proof of this fact. 
 Using the rule \eqref{fh}, we can calculate
 \begin{align}
     K_{W,v,y}(Z) u & = H(Z) \left( {\rm id}_{{\mathbb 
     C}^{n}} \otimes \sigma_{\cH}\right)(u) \left( {\rm id}_{{\mathbb 
     C}^{1 \times m}} \otimes \sigma_{\cH}\right)(v) H(W)^{*}y \notag \\
     & = H(Z) \left( {\rm id}_{{\mathbb C}^{n \times m}} \otimes 
     \sigma_{\cH}\right)(uv) H(W)^{*} y.
 \label{eval}
 \end{align}
 If $K_{W',v',y'}$ ($W' \in \Omega_{m'}$, $v' \in \cA^{1 \times m'}$, 
 $y' \in \cY^{m'}$) is a second kernel element as in \eqref{KerEl} 
 and we apply the formula \eqref{reprod4} with $K_{W',v',y'}$ in 
 place of $f_{h}$, as a consequence of the evaluation formula 
 \eqref{eval} we see that
 \begin{align}
     \langle K_{W,v,y},\, K_{W',v',y'} \rangle_{\cH_{\ell}} & =
     \langle K_{W,v,y}(W') v^{\prime *}, y' \rangle_{\cY^{m'}} \notag \\
     & = \langle H(W') \left( {\rm id}_{{\mathbb C}^{m' \times m}} 
     \otimes \sigma_{\cH} \right) (v^{\prime *} v) H(W)^{*} y, y' 
     \rangle_{\cY^{m'}} \notag \\
     & = \langle K(W',W)(v^{\prime *} v) y, y' \rangle_{\cY^{m'}}
     \label{KerElInnerProd}
 \end{align}
 where we have 
 \begin{equation}  \label{KolDecom'}
   K(Z,W)(P) = H(Z) \left( {\rm id}_{{\mathbb C}^{n \times m}} 
   \otimes \sigma_{\cH} \right)(P) H(W)^{*}
 \end{equation}
 for in general $Z \in \Omega_{n}$, $W \in \Omega_{m}$ and $P \in 
 \cA^{n \times m}$.  The definition \eqref{KolDecom'} of $K$ exhibits a Kolmogorov 
 decomposition for $K$.  Hence, as a consequence of (3) $\Rightarrow$ 
 (1) in Theorem \ref{T:cpker}, we see that $K$ is completely positive.
 Furthermore, $H$ being a global/nc function implies that $K$ is a 
 global/nc kernel, as was verified as part of the proof of (3) 
 $\Rightarrow$ (1) in Theorem \ref{T:cpker}.
 
 From the reproducing property \eqref{reprod4}, one can read off as 
 in the proof of Theorem \ref{T:cpker} or \ref{T:RKHS} the action of 
 the representation $\sigma(a)^{*}$ on a kernel element $K_{W,v,y}$:
 $$
   \sigma(a)^{*} \colon K_{W,v,y} \mapsto K_{W,a^{*}v,y}.
 $$
 Then one can use the fact that the inner product in \eqref{KerElInnerProd} as a 
 function of the pair $(v,v')$ depends only on the product $v^{\prime *} 
 v$ to see that $\sigma(a)^{*} = \sigma(a^{*})$, i.e., indeed 
 $\sigma \colon \cA \to \cL(\cH_{\ell})$ is a unital 
 $*$-representation.
 
 By uniqueness in the Riesz-Frechet theorem, it follows that the space 
 $\cH_{\ell}$ is isometrically identical to the global/nc reproducing 
 kernel spaces $\cH(K)$ with kernel $K$ given by \eqref{KolDecom'}.
\end{proof}

\begin{remark}  \label{R:3to2}
One can also view Theorem \ref{T:lifted} as an alternative direct 
path to the proof of (3) $\Rightarrow$ (2) in Theorem \ref{T:cpker}; 
one constructs the reproducing kernel Hilbert space $\cH(K)$ directly 
from the function $H$ and a unital $*$-representation $\sigma_{\cH}$  
in the Kolmogorov decomposition for $K$ rather than from $K$ itself 
as in (1) $\Rightarrow$ (2).

In the case where $\cH = \cH(K)$ and the function $H \colon \Omega 
\to \cL(\cH(K), \cY)$ is given concretely by \eqref{defH} with unital 
$*$-representation $\sigma_{\cH} = \sigma$ given by \eqref{rep1}, an
instructive exercise is to verify directly that indeed we recover 
$\cH(K)$ as $\cH(K) = \cH_{\ell}(H, \sigma)$.

Let $Z \in \Omega_{n}$, $\cH = \cH(K)$.  For $h \in \cH(K)$, $h(Z) 
\in \cL(\cA^{n}, \cY^{n})$.  We use \eqref{defH} to define $H(Z)$:
$$
  H(Z) \sbm{ h_{1} \\ \vdots \\ h_{n}} = h_{1}(Z) (E^{(n)}_{1} 
  \otimes 1_{\cA}) + \cdots + h_{n}(Z)(E^{(n)}_{n} \otimes 1_{\cA}).
$$
For $u = \sbm{ u_{1} \\ \vdots \\ u_{n}} \in \cA^{n}$ and $h \in \cH 
= \cH(K)$, we use \eqref{fh} to define $f_{h}$:
\begin{align*}
    f_{h}(Z)(u) & = H(Z) \sbm{  \sigma(u_{1})(h) \\ \vdots \\ 
    \sigma(u_{n})(h) } \\
    & = (\sigma(u_{1})h)(Z)(E^{(n)}_{1} \otimes 1_{\cA}) + \cdots + 
    (\sigma(u_{n})h)(E^{(n)}_{n} \otimes 1_{\cA}) \\
    & = h(Z) \sbm{ u_{1} \\ 0 \\ \vdots \\  0} + \cdots + h(Z) \sbm{ 0 \\ 
    \vdots \\ 0 \\ u_{n}} \\
    & = h(Z) \sbm{u_{1} \\ \vdots \\ u_{n}}.
\end{align*}
i.e., the identification map $h \mapsto f_{h}$ between $\cH = \cH(K)$ 
and $\cH_{\ell}$ in this case is the identity:
$$
    f_{h}(Z) = h(Z) \text{ for all } Z \in \Omega.
$$
Furthermore, the relation \eqref{sigmaH} identifies the 
representation $\sigma = \sigma_{\cH(K)}$ already specified on $\cH = 
\cH(K)$ with the representation $\sigma$ on $\cH_{\ell}$ specified by
\eqref{sigma-concrete}. Furthermore, in the proof of (2) 
$\Rightarrow$ (3) in Theorem \ref{T:cpker}, we identified the 
Kolmogorov decomposition \eqref{Koldecom} with $H$ being given by 
\eqref{defH}.  On the other hand, we have seen that this same 
expression gives the reproducing kernel for the space $\cH_{\ell}$ 
(see \eqref{KolDecom'}).  We conclude that $\cH(K) = \cH_{\ell}(H, 
\sigma)$ identically and isometrically. 
 \end{remark}
 
\subsection{Special cases and examples}  \label{S:special}

\subsubsection{The special case $\cS = \{s_{0}\}$} 
Let us consider the special case where
$\cS = \{ s_{0}\}$, $\Omega_{n} = \{ s_{0} \otimes 
I_{n}\}$, $\cA$ = a $C^{*}$-algebra.  Let $K \colon \Omega 
\times \Omega \to \cL(\cA, \cL(\cY))_{\rm nc}$ be a cp global 
kernel.  Then there is only one choice of $Z \in \Omega_{n}$, namely 
$Z = s_{0} \otimes I_{n}$; let $\varphi^{(n,m)} = K(s_{0} \otimes 
I_{n}, s_{0} \otimes I_{m}) \in \cL(\cA^{n \times m}, \cL(\cY)^{n 
\times m})$. The content of the ``respects direct sums'' property 
\eqref{kerds} for the kernel $K$ is that
$$
  K(s_{0} \otimes I_{n}, s_{0} \otimes I_{m}) = {\rm id}_{{\mathbb 
  C}^{n \times m}} \otimes K(s_{0}, s_{0}),
$$
i.e., 
$$
  K(s_{0} \otimes I_{n}, s_{0} \otimes I_{m}) = \varphi^{(n,m)}
$$
where we set $\varphi = K(s_{0}, s_{0})$ and define $\psi^{(n,m)} =
{\rm id}_{{\mathbb C}^{n\times m}} \otimes \psi$ as in \eqref{cp} and 
\eqref{cb}:
$$
  \varphi^{(n,m)} = {\rm id_{{\mathbb C}^{n \times m}} \otimes \varphi} 
  \colon [P_{ij}] \mapsto [ \varphi(P_{ij})]. 
$$
Thus, for this case where $\cS = \{s_{0}\}$ is a singleton set, $K$ 
being a completely positive kernel is the same as $K(s_{0}, s_{0}) =
\varphi \colon \cA \to \cL(\cY)$ being a cp map and conversely,
$\varphi \colon \cA \to \cL(\cY)$ being a cp map is equivalent to 
$\varphi$ having a unique extension to a cp global kernel $K$ on the nc 
envelope $[\{s_{0}\}]_{\rm nc}$ of the singleton set $\{s_{0}\}$, 
where we set $K(s_{0}, s_{0}) = \varphi$.
Moreover the Kolmogorov decomposition 
\eqref{Koldecom} for this case, when restricted 
to level 1, becomes just a formulation of  
the Stinespring dilation theorem \cite{Stinespring} 
(see also \cite[Theorem 3.1]{Paulsen}):  {\em given a 
completely positive map $\varphi \colon \cA \to \cL(\cY)$, there is a 
Hilbert space $\cX$ equipped with a unital 
$*$-representation $\sigma \colon \cA \to \cL(\cX)$  and an operator 
$H(s_{0}) \in \cL(\cX, \cY)$ so that $\varphi(a) = H(s_{0}) 
\sigma(a) H(s_{0})^{*}$}.

Our  general results concerning cp global 
kernels gives the following finer structure for the Stinespring 
dilation space.  We note that the following result is simply the 
specialization of Theorem \ref{T:cpker} to the case where 
$\Omega$ is the  nc envelope of the singleton set $\Omega_{1} 
= \{s_{0}\}$.

\begin{theorem}  \label{T:Stinespring}
    Let $\cA$ be a $C^{*}$-algebra and let $\cY$ be a Hilbert space.
    Suppose that $\varphi \colon \cA \to \cL(\cY)$ is a given linear 
    map.  Then the following are equivalent.
    \begin{enumerate}
	\item $\varphi$ is completely positive, i.e., ${\rm 
	id}_{{\mathbb C}^{n \times n}} \otimes \varphi \colon \cA^{n 
	\times n} \to \cL(\cY)^{n \times n}$ maps positive elements 
	to positive elements for each $n \in {\mathbb N}$.
	
\item There is a Hilbert space $\cH(\varphi)$ whose elements $f$ are in the space 
$\cL(\cA, \cY)$ of bounded linear operators from $\cA$ to $\cY$ such 
that:
\begin{enumerate}
    \item for $v \in \cA$ and $y \in \cY$, the 
    kernel element $K_{v,y}$, identified as an element of $\cL(\cA, 
    \cY)$ via the formula
    $$
      K_{v,y}(u) = \varphi(uv) y
   $$
   for $u \in \cA$, belongs to $\cH(\varphi)$.
   
   \item The kernel elements $K_{v,y}$ have the reproducing property:
   for $f \in \cH(\varphi)$, 
   $v  \in \cA$ and  $y \in  \cY$,
   $$
  \langle  f(v^{*}), y \rangle_{\cY} = \langle f, K_{v,y} 
   \rangle_{\cH(\varphi)}.
   $$
   
   \item $\cH(\varphi)$ is equipped with a unital $*$-representation 
   $\sigma$ mapping $\cA$ to $\cL(\cH(\varphi))$ such that
   $$
    (\sigma(a) f)(u) = f(ua)
    $$
    for $a \in \cA$, $u \in \cA$, with action on kernel 
    elements $K_{v,y}$ ($v \in \cA$, $y \in \cY$) given by
    $$
     \sigma(a) \colon K_{v,y} \mapsto K_{av,y}.
    $$
 \end{enumerate}
 
 \item $\varphi$ has a Stinespring dilation: there is a Hilbert space 
 $\cX$ equipped with a unital $*$-representation $\sigma_{\cX} \colon 
 \cA \to \cL(\cX)$ and a bounded linear operator $H(s_{0}) \in 
 \cL(\cX, \cY)$ so that $\varphi(a) = H(s_{0}) \sigma_{\cX}(a) 
 H(s_{0})^{*}$.
\end{enumerate}
\end{theorem}

\begin{remark} \label{R:Stinespring}
     As a consequence of Remark \ref{R:modelKol} specialized to 
    the setting here, we see that we may take the Stinespring dilation 
    space $\cX$ in part (3) of Theorem \ref{T:Stinespring} to be the 
    reproducing kernel space $\cH(\varphi)$ with associated map 
$H(s_{0}) \in \cL(\cH(\varphi), \cY)$ given by
$$
  H(s_{0}) \colon f \mapsto f(1_{\cA})
$$
with adjoint $H(s_{0})^{*}$ given by
$$
  H(s_{0})^{*} \colon y \mapsto K_{1_{\cA}, y}.
$$
This gives a new geometric picture for the Stinespring dilation space 
$\cX$; we refer to the paper of Muhly-Solel \cite{MS2002} for an alternative 
geometric picture.

Also already noted in the original construction of Stinespring \cite{Stinespring}, 
the space $\cX = \cH(\varphi)$ can itself be considered as 
an Aronszajn reproducing kernel Hilbert space with the algebra $\cA$ 
taken to be the point set, with kernel $ K_{\varphi} \colon \cA 
\times \cA \to \cL(\cY)$ given by
$$
   K_{\varphi} (a,b) = \varphi(b^{*}a).
$$

At first impression it appears that part (2) of Theorem 
\ref{T:Stinespring} is an oversimplification of part (2) of Theorem 
\ref{T:cpker} since part (2) of Theorem \ref{T:Stinespring} mentions 
only kernel elements of the form $K_{v,y}$ with $v \in \cA$ and $y 
\in \cY$ but not with $v \in \cA^{1 \times m}$ and $y \in \cY^{m}$.  
The explanation is that the additional kernel elements $K_{V, Y}$ (say
with $V = \begin{bmatrix} v_{1} & \cdots & v_{m}\end{bmatrix} \in 
\cA^{1 \times m}$ and $y = \sbm{ y_{1} \\ \vdots \\ y_{m}} \in 
\cY^{m}$) are present, but do not add any information since in this 
case we have the linearity relations
$$
 K_{V, Y} = \sum_{i=1}^{m} K_{v_{i}, y_{i}}.
$$
This is a consequence of the fact that the level-$m$ point-set 
$\Omega_{m}$ consists only of the single diagonal point $s_{0} \otimes I_{m}$.
Note first that $K_{V,Y}$ is really $K_{s_{0} \otimes I_{m}, V, Y}$ 
in the notation of Theorem \ref{T:cpker}.
Identifying $f$ with $f(s_{0})$ when convenient (the meaning should 
be clear from the context) and using that 
$f(s_{0} \otimes I_{m}) = f(s_{0}) \otimes I_{m}$ (since $f \in 
\cH(\varphi)$ respects direct sums), 
we see from the general reproducing property \eqref{reprod} that
\begin{align*}
    \langle f, K_{V,Y} \rangle_{\cH(\varphi)} 
    & =  \langle f(s_{0} \otimes I_{m})(V^{*}), y \rangle_{\cY^{m}} \\
    & = \left\langle \sbm{ f(s_{0}) & & \\ & \ddots & \\ & & f(s_{0}) }
    \sbm{ v_{1}^{*} \\ \vdots \\ v_{m}^{*}}, \, \sbm{ y_{1} \\ \vdots 
    \\ y_{m}} \right\rangle_{\cY^{m}} \\
    & = \sum_{i=1}^{m} \langle f, K_{v_{i}, y_{i}} \rangle_{\cH(K)} 
    \\  & = \left\langle f, \sum_{i=1}^{m} K_{v_{i}, y_{i}} \right\rangle_{\cH(K)}
\end{align*}
whence we conclude that $K_{V,Y} = \sum_{i=1}^{m} K_{v_{i}, y_{i}}$.
\end{remark}

A notion closely related to complete positivity is that of complete 
boundedness which can be formulated more generally for maps between 
operator spaces $\cV_{1}$ and $\cV_{0}$ (recall the definitions from 
Subsection \ref{S:ncsets}).  Given operator spaces 
$\cV_{1}$ and $\cV_{0}$ and a linear map $\varphi \colon \cV_{1} \to 
\cV_{0}$, we say that $\varphi$ is \textbf{completely bounded} (cb) if 
there is a constant $M < \infty$ so that $\| \varphi^{(n)}\| \le M$ 
for all $n \in {\mathbb N}$; the smallest such $M$ is defined to be 
the \textbf{completely bounded norm} of $\varphi$, denoted as 
$\|\varphi\|_{\rm cb}$.  As yet another piece of useful terminology, 
let us say that any Hilbert space $\cH$ whose elements consist of 
global functions $f \colon \Omega \to \cL(\cA, \cY)_{\rm nc}$ as in 
the hypotheses of Theorem \ref{T:RKHS} is a \textbf{nc functional 
Hilbert space equipped with the $*$-representation} $\sigma_{\cH} \colon \cA 
\to \cL(\cH)$ given by formula \eqref{rep2}.  In the setup of Theorem 
\ref{T:RKHS}, it was only assumed that the point evaluation maps 
$\boldsymbol{\rm ev}_{Z} \colon f \mapsto f(Z)$ were bounded.  In fact it 
turns out that each $\boldsymbol{\rm ev}_{Z} \colon \cH \to \cY$  as well as the 
factor $H(Z) \colon \cH \to \cY^{n}$ (for each $Z \in \Omega_{n}$) are 
completely bounded, as summed up in the following result.   
Here the Hilbert spaces $\cH$ and $\cY$ are given their 
natural column-space operator structure, i.e., we identify $\cH$ with 
$\cL({\mathbb C}, \cH)$ and $\cY$ with $\cL({\mathbb C}, \cY)$ in the 
natural way:
$$ h \in \cH \cong {\mathbf h} \colon c \in {\mathbb C} \mapsto c h 
\in \cH.
$$

\begin{proposition}  \label{P:cbRKHS}
    \begin{enumerate}
	\item Suppose that $K$ is a cp global kernel and that $Z \in 
	\Omega_{n}$ and $W \in \Omega_{m}$.  Then $K(Z,W)$ is cb 
	with 
	\begin{align}  
	\|K(Z,W)\|_{\rm cb} & \le \| K\left( \sbm{ Z & 0 \\ 0 & 
	W}, \sbm{Z & 0 \\ 0 & W} \right)(1_{\cA^{(n+m) \times 
	(n+m)}})\|   \notag \\
	& = \max \{ \|K(Z,Z)(1_{\cA^{n \times n}})\|,  \,
	\|K(W,W)(1_{\cA^{m \times m}}) \|\}.
	 \label{Kercb}
	\end{align}
Moreover, 
    \begin{equation}   \label{Kercb'}
	\| K \left( \sbm{ Z & 0 \\ 0 & Z}, \sbm{ Z & 0 \\0 & Z} \right) \|_{\rm cb} = 
	\| K(Z,Z)(1_{\cA^{n \times n}}) \|.
\end{equation}
	
\item Suppose that $\cH = \cH(K)$ is a nc functional Hilbert space equipped with 
    canonical $*$-representation $\sigma_{\cH}$ as defined above and 
    let $W \in \Omega_{m}$.
   Then $f(W) \in \cL(\cA^{m}, \cY^{m})$ is completely bounded with cb-norm 
  satisfying
  \begin{equation}  \label{cbnormf(W)}
   \| f(W) \|_{\cL_{\rm cb}(\cA^{m}, \cY^{m})} \le \| f \|_{\cH(K)} \, 
   \| K(W,W)(1_{\cA^{m \times m}}) \|^{1/2}.
   \end{equation}
 Moreover, if  $\boldsymbol{\rm ev}^{(r,s)}_{W} \colon \cH(K) \to 
 \cL((\cA^{m})^{r \times s}, (\cY^{m})^{r \times s})$ is the 
 point-evaluation operator given by
 $$
 \boldsymbol{\rm ev}^{(r,s)}_{W} \colon f \mapsto {\rm id}_{{\mathbb 
 C}^{r \times s}} \otimes f(W), 
$$
 then
  \begin{equation}   \label{normevW}
   \| \boldsymbol{\rm ev}_{W}^{(r,s)} \|_{\cL(\cH, \cL((\cA^{m})^{r 
   \times s}, (\cY^{m})^{r \times s}))} = 
\sup_{ V, \bc} \| K(W,W)( V \bc \bc^{*} V^{*}) \|^{1/2}
\end{equation}
where the supremum is taken over $V \in (\cA^{m})^{r \times s}$ of 
norm at most 1 and over $\bc \in {\mathbb C}^{s}$ of norm at most 1.  
In particular, we have the estimate
\begin{equation}  \label{est-normevW}
    \| \boldsymbol{\rm ev}_{W}^{(r,s)} \|_{\cL(\cH, \cL((\cA^{m})^{r 
   \times s}, (\cY^{m})^{r \times s}))} \le \|K(W,W)(1_{\cA^{m \times 
   m}}) \|^{1/2}
\end{equation}
and, for the case $m=1$, we have the equalities
\begin{equation} \label{m=1}
    \| \boldsymbol{ev}_{W}\|_{\cL(\cH, \cL(\cA, \cY))} 
    =  \| \boldsymbol{ev}_{W}\|_{\cL_{\rm cb}(\cH, \cL(\cA, \cY))} 
    = \| K(W,W)(1_{\cA})\|^{1/2}
\end{equation}
where we have set $\boldsymbol{\rm ev}_{W} = \boldsymbol{\rm 
ev}_{W}^{(1,1)}$.

    \item Suppose that $K \colon \Omega \times \Omega \to \cL(\cA, 
    \cL(\cY))$ is a cp global kernel with nc function $H$ in its Kolmogorov decomposition 
    \eqref{Koldecom} and suppose that $Z$ be a point in $\Omega_{n}$. Then the 
    map $H(Z) \colon \cX^{n} \to \cY^{n}$ is completely bounded with 
    cb norm given by
    \begin{equation}   \label{normH(Z)}
	\|H(Z)\|_{\rm cb} = \|H(Z)\| = \| K(Z,Z)(1_{\cA^{n \times 
	n}})\|^{1/2}.
\end{equation}
 \end{enumerate}
 \end{proposition}
 
 \begin{proof}
     To prove (1), suppose that $K$ is a cp global kernel and we 
     fix $Z \in \Omega_{n}$ and $W \in \Omega_{m}$.
     A general fact is that a cp map 
     $\varphi \colon \cA \to \cL(\cY)$ is automatically cb with $\| \varphi \|_{\rm cb} = \| 
     \varphi(1_{\cA})\|$.  In particular, we conclude that, for 
     each fixed $Z \in \Omega_{n}$ and $W \in \Omega_{m}$, 
     $K\left(\sbm{Z & 0 \\ 0 & W}, \sbm{Z & 0 \\ 0 & W}\right) \colon 
     \cA^{(n+m) \times (n+m)} 
     \to \cL(\cY)^{(n+m) \times (n+m)}$ is cb with $\| K(\sbm{Z& 0 \\ 
     0 & W}, \sbm{Z & 0 \\ 0 & W}) \|_{\rm cb} = 
     \| K\left(\sbm{Z & 0 \\ 0 & W}, \sbm{Z & 0 \\ 0 & 
     W}\right)(I_{n+m} \otimes 1_{\cA})\|$.  But a consequence of the 
     ``respects direct sums'' property is that
     $$
       K\left( \sbm{ Z & 0 \\ 0 & W}, \sbm{Z & 0 \\ 0 & W} 
       \right)\left( \sbm{ 0 & P \\ 0 & 0} 
       \right) =
       \begin{bmatrix} 0 & K(Z,W)(P) \\ 
	  0 & 0 \end{bmatrix}.
 $$
 from which we read off that
 $$
 \|{\rm id}_{{\mathbb C}^{n \times n}} \otimes K(Z,W) \| \le
 \left\| {\rm id}_{{\mathbb C}^{n \times n}} \otimes K\left( \sbm{ Z 
 & 0 \\ 0 & W}, \sbm{ Z & 0 \\ 0 & W} \right) \right\|
 $$
 and hence
 $$
   \| K(Z,W) \|_{\rm cb} \le \left\| K\left( \sbm{Z & 0 \\ 0 & W}, 
   \sbm{Z & 0 \\ 0 & W} \right) \right\|_{\rm cb} = \left\| K\left( \sbm{ Z& 0  \\ 0 
   & W}, \sbm{ Z & 0 \\ 0 & W}\right)( 1_{\cA^{N \times N}}) 
   \right\|
 $$
 ($N = n + m $) and it follows that $K(Z,W)$ is cb with cb-bound as in \eqref{Kercb}.
 The equality in \eqref{Kercb} follows from the identity
 $$
    K\left( \sbm{ Z & 0 \\ 0 & W}, \sbm{ Z & 0 \\ 0 & W}\right) 
   (1_{\cA^{N \times N}})  = \begin{bmatrix} K(Z,Z)(1_{\cA^{n \times 
   n}}) & 0 \\ 0 &  K(W,W)(1_{\cA^{m \times m}}) 
   \end{bmatrix},
 $$
 a consequence of the ``respects direct sums'' property \eqref{kerds}.
 Similarly, the cb-norm of $K\left( \sbm{ Z & 0 \\ 0 & Z},  \sbm{ Z & 0 \\ 0 & Z} \right)$ is 
 $\|K\left( \sbm{ Z & 0 \\ 0 & Z}\right)(1_{\cA^{2n \times 2n}})\|$. 
Then the formula \eqref{Kercb'} follows from the  identity
 $$
 K\left( \sbm{ Z & 0 \\ 0 & Z},\sbm{ Z & 0 \\ 0 & Z} \right)(1_{\cA^{2n \times 2n}}) = 
 \begin{bmatrix} K(Z,Z)(1_{\cA^{n}}) & 0 \\ 0 & K(Z,Z)(1_{\cA^{n}}) 
     \end{bmatrix},
 $$

 To prove (2), we assume that $\cH = \cH(K)$ for the cp global kernel 
 $K$ and $f \in \cH$.  We wish to estimate the norm of 
 $$
 {\rm id}_{{\mathbb C}^{r \times s}} \otimes f(W) \colon [ v_{ij} ] 
 \mapsto [f(W)(v_{ij})] \in (\cY^{m})^{r \times s} \cong \cL({\mathbb 
 C}^{s}, (\cY^{m})^{r})
 $$
 where $V = [v_{ij}] \in (\cA^{m})^{r \times s}$.  Let us use the 
 notation $\bv_{i} = \begin{bmatrix} v_{i1} & \cdots & v_{is} 
\end{bmatrix}$ for the $i$-th row of $V$.  Choose $\bc = \sbm{ c_{1} 
\\ \vdots \\ c_{s}} \in {\mathbb C}^{s}$.  Viewing $({\rm 
id}_{{\mathbb C}^{r \times s}} \otimes f(W))(V)$ as an 
operator from ${\mathbb C}^{s}$ to $(\cY^{m})^{r}$, we have
\begin{align*}
& ({\rm id}_{{\mathbb C}^{r \times s}} \otimes f(W))(V) \colon
  \bc \mapsto {\rm col}_{1 \le i \le r} \left[ \sum_{j=1}^{s} f(W) (v_{ij}) 
  c_{j} \right]  \\
  & \quad = {\rm col}_{1 \le i \le r} \left[ \sum_{j=1}^{s} f(W) (v_{ij} 
  c_{j}) \right] = {\rm col}_{1 \le i \le r} \left[  f(W) (\bv_{i} 
  \bc)\right].
\end{align*}
If $y = \sbm{ y_{1} \\ \vdots \\ y_{r}}$ is an arbitrary element of 
$(\cY^{m})^{r}$, we then have
\begin{align}
&  \langle ({\rm id}_{{\mathbb C}^{r \times s}} \otimes f(W))(V) \bc, y 
  \rangle_{(\cY^{m})^{r}} = \sum_{i=1}^{r} \langle f(W)(\bv_{i} \bc), 
  y_{i} \rangle_{\cY^{m}} \notag \\
  & \quad = \sum_{i=1}^{r}\langle f, K_{W, \bc^{*} \bv_{i}^{*}, y_{i}} 
  \rangle_{\cH(K)} \notag \\
  & \quad = \left\langle f, \sum_{i=1}^{r} K_{W, \bc^{*} \bv_{i}^{*}, 
  y_{i}} \right\rangle_{\cH(K)}.
  \label{reprod'}
\end{align}
This leads to the estimate
\begin{equation}   \label{est1}
   \left|  \langle ({\rm id}_{{\mathbb C}^{r \times s}} \otimes 
   f(W))(V) \bc, y 
  \rangle_{(\cY^{m})^{r}} \right| \le
 \| f \|_{\cH(K)} \, \left\| \sum_{i=1}^{r} K_{W, \bc^{*} 
 \bv_{i}^{*},y_{i}} \right\|_{\cH(K)}.
\end{equation}
We next estimate
\begin{align}
  &  \left\| \sum_{i=1}^{r} K_{W, \bc^{*}\bv_{i}^{*},y_{i}} 
    \right\|^{2}_{\cH(K)}  = 
 \sum_{i,j=1}^{r}\left \langle K(W,W) \left(\bv_{i} \bc \bc^{*} 
 \bv_{j}^{*}\right)  y_{j}, y_{i} \right\rangle_{\cY^{m}}  \notag \\
 & \quad = \left\langle K\left( \bigoplus_{1}^{r} W, 
 \bigoplus_{1}^{r} W \right) (V \bc \bc^{*} V^{*}) y, \, y 
 \right\rangle_{(\cY^{m})^{r}} \notag \\
 & \quad \le \left\| K\left( \bigoplus_{1}^{r}W, \bigoplus_{1}^{r}W 
 \right) \right\|_{\cL(\cA^{mr}, \cL(\cY^{mr}))} \| V \bc \bc^{*} V^{*} 
 \|_{\cA^{mr \times mr}}\,  \| y \|^{2}_{(\cY^{m})^{r}}
 \label{est2}
\end{align}
By \eqref{Kercb'} and an easy induction, 
\begin{equation}   \label{est3}
  \left\| K\left(\bigoplus_{1}^{r} W, \bigoplus_{1}^{r} W \right) 
  \right\|_{\cL(\cA^{mr \times mr}, \cL(\cY^{mr}))} = \left\| 
  K(W,W)(1_{\cA^{m \times m}}) \right\|_{\cL(\cY^{m})}.
\end{equation}
Note next that $V \bc \bc^{*} V^{*} \preceq \| \bc \|^{2}_{{\mathbb 
C}^{r}} V V^{*}$ in    $\cA^{mr \times mr}$ and hence
\begin{equation}   \label{est4}
    \| V \bc \bc^{*} V^{*} \|_{\cA^{mr \times mr}} \le \| \bc 
    \|^{2}_{{\mathbb C}^{r}} \| V \|^{2}_{\cA^{mr \times s}}.
\end{equation}
Putting all the pieces \eqref{est1}, \eqref{est2}, \eqref{est2}, 
\eqref{est3} together, we arrive at
\begin{equation}   \label{est5}
    \| ({\rm id}_{{\mathbb C}^{r \times s}} \otimes f(W))(V) \bc \|
    \le \|f \|_{\cH(K)} \|K(W,W)(1_{\cA^{m \times m}}) \|^{1/2} \| V 
    \| \| \bc\|
\end{equation}
from which the estimate \eqref{cbnormf(W)} follows.

We use the generalized reproducing property \eqref{reprod'} to compute the norm of 
$\boldsymbol{\rm ev}_{W}^{(r,s)}$ as follows:
\begin{align*}
    \|\boldsymbol{\rm ev}_{W}^{(r,s)} \| & = 
    \sup_{\|f \|_{\cH} \le 1} \| \boldsymbol{\rm ev}_{W}^{(r,s)}(f) 
    \| \\
    & = \sup_{\|f \| \le  1,\, \| V \| \le 1,\, \| \bc \| \le 1}
    \|({\rm id}_{{\mathbb C}^{r \times s}} \otimes f(W))(V) \bc \|  \\
    & = \sup_{\| f \| \le 1, \, \| V \| \le 1, \| \bc \| \le 1, \| y 
    \| \le 1} | \langle ({\rm id}_{{\mathbb C}^{r \times s}} \otimes 
    f(W))(V) \bc, y \rangle | \\
    & = \sup_{\| f \| \le 1, \, \| V \| \le 1, \| \bc \| \le 1, \| y 
    \| \le 1} \left| \left\langle f, \, \sum_{i=1}^{r} K_{W, \bc^{*} 
    \bv_{i}^{*}, y_{i}} \right\rangle_{\cH(K)} \right| \text{ (by 
    \eqref{reprod'}) } \\
    & = \sup_{ \| V \| \le 1, \| \bc \| \le 1, \| y 
    \| \le 1}  \left\| \sum_{i=1}^{r} K_{W, \bc^{*} 
    \bv_{i}^{*}, y_{i}} \right\| \\
    & =  \sup_{ \| V \| \le 1, \| \bc \| \le 1, \| y 
    \| \le 1} \left\langle K\left(\bigoplus_{1}^{r} W, \bigoplus_{1}^{r} W 
    \right) (V \bc \bc^{*} V^{*}) y, y \right\rangle^{1/2}
\end{align*}
where the last step is by the first part of the calculation 
\eqref{est2}. This completes the verification of \eqref{normevW}.

From the rest of the calculation \eqref{est2} we 
arrive at the uniform estimate (independent of $r$ and $s$)
$$
\|\boldsymbol{\rm ev}_{W}^{(r,s)} \| \le \| K(W,W)(1_{\cA^{m \times 
 m}})\|^{1/2}.
 $$
 
 For the case $m=1$ and $r=s=1$, we may specialize $V$ to $1_{\cA}$ 
 and $\bc$ to $1 \in {\mathbb C}$ to get
 \begin{align*}
     \| \boldsymbol{\rm ev}_{W}\|&  = \sup_{ \| V \| \le 1, \| \bc \| \le 1, \| y 
    \| \le 1} \langle K( W, W) (V \bc \bc^{*} V^{*}) y, y \rangle^{1/2} \\
    & \ge \sup_{\| y \| \le 1} \langle K(W,W)(1_{\cA}) y, y 
   \rangle^{1/2}  \\
   & = \| K(W,W)(1_{\cA}) \|^{1/2}
 \end{align*}
 We therefore have the squeeze play
\begin{align*}
\|K(W,W)(1_{\cA})\|^{1/2} & = 
\|\boldsymbol{\rm ev}_{W}\|_{\cL(\cH, \cL(\cA, \cY))} \\
& \le \|\boldsymbol{\rm ev}_{W}\|_{\cL(\cH, \cL_{\rm cb}(\cA, \cY))}
\le \|K(W,W)(1_{\cA})\|^{1/2}
\end{align*}
from which the string of equalities \eqref{m=1} fo the case $r=s=1$ 
and $m=1$ follows.

To prove (3), we use the formula \eqref{defH} for $H(Z)$.  For $Z \in 
\Omega_{n}$ and $F = [F_{ij}] \in (\cH(K)^{n})^{r \times s}$, we have
$$
({\rm id}_{{\mathbb C}^{r \times s}} \otimes H(Z))(F) \in 
(\cY^{n})^{r \times s} \cong \cL({\mathbb C}^{s}, \cY^{nr}).
$$
For $c = [c_{j}]_{j=1}^{s} \in {\mathbb C}^{s}$, we therefore wish to estimate the norm of
\begin{align*}
({\rm id}_{{\mathbb C}^{r \times s}} \otimes H(Z))(F) \cdot c & =
[ H(Z)(F_{ij})] \cdot [c_{j}] =
\left[ \sum_{j=1}^{s} c_{j} H(Z) (F_{ij})  \right]_{i=1}^{r} \\
& = \left[ \sum_{j=1}^{s} c_{j} F_{ij}(Z) (1_{\cA^{n \times n}}) 
\right]_{i=1}^{r}. 
\end{align*}
For $y = [y_{i}]_{i=1}^{r} \in \cY^{nr}$, we compute
\begin{align*}
    & \left\langle \left[ H(Z)(F_{ij}) \right] \cdot [c_{j}], y \right\rangle_{\cY^{mr}} =
 \left\langle \left[ \sum_{j=1}^{s} c_{j} F_{ij}(Z)(1_{\cA^{n \times n}}) \right]_{i=1}^{r}, 
 [y_{i}]_{i=1}^{r}  \right\rangle_{\cY^{nr}} \\
& \quad = \sum_{i=1}^{r} \sum_{j=1}^{s} c_{j} \left\langle F_{ij},\, 
K_{Z,1_{\cA^{n \times n}},y_{i}} \right\rangle_{\cH(K)^{n}} \\
& =\left\langle \left[ \sum_{j=1}^{s} F_{ij} c_{j} 
\right]_{i=1}^{r}, \, K_{ \oplus_{1}^{r} Z,\, 1_{\cA^{nr \times nr}}, y} 
\right\rangle_{\cH(K)^{nr}}.
\end{align*}
We then estimate
\begin{align*}
   & | \langle  [H(Z)(F_{ij})] \cdot [c_{j}], y \rangle_{\cY^{mr}} |
    \le \| [ \sum_{j=1}^{s} F_{ij} c_{j} ]_{i=1}^{r} \| \cdot 
\|K(\oplus_{1}^{r} Z, \oplus_{1}^{r} Z)(1_{\cA^{nr \times nr}}) 
\|^{1/2} \\
& \quad = \| [ \sum_{j=1}^{s} F_{ij} c_{j} ]_{i=1}^{r} 
\|_{\cH(K)^{nr}} \cdot \| K(Z,Z)(1_{\cA^{n \times n}})\|^{1/2}
\end{align*}
and it follows that 
$$
  \|  {\rm id}_{{\mathbb C}^{r \times s}} \otimes H(Z) \| \le \| 
  K(Z,Z)(1_{\cA^{n \times n}}) \|^{1/2}.
$$
for all $r, s \in {\mathbb N}$.  In particular we get the estimate
$$
  \| H(Z) \|_{\rm cb}  \le \| K(Z,Z)(1_{\cA^{n \times n}}) \|^{1/2}.
$$
On the other hand, from \eqref{defH} and \eqref{bevWV**} we read off that
$$
  \| H(Z) \|  = \sup_{\|y\| \le 1} \|K_{Z,1_{\cA^{n \times n}}, 
  y}\|_{\cH(K)^{n}} = \|K(Z,Z)(1_{\cA^{n \times n}})\|^{1/2}.
$$
We again get a squeeze play
\begin{align*}
    \|K(Z,Z)(1_{\cA^{n \times n}})\|^{1/2} = \|H(Z)\| \le 
    \|H(Z)\|_{\rm cb} \le \| K(Z,Z)(1_{\cA^{n \times n}})\|^{1/2}
\end{align*}
from which the string of equalities \eqref{normH(Z)} follows.
\end{proof}

\begin{remark} \label{R:EffrosRuan} \textbf{Boundedness vs.\ Complete 
    Boundedness in general.}  We have several comments exploring the 
    connections of Proposition \ref{P:cbRKHS} with the operator 
    algebra literature.
    
    \smallskip

    \textbf{1.} We note that statement (3) in Proposition \ref{P:cbRKHS} assures 
    us that the bounded operator $H(Z)$ between Hilbert spaces 
    $\cH(K)^{n}$ and $\cY^{n}$ is in fact completely bounded with cb 
    norm equal to its operator norm in $\cL(\cH(K)^{n}, \cY^{n})$.  
    This in fact is a general phenomenon for Hilbert space operators:
    in our notation, 
    {\em for Hilbert spaces $\cH$ and $\cK$, a linear operator $T$ 
    from $\cH$ to $\cK$ is bounded if and only if it is completely 
    bounded (as an operator between the operator spaces $\cH_{\rm 
    col}$ and $\cK_{\rm col}$) and moreover the identity map from $\cL(\cH, \cK)$ to 
    $\cL_{\rm cb}(\cH_{\rm col}, \cK_{\rm col})$ is a complete isometry} (see the result 
    of Effros-Ruan \cite[Theorem 4.1]{ER}).  
   
    \smallskip 
    
   \textbf{2.} Similarly, in case $\cA = {\mathbb C}$, $W \in \Omega_{1}$ and $f 
    \in \cH(K)$,  $f(W) \in \cL({\mathbb C}, \cY)$ is a Hilbert space operator 
    and hence $f(W)$ bounded as an element of $\cY \cong \cL({\mathbb 
    C}, \cY)$ implies that $f(W)$ is completely bounded with cb norm 
    equal to its operator norm. However, 
    in case $\cA \ne {\mathbb C}$ and/or $W \in \Omega_{m}$ with $m>1$, it would appear that 
    statement (2) in Proposition \ref{P:cbRKHS} is not automatic from 
    more general considerations. However, according to another result 
    of Effros-Ruan \cite{ER} as reformulated by Pisier (see 
    \cite[page 3986]{Pisier2004}), an element $u$ of $\cL(\cA^{m}, 
    \cY^{m})$ has cb norm at most 1 if and only if, for any 
    finite sequence $x_{1}, \dots, x_{N}$ of elements from $\cA$, 
    it happens that
 $$
     \sum_{i=1}^{N} \| u(x_{i})\|^{2}_{\cY^{m}} \le \| \sum_{i=1}^{N} 
     x_{i}^{*} x_{i} \|_{\cA^{m}}.
$$
After a rescaling, we get:  {\em $u$ in $\cL(\cA^{m}, 
    \cY^{m})$ has cb norm at most $M$ if and only if, for 
    any finite sequence $x_{1}, \dots, x_{n}$ of elements from 
    $\cA^{m}$,}
    \begin{equation}   \label{ERcriterion}
     \sum_{i=1}^{N} \| u(x_{i})\|^{2}_{\cY^{m}} \le M^{2} \| \sum_{i=1}^{N} 
     x_{i}^{*} x_{i} \|_{\cA^{m}}.
\end{equation} 

Let us apply this criterion for the case where $u = f(W)$ for an $f 
\in \cH(K)$, $W \in \Omega_{m}$ and $x_{i} = \sbm{ x_{i1} \\ \vdots 
\\ x_{im} } \in \cA^{m}$ for $i = 1, \dots, N$.  Making use 
of the canonical $*$-action $\sigma_{\cH(K)} = \sigma$ of $\cA$ on $
\cH(K)$ given by \eqref{rep1}, we note that
\begin{align*}
    f(W)(x_{i}) &  = f(W) \left( \sum_{j=1}^{m} E^{(m)}_{j} \otimes 
    x_{ij} \right)  \\
    & = \sum_{j=1}^{m} \left( \sigma(x_{ij})f \right)\left(W \right) 
    \left( E_{j}^{(m)} \otimes 1_{\cA} \right).
\end{align*}
For $y \in \cY^{m}$ we then compute
\begin{align*}
    \langle f(W)(x_{i}),y \rangle_{\cY^{m}} & = 
\left\langle \sum_{j=1}^{m} \left(\sigma(x_{ij})f \right) \left( W 
 \right) \left(E^{(m)}_{j}  \otimes 1_{\cA}\right), y \right\rangle_{\cY^{m}}  \\
 & = \sum_{j=1}^{m} \left\langle \sigma(x_{ij}) f, K_{W, 
 E_{j}^{(m)*}\otimes 1_{\cA} , y } \right\rangle_{\cY^{m}}.
\end{align*} 
Hence
\begin{align*}
    | \langle f(W)(x_{i}), y  \rangle_{\cY^{m}} | & =
    \left| \left\langle \sum_{j=1}^{m} \sigma(x_{ij})f, K_{W, 
    E_{j}^{(m)*} \otimes 1_{\cA}, y} \right\rangle_{\cH(K)} \right|  \\
    & \le \left\| \sum_{j=1}^{m} \sigma(x_{ij})f \right\| 
    \left\langle K(W,W)\left(E_{j}^{(m)} E_{j}^{(m)*} \otimes 1_{\cA} \right) y, y 
    \right\rangle^{1/2}  \\
    & \le \left\| \sum_{j=1}^{m} \sigma(x_{ij}) f \right\| 
    \|K(W,W)(1_{\cA^{m \times m}}) \|^{1/2} \|y\| 
\end{align*}
and we conclude that
$$
  \| f(W) (x_{i})\|^{2} \le \| \sum_{j=1}^{m} \sigma(x_{ij}) f \|^{2}
   \|K(W,W)(1_{\cA^{m \times m}}) \|.
$$
Let us set $M: = \|K(W,W)(1_{\cA^{m \times m}})\|^{1/2}$ and now sum over 
$i$ to get
\begin{align*}
    \sum_{i=1}^{N} \| f(W) x_{i}) \|^{2} & \le M^{2} \sum_{i=1}^{N} 
    \left\langle \sum_{j,\ell = 1}^{m} \sigma(x_{ij}) f, 
    \sigma(x_{i\ell})f \right\rangle_{\cH(K)}  \\
    & = M^{2} \sum_{i=1}^{N} \left\langle \sigma\left( \sum_{j,\ell = 
    1}^{m} x_{il}^{*} x_{ij} \right) f,\, f \right\rangle_{\cH(K)} \\
     & \le M^{2} \left\| \sum_{i=1}^{N} x_{i}^{*} x_{i} 
    \right\|_{\cA^{m}} \| f \|^{2}_{\cH(K)}.
\end{align*}
As a consequence of the Effros-Ruan--Pisier criterion 
\eqref{ERcriterion}, we conclude that $f(W)$ is a cb map from 
$\cA^{m}$ to $\cY^{m}$ with cb norm at most $\|f\|_{\cH(K)} 
\cdot \|K(W,W)(1_{\cA^{m \times m}})\|$, thereby giving an alternate 
proof of statement (2) in Proposition \ref{P:cbRKHS}.

\smallskip

\textbf{3.} An equivalent formulation of the Effros-Ruan criterion is (see 
\cite[page 3986]{Pisier2004}):  {\em $u$ in $\cL(\cA, \cY)$ has 
cb norm at most 1 if and only if there is a state $\varphi$ on $\cA$ such 
that }
\begin{equation} \label{ERcriterion'}
\text{\em for all } x \in \cA,\, \| u(x)\|^{2}_{\cY} \le \varphi(x^{*}x).
\end{equation}
We can use our theory of global Reproducing Kernel Hilbert Spaces to 
prove the sufficiency side (presumably the easy side) of this 
criterion as follows.  Assume that there is a state $\varphi$ on $\cA$ so 
that \eqref{ERcriterion'} holds.  Use the state $\varphi$ to define an 
inner product on $\cA$:
\begin{equation}   \label{innerprod}
\langle a, b \rangle_{\cH^{\circ}} = \varphi(b^{*} a).
\end{equation}
View the elements of $\cA$ as elements of $\cL(\cA, \cY)$ according 
to the formula
$$
   a \cong f_{a} \colon x \mapsto u(xa).
$$
A consequence of the assumption \eqref{ERcriterion'} is that 
$$
   \| f_{a}(x) \|^{2}_{\cY} = \| u(xa)\|^{2}_{\cY} \le \varphi(a^{*}x^{*}xa) \le \| x 
   \|^{2}_{\cA} \varphi(a^{*} a) = \|x\|^{2}_{\cA} \| a \|^{2}_{\cH^{\circ}}
$$
for each $x \in \cA$.  Hence when we let $\cH$ be the completion of 
$\cH^{\circ}$ in the $\cH^{\circ}$ inner product, the elements $f$ of 
the completion can still be identified as elements of  $\cL(\cA, \cY)$
with the property that $\|f \|_{\cL(\cA, \cY)} \le \| f \|_{\cH}$. 

For $v$ a fixed element of $\cA$, one can check that the map
$\sigma(v) \colon f_{a} \mapsto f_{va}$ is a $*$-representation of 
$\cA$ on $\cL(\cH^{\circ})$ and extends to a $*$-representation of $\cA$ 
on $\cL(\cH)$ of the functional form:
$$
  (\sigma(v) f)(x) = f(xv).
$$

We use this $*$-representation to show that, for each $f \in \cH$,
the map 
$$
{\rm id}_{{\mathbb C}^{n}} \otimes f \colon \sbm{a_{1} \\ \vdots \\ a_{n}} \mapsto 
\sbm{ f(a_{1}) \\ \vdots \\ f(a_{n}) }
$$
has $\cL(\cA^{n}, \cY^{n})$-norm bounded by $\| f \|_{\cH}$ as well.
Indeed, note that
\begin{align*}
    \left\| \sbm{ f(a_{1}) \\ \vdots \\ f(a_{n}) } \right\|^{2}_{\cY^{n}} 
    & = \sum_{i=1}^{n} \| f(a_{i})\|^{2}_{\cY} \\
    & = \sum_{i=1}^{n} \| (\sigma(a_{i}) f)(1_{\cA}) \|^{2}_{\cH} \\
    & \le \sum_{i=1}^{n} \| \sigma(a_{i}) f\|^{2}_{\cH} \| 
    1_{\cA}\|^{2}_{\cA} \\
& = \sum_{i=1}^{n} \langle \sigma(a_{i}^{*} a_{i}) f, \, f \rangle_{\cH} \\ 
& = \left \langle \sigma\left( \sum_{i=1}^{n} a_{i}^{*} a_{i} \right) 
f, \, f \right\rangle_{\cH}  \\
& \le \left\| \sum_{i=1}^{n} a_{i}^{*} a_{i} \right\|_{\cA} \| f 
\|^{2}_{\cH}  \\
& = \left\| \sbm{ a_{1} \\ \vdots \\ a_{n}} \right\|^{2}_{\cA^{n}} \| 
f \|^{2}_{\cH}.
\end{align*}

We have now verified that $\cH$ satisfies all the hypotheses of 
Theorem \ref{T:RKHS} specialized to the case where $\Omega = 
\amalg_{n=0}^{\infty} \{ s_{0} \otimes I_{n}\}$. We conclude that 
$\cH = \cH(\varphi)$ is a nc reproducing kernel Hilbert space over 
the nc envelope  $\amalg_{n=1}^{\infty} \{ s_{0} \otimes 
I_{n}\}$ of a singleton-point set $\{ s_{0}\}$ for some completely 
positive map $\varphi \colon \cA \to 
\cL(\cY)$. By statement (2) 
in Proposition \ref{P:cbRKHS}, it follows that each $f \in \cH$ 
is actually completely bounded: 
$$
\| f \|_{\cL(\cA, \cY)} \le  \|f \|_{\cL_{\rm cb}(\cA, \cY)}   \le \|f \|_{\cH}.
$$
as an element of $\cL(\cA, \cY)$.  In 
particular, $u = f_{1_{\cA}}$ is completely bounded, and the 
sufficiency direction of the second Effros-Ruan criterion 
\eqref{ERcriterion'} follows.
\end{remark}

\subsubsection{Reproducing kernel Hilbert spaces in the sense of 
Aronszajn}  \label{S:Aronszajn}
Let $\Omega$ be a set of points (not necessarily having 
any noncommutative structure), let $\cY$ be a Hilbert space 
and suppose that $K$ is a function from $\Omega \times \Omega$ to 
$\cL(\cY)$.  We say that $K$ is a \textbf{positive kernel}  (in the 
sense of Aronszajn \cite{Aron} who worked out much of the theory for 
the case $\cY = {\mathbb C}$) if
$$
  \sum_{i,j= 1}^{N} \langle K(z_{i},z_{j}) y_{j}, y_{i} 
  \rangle_{\cY} \ge 0
$$
for all choices of points $z_{1}, \dots, z_{N}$ and vectors 
$y_{1}, \dots, y_{n} \in \cY$, $N=1,2,\dots$; in other words, the 
matrix $[ K(w_{i}, w_{j})]_{i,j=1}^{N}$ is positive in 
$\cL(\cY)^{N \times N}$ for all choices of $z_{1}, \dots, 
z_{N} \in \Omega$, for any $N=1,2,\dots$.   
Equivalently, if we let $\widetilde \Omega = [\Omega]_{\rm nc}$ be the nc envelope 
of $\Omega$ (see Section \ref{S:ncsets}), i.e., the nc set $\widetilde \Omega$ with
$$
 \widetilde \Omega_{n} = \left\{ \sbm{ z_{1} & & \\ & \ddots & \\ & & z_{n} } 
 \colon z_{1}, \dots, z_{n} \in \Omega \right\}
 $$
 and we define $\widetilde K \colon \widetilde \Omega \times 
 \widetilde \Omega \to \cL({\mathbb 
 C}, \cL(\cY))_{\rm nc}$ by
 $$
 \widetilde K(Z,W)(P) = [ K(z_{i}, w_{j}) p_{ij} ]
 $$
 if $Z = \sbm{ z_{1} & & \\ & \ddots & \\ & & z_{n}}$, 
 $W = \sbm{ w_{1} & & \\ & \ddots & \\ & & w_{m}}$, and
 $P = [ p_{ij}]_{1 \le i \le n, \, 1 \le j \le m}$,
 then $\widetilde K$ is a cp global kernel as defined in Section 
 \ref{S:main} and in fact is the unique extension of $K$ (defined on 
 $\Omega \times \Omega$ to a cp global kernel on $[\Omega]_{\rm nc} 
 \times [\Omega]_{\rm nc}$.  Conversely, if $\cA = {\mathbb C}$ and if
 $\widetilde K$ is any global kernel  on $\widetilde \Omega = [\Omega]_{\rm nc}$, 
 then $K(z,w): = \widetilde K(z,w)(1)$ is a positive kernel in the 
 sense of Aronszajn on $\Omega$.
 The versions of Theorems \ref{T:cpker} and \ref{T:RKHS} for this 
 special case have been mainstays in the operator theory literature 
 for many decades now (see e.g.\ \cite{AMcC-book}).
 
 Conversely, there are a couple of  ways to associate a positive kernel in 
 the sense of Aronszajn to a general cp global kernel which we now discuss. 
 
 \smallskip
 
 \textbf{1. Fix the $\cA$-argument.}  Let $K \colon \Omega \times 
 \Omega \to \cL(\cA, \cL(\cY))_{\rm nc}$ be a cp global kernel and let $P$ be a 
 fixed positive element of $\cA^{n \times n}$ for some $n \in {\mathbb N}$.
 Then it is easily seen that the kernel $K_{P} \colon \Omega_{n} \times 
 \Omega_{n} \to \cL(\cY^{n})$ defined by
 $$
   K_{P}(Z,W) = K(Z,W)(P)
 $$
 is a positive kernel in the sense of Aronszajn.  There is a 
 resulting reproducing kernel Hilbert space $\cH(K)$  whose elements 
 are $\cY^{n}$-valued functions defined on the set of points 
 $\Omega_{n}$.  Any Kolmogorov decomposition \eqref{Koldecom} for the 
 cp global kernel $K$ induces a standard Aronszajn-Kolmogorov 
 decomposition for $K_{P}$:  if $P \succeq 0$ in $\cA^{n \times n}$ and 
 we factor $P$ as $P = U U^{*}$ with $U \in \cA^{n \times k}$, then
 $$
   K_{P}(Z,W) = H_{P}(Z) H_{P}(W)^{*} \text{ where } H_{P}(Z) : = H(Z) 
   (I_{{\mathbb C}^{n \times k}} \otimes \sigma)(U).
 $$
 
 In the special case where $\cA = {\mathbb C}$, the next result shows 
 that complete positivity 
 of the nc kernel $K$ can often be checked just by looking at 
 various kinds of positivity for the kernels $K_{I_{n}}$ for each 
 $n=1,2,\dots$.
 
 \begin{theorem}  \label{T:cp=p} 
     Assume that $\cV$ is an operator space and that $\Omega$ is a nc 
     subset of $\cV$.
     Suppose that $K \colon \Omega \times \Omega \to \cL({\mathbb 
     C}, \cY)_{\rm nc}$ is a nc kernel. Assume either:
     \begin{enumerate}
	 \item[(a)]  (i) the underlying vector space $\cV$ is 
	 ${\mathbb C}^{d}$ (see the second bullet in the discussion 
	 in Subsection \ref{S:ncsets}) and the nc set $\Omega \subset 
	 {\mathbb C}^{d}_{\rm nc} \cong \amalg_{n=1}^{\infty} ({\mathbb 
	 C}^{n \times n})^{d}$ has the special form 
	$\Omega = \amalg_{n=0}^{\infty} N(0; \epsilon)_{n}$ where 
	$N(0; \epsilon)_{n}$ is the $\epsilon$-ball in $({\mathbb 
	C}^{n \times n})^{d}$ centered at the origin of radius 
	$\epsilon$ for some $\epsilon > 0$, and
	  (ii) 
	 $K_{I_{n}} \colon \Omega_{n} \times \Omega_{n} \to \cL(\cY^{n})$ 
	 given by $K_{I_{n}}(Z,W) = K(Z,W)(I_{n})$ is a positive 
	 kernel in the sense of Aronszajn on $\Omega_{n} = N(0; 
	 \epsilon)_{n}$ for each $n \in {\mathbb 
	 N}$, or
	 
	 \item[(b)]  (i) $\Omega$ is \textbf{similarity-invariant}:  given $W 
	 \in \Omega_{n}$ and $S \in {\mathbb C}^{n \times n}$ with 
	 $S$ invertible, it holds that $S W S^{-1} \in \Omega_{n}$, 
	 and (ii) $K_{I_{n}}(Z,Z) : = K(Z,Z)(I_{n})$ is positive 
	 semidefinite for all $Z \in \Omega_{n}$ for all $n \in 
	 {\mathbb N}$.
\end{enumerate}
Then $K$ is a cp nc kernel.
 \end{theorem}
 
 \begin{proof}
     The proof of case (a) which we have makes use of the connections 
     between nc RKHSs studied here and the formal nc RKHSs studied 
     in \cite{NFRKHS}; we therefore postpose the proof of case (a) to 
     Subsection \ref{S:NFRKHS} where formal nc RKHSs are reviewed.

     It remains to show that hypothesis (b) $\Rightarrow$ $K$ is cp.
     We therefore assume (b), so $K(Z,Z)(I_{n}) \succeq 0$ for each 
     $n \in {\mathbb N}$.  We must show that then $K(Z,Z)(P) \succeq 
     0$ for any $P \succeq 0$ in ${\mathbb C}^{n \times n}$.  Assume 
     first that $P$ is invertible.  Then $P$ factors as $P = S S^{*}$ 
     with $S \in {\mathbb C}^{n \times n}$ invertible.  By assumption 
     $\widetilde Z = S^{-1} Z S$ is again in $\Omega_{n}$. By 
     similarity-invariance of the nc kernel $K$, we have
     $$
     K(Z,Z)(P) = K(Z,Z)(S S^{*}) = S\,  K(\widetilde Z, \widetilde Z) 
     (I) \, S^{*}  \succeq 0.
     $$
     As any positive $P$ can be approximated by strictly positive 
     definite operators, it follows that $K$ is cp.
  \end{proof}  
  
  As an immediate corollary of case (a) of Theorem \ref{T:cp=p}, 
  we obtain the following result.
  
  \begin{corollary} \label{C:nilp}  
      Suppose that the nc subset $\Omega$ of ${\mathbb C}^{d}_{\rm 
      nc} \cong    \amalg_{n=1}^{\infty}({\mathbb C}^{n \times n})^{d}$ 
      is taken to be $({\rm Nilp})^{d} = \amalg_{n=1}^{\infty} ({\rm 
      Nilp})^{d}_{n}$ where $({\rm Nilp})^{d}_{n}$ consists of 
      $d$-tuples $Z = (Z_{1}, \dots, Z_{d})$ of $n \times n$ matrices 
      which are jointly nilpotent in the sense that $Z^{\alpha} = 0$ 
      as soon as the length $|\alpha|$ of the word $\alpha$ is 
      sufficiently large.  Suppose that $K$ is a nc kernel on $({\rm 
      Nilp})^{d}$ with the property that $K(Z,Z)(I_{n}) \ge 0$ for all 
      $Z \in ({\rm Nilp})^{d}_{n}$ for all $n \in {\mathbb N}$.  Then 
      $K$ is cp.
   \end{corollary}
   
   \begin{proof}  It suffices to observe that $({\rm Nilp})^{d}$ is 
       similarity-invariant for each $n \in {\mathbb N}$ and then 
       quote case (b) of Theorem \ref{T:cp=p}.
   \end{proof}
 
  To make precise the connections between $\cH(K_{P})$ and $\cH(K)$ in 
 general, we use the model Kolmogorov decomposition for $K$ as in \eqref{defH} 
 and \eqref{modelKoldecom}.  Fix a factorization $P = U U^{*}$ of $P$ 
 where, say, $U \in \cA^{n \times k}$.
 It then follows that the space $\cH(K_{P})$ can be identified as the 
 lifted norm space
 $$
   \cH(K_{P}) = \{ H(\cdot) \sigma(U) f \colon f \in \cH(K)^{k} \}
 $$
 where $H$ is given by \eqref{defH}, i.e., as the space of functions
 $$
  Z \mapsto \left({\rm id}_{{\mathbb C}^{n \times k}} \otimes \sigma\right)(U) 
   \sbm{ f_{1} \\ \vdots \\ f_{k}}(Z) \left(1_{\cA^{n 
  \times n}}\right) =:  \sbm{ f_{1} \\ \vdots \\ f_{k}}  (Z) (U)
  $$
 with norm so that the map
 $$
 \Gamma_{U} \colon f \mapsto  \sbm{ f_{1} \\ \vdots \\ f_{k}}(Z) (U)
 $$
 is a coisometry:
 $$ \| \Gamma_{U} f \|_{\cH(K_{P})} = \| P_{({\rm Ker} 
 \Gamma_{U})^{\perp}} f\|_{\cH(K)^{k}}.
 $$
 Note that 
 \begin{align*}
     & {\rm Ker}\, \Gamma_{U} = \left\{ \sbm{ f_{1} \\ \vdots \\ f_{k}} 
 \in \cH(K)^{k} \colon \sbm{ f_{1} \\ \vdots \\ f_{k}}(Z)(U) = 0 
 \text{ for all } Z \in \Omega_{n} \right\}, \\
 & ({\rm Ker}\,  \Gamma_{U})^{\perp} = \overline{\rm span} \left\{ K_{W, 
 U^{*}, y} \colon W \in \Omega_{n},\, y \in \cY^{n}\right\}.
 \end{align*}
 
 Let us now assume that $\cA$ is a von Neumann algebra (e.g., $\cA = 
 \cL(\cK)$ for a Hilbert space $\cK$).  Then the Douglas lemma (see 
 \cite{Douglas}) holds 
 inside $\cA$: {\em $U' \in \cA^{n \times k'}$, $U \in \cA^{n \times k}$ 
 and $U' U^{\prime *} \preceq U U^{*}$ $\Rightarrow$ there is a $C \in \cA^{k 
 \times k'}$ with $C C^{*} \le 1_{\cA^{k \times k}}$ and $UC = U'$.}
 If $P' = U' U^{\prime *} \le P = U U^{*}$ in $\cA^{n \times n}$, 
 then $K(\cdot, \cdot)(P') \preceq K(\cdot, \cdot)(P)$ as 
 $\cL(\cY)$-valued kernels on $\Omega_{n}$ and it follows that 
 $\cH(K_{P'})$ is contained contractively in $\cH(K_{P})$.  Using the 
 factorization $P' = U' U^{\prime *}$ and $P = U U^{*}$, we then 
 have 
 \begin{align*}
  \cH(K_{P'}) = &\left\{ Z \mapsto \sbm{ f_{1} \\ \vdots \\ 
     f_{k'}}(Z)(U') = \left( ({\rm id}_{{\mathbb C}^{k \times k'}} \otimes 
     \sigma)(C) \sbm{ f_{1} \\ \vdots \\ f_{k'}} \right)(Z) (U) 
     \colon \right. \\
  & \quad \left.  \sbm{ f_{1} \\ \vdots \\ f_{k'} } \in \cH(K)^{k'}
     \right\}, \\
   \cH(K_{P})&  = \left\{ Z \mapsto \sbm{ f_{1} \\ \vdots \\ f_{k} }(Z) 
 (U) \colon \sbm{f_{1} \\ \vdots \\ f_{k}} \in \cH(K)^{k} \right\}
\end{align*}
and we see that $({\rm id}_{{\mathbb C}^{k \times k'}}  \otimes \sigma)(C) 
\colon \cH(K)^{k'} \to \cH(K)^{k}$ satisfies $\sigma(C) ({\rm Ker}\, 
\Gamma_{U'}) \subset {\rm Ker}\, \Gamma_{U}$ and hence defines a 
mapping from $\cH(K_{P'})$ to $\cH(K_{P})$.  However it is not the 
case that $({\rm id}_{{\mathbb C}^{k \times k'}} \otimes \sigma)(C)$ maps 
$({\rm Ker}\, \Gamma_{U'})^{\perp}$ into $({\rm Ker}\, 
\Gamma_{U})^{\perp}$ so, after the identifications
$$
   \cH(K_{P}) \cong \cH(K)^{k} \ominus {\rm Ker}\, \Gamma_{U}, \quad
   \cH(K_{P'}) \cong \cH(K)^{k} \ominus {\rm Ker}\, \Gamma_{U'},
$$
the map $(I_{{\mathbb C}^{k \times k'}} \otimes \sigma)(C) \colon 
\cH(K_{P'}) \to \cH(K_{P})$ becomes
$$
  P_{({\rm Ker}\, \Gamma_{U})^{\perp}} (I_{{\mathbb C}^{k \times k'}} 
  \otimes \sigma)(C)|_{({\rm Ker}\, \Gamma_{U'})^{\perp}}.
$$
On the other hand the map $({\rm id}_{{\mathbb C}^{k \times k'}} \otimes 
\sigma)(C)^{*}$ from $({\rm Ker}\, \Gamma_{U})^{\perp}$ to $({\rm 
Ker}\, \Gamma_{U'})^{\perp}$ is given explicitly on kernel elements by
$$
({\rm id}_{{\mathbb C}^{k \times k'}} \otimes \sigma)(C)^{*} \colon
K_{W,U^{*},y} \mapsto K_{W, U^{\prime *}, y} \text{ with } 
U^{\prime *} = C^{*} U^{*}.
$$
However this calculus of spaces $\cH(K_{P})$ has awkward aspects when 
one tries to apply it to get at more detailed structure of the 
original space $\cH(K)$.  Unresolved issues are:  
\begin{itemize}
    \item Analyze in terms of $P$ and $P'$ when $\cH(K_{P}) \cap 
    \cH(K_{P'}) = \{0\}$ and when $\cH_{P'} \subset \cH_{P}$.
    
    \item Analyze how to recover $\cH(K)$ from all the auxiliary 
    Aronszajn reproducing kernel Hilbert spaces
     $\{ \cH(K_{P}) \colon P \in \cA^{n \times n} \text{ with } 
    P \ge 0\}$.
 \end{itemize}
 
 \smallskip
 
 \textbf{2. Use the $\cA$-argument  to enlarge the set of points.}
 In this approach, we view the set of points as $\Omega \times \cA$ 
 rather than as $\Omega$; in case $\Omega$ is the nc envelope $[ \{ 
 s_{0}\}]_{\rm nc}$ of a singleton set $\{s_{0}\}$, this idea appears 
 in statement (2) of Theorem \ref{T:Stinespring} above.
 Given a cp global kernel $K$ from $\Omega 
 \times \Omega$ to $\cL(\cA, \cL(\cY))_{\rm nc}$ and $n \in {\mathbb N}$, 
  we define an Aronszajn-type kernel $K_{n}$ from $(\Omega_{n} 
  \times \cA^{n}) \times (\Omega \times \cA^{n})$ to $\cL(\cY)$ by
  $$
   \widetilde K_{n}((Z,u), (W,v)) = K(Z,W)(u v^{*}).
  $$
  Then $K_{n}$ has Aronszajn-Kolmogorov decomposition
  $$
  K_{n}((Z,u), (W,v)) = H_{n}(Z,u) H_{n}(W,v)^{*} \text{ with }
  H_{n}(Z,) = H(Z) ({\rm id}_{{\mathbb C}^{n}} \otimes \sigma)(u).
  $$
  Then $\cH(K_{n})$ can be identified as the range space of the map
  $\Gamma_{n} \colon \cH(K) \to \cH(K_{n})$ given by
  $$
   \Gamma_{n} \colon f \mapsto \left( (Z,u) \mapsto \left( ({\rm 
   id}_{{\mathbb C}^{n}} \otimes \sigma)(u) f 
   \right)(Z)(1_{\cA^{n \times n}} ) \right)
  $$
  with norm so that $\Gamma_{n}$ is a coisometry.  Note that
  $$
   ({\rm Ker}\, \Gamma_{n})^{\perp} = \overline{\rm span} \{K_{W, 
   v^{*}, y} \colon W \in \Omega_{n}, v \in \cA^{n}, y \in \cY^{n} \}
  $$
  This collection of subspaces forms an increasing sequence of 
  subspaces with dense union in $\cH(K)$, so the corresponding 
  orthogonal projection $P_{n}$ of $\cH(K)$ onto $({\rm Ker}\, 
  \Gamma_{n})^{\perp}$ converges strongly to the $I_{\cH(K)}$.  We 
  therefore have, for all $f \in \cH(K)$,
  $$
  \| f\|_{\cH(K)} = \lim_{n \to \infty} \| P_{n} f \|_{\cH(K)} = \lim \| 
  \Gamma_{n} f \|_{\cH(K_{n})}.
  $$
  
  It is also possible to consider instead functions on $\Omega_{n} 
  \times \cA^{n \times N}$ for any $N \in {\mathbb N}$ in 
  the Aronszajn-type reproducing kernel Hilbert space determined by 
  the kernel $K_{n,N}$ defined by
  $$
  K_{n,N}((Z,U), (W,V)) = K(Z,W)(U V^{*})
  $$
  and then obtain results analogous to those mentioned in the 
  previous paragraph, but with $\cH(K)^{N}$ replacing $\cH(K)$.
  In this way we arrive at a picture of the reproducing kernel 
  Hilbert space $\cH(K)$ associated with a cp global kernel as an asymptotic 
  limit of Aronszajn-type reproducing kernel Hilbert spaces 
  $\cH(K_{n})$, or more generally, $\cH(K)^{N}$ as the asymptotic 
  limit of the family $\cH(K_{n,N})$.

\subsubsection{Completely positive kernels in the sense of 
Barreto-Bhat-Liebscher-Skeide} \label{S:BBLS}
A unified setting for completely positive maps between 
$C^{*}$-algebras and Aronszajn-type positive kernels already appears 
in the work of Barreto-Bhat-Liebscher-Skeide \cite[Section 3]{BBLS}.
Given a point set $\Omega_{1}$, a $C^{*}$-algebra $\cA$ and a Hilbert 
space $\cY$, we say that the function $K$ from $\Omega_{1} \times 
\Omega_{1}$ 
to $\cL(\cA, \cL(\cY))$ is a \textbf{completely positive kernel} (in 
the sense of Barreto-Bhat-Liebscher-Skeide) if (one among several 
equivalent definitions), for all choices of $z_{1}, \dots, 
z_{N} \in \Omega_{1}$, $a_{1}, \dots, a_{N} \in \cA$, and $y_{1}, 
\dots, y_{N} \in \cY$ for any $N \in {\mathbb N}$,
\begin{equation}  \label{BBLS1}
\sum_{i=1}^{N} \langle K(z_{i}, z_{j})(a_{i}^{*} a_{j}) y_{j}, y_{i} 
\rangle_{\cY} \ge 0,
\end{equation}
or equivalently, the map from $\cA$ to $\cL(\cY)^{N \times N}$ given 
by
\begin{equation}   \label{BBLS2}
 a \mapsto [K(z_{i}, z_{j})(a)]_{i,j=1 \dots, N}
\end{equation}
is a completely positive map for any choice of $z_{1}, \dots, z_{N} 
\in \Omega_{1}$ for any $N=1,2,\dots$.
Indeed, in case $\cA = {\mathbb C}$ and we define $\widetilde K$ 
from $\Omega_{1} \times \Omega_{1}$ into $\cL(\cY)$ by $\widetilde 
K(z,w) = K(z,w)(1)$, then $\widetilde K$ is an Aronszajn-type kernel 
and any Aronszajn-type kernel arises in this way from a 
BBLS-completely positive kernel with $\cA = {\mathbb C}$.  Similarly, 
if $\Omega_{1}$ consists of a single point, say $z_{0}$, and we 
define $\varphi \colon \cA \to \cL(\cY)$ by $\varphi(a) = K(z_{0}, 
z_{0})(a)$, then $\varphi$ is a (linear) completely positive map and 
any completely positive map arises in this way from a BBLS-positive 
kernel over the 1-point set $\Omega_{1} = \{z_{0}\}$ (this last 
observation already appears in Stinespring's paper \cite[proof of 
Theorem 1]{Stinespring}).

To put the ideas into our framework of cp global kernels, we proceed as follows.
Let $\Omega$ be the nc envelope $[\Omega]_{\rm nc}$ of $\Omega_{1}$:
\begin{equation}   \label{globext}
\Omega = \amalg_{n=1}^{\infty}
 \left\{ \sbm{ z_{1} & & \\ & \ddots & \\ & & z_{n} } \colon z_{j} \in \Omega_{1}
 \text{ for } 1 \le j \le n \right\}.
\end{equation}
Given a function $K$ from $\Omega_{1} \times \Omega_{1}$ to $\cL(\cA, 
\cL(\cY))$, let $\widetilde K$ be the unique extension of $K$ to a 
global kernel on $\Omega$ (i.e., $\widetilde K$ is the unique 
extension of $K$ to a function mapping $\Omega_{n} \times \Omega_{m}$ 
into $\cL(\cA^{n\times m}, \cL(\cY)^{n \times m})$ for all $n,m \in 
{\mathbb N}$ which satisfies the ``respects direct sums'' property 
\eqref{kerds}). We leave it to the reader to verify: {\em Then 
$\widetilde K$ is  a cp global kernel.  
Conversely, for the special case where $\Omega$ is the nc envelope of $\Omega_{1}$,
any cp global kernel arises in this way from a BBLS-completely 
positive kernel on $\Omega_{1}$.}
Thus, from the point of view here, the notion of BBLS-completely 
positive kernel corresponds to the special case of a general cp global 
kernel where the set of points $\Omega$ is {\em commutative} (i.e., 
$\Omega_{n}$ consists only of commuting diagonal matrices).

We can use the BBLS-complete positivity condition to define a notion 
of completely positive global/nc kernel on subsets $\Omega$ which are 
not necessarily nc subsets of the ambient nc set $\cS_{\rm nc}$ and or 
$\cV_{\rm nc}$.  Suppose first that $\Omega$ is any subset of 
$\cS_{\rm nc}$ and that $K \colon \Omega \times \Omega \to \cL(\cA, 
\cL(\cY))_{\rm nc}$ is a graded kernel, i.e., $K$ satisfies condition 
\eqref{kergraded} with $\cV_{1} = \cA$ and $\cV_{0} = \cL(\cY)$.
Let us say that $K$ is a \textbf{cp global kernel} on $\Omega$ if
\begin{enumerate}
    \item $K$ \textbf{respects direct sums}, i.e., $K$ satisfies condition 
    \eqref{kerds} for any choice of $Z, \widetilde Z, W, \widetilde W 
    \in \Omega$ for which it is the case that $\sbm{Z & 0 \\ 0 & 
    \widetilde Z}$ and $\sbm{W & 0 \\ 0 & 
    \widetilde W}$ are also in $\Omega$, and
\item $K$ satisfies the BBLS-complete positivity condition 
\eqref{BBLS1} or equivalently \eqref{BBLS2}, i.e.,
\begin{equation}   \label{BBLS1'}
\sum_{i=1}^{N} \langle K(Z^{(i)}, Z^{(j)})(a_{i}^{*} a_{j}) y_{j}, y_{i} 
\rangle_{\cY} \ge 0,
\end{equation}
for all $Z^{(1)} \in \Omega_{n_{1}}$, $\dots$, $Z^{(N)} \in 
\Omega_{n_{N}}$, $a_{1} \in \cA^{\kappa \times n_{1}}$, $\dots$,
$a_{N} \in \cA^{\kappa \times n_{N}}$ for any $N, \kappa \in {\mathbb 
N}$, or equivalently,  the map from $\cA^{{\mathbf N} \times {\mathbf 
N}}$ to $\cL(\cY)^{{\mathbf N} 
\times {\mathbf N}}$ given by
\begin{equation}  \label{BBLS2'}
[a_{ij}]_{i,j =1, \dots, N} \mapsto [ K(Z^{(i)}, Z^{(j)})(a_{ij}) ]_{i,j=1, \dots,N}
\end{equation}
is completely positive for any choice of $Z^{(1)} \in \Omega_{n_{1}}$, $\dots$, 
$Z^{(N)} \in \Omega_{n_{N}}$ where we set ${\mathbf N} = \sum_{i=1}^{N} 
n_{i}$ (where $[a_{ij}]_{i,j=1, \dots N}$ with $a_{ij} \in \cA^{n_{i} 
\times n_{j}}$ is a generic element of $\cA^{{\mathbf N} \times 
{\mathbf N}}$).
\end{enumerate}
If $\Omega$ is not already a nc subset, we extend $K$ to 
$\Omega_{\rm nc} \times \Omega_{\rm nc}$ (where $\Omega_{\rm nc}$ is 
the nc envelope of $\Omega$ consisting of all possible finite direct 
sums of elements of $\Omega$) by
\begin{equation}   \label{tildeK}
\widetilde K\left( \sbm{ Z^{(1)} & & \\ & \ddots & \\ & & Z^{(N)}},
 \sbm{ W^{(1)} & & \\ & \ddots & \\ & & W^{(M)}}\right)([a_{ij}]) =
 [ K(Z^{(i)}, W^{(j)})([a_{ij}]) ]_{i = 1, \dots, N; j = 1, \dots, M}
\end{equation}
for any $Z^{(1)}, \dots, Z^{(N)}, W^{(1)}, \dots, W^{(N)} \in \Omega$.  If it 
happens that $\sbm{ Z^{(1)} & & \\ & \ddots & \\ & & Z^{(N)}}$ 
and  $\sbm{ W^{(1)} & & \\ & \ddots & \\ & & W^{(M)}}$ are already in 
$\Omega$, then the formula \eqref{tildeK} is consistent with how $K$ 
is already defined by the ``respects direct sums'' condition.
Then $\widetilde K$ is a cp global kernel on $\Omega_{\rm nc} \times 
\Omega_{\rm nc}$ which when restricted to $\Omega \times \Omega$ 
agrees with $K$.

Similarly, if $\Omega$ is a (not necessarily nc) subset of $\cV_{\rm 
nc}$  for a complex vector space $\cV$ and if $K$ is a 
graded function from $\Omega \times \Omega$ into $\cL(\cA, \cL(\cY))_{\rm 
nc}$ (so \eqref{kergraded} is satisfied, we say that $K$ is a \textbf{cp nc kernel}
on $\Omega$ if
\begin{enumerate}
    \item $K$ \textbf{respects intertwinings}, i.e., condition 
    \eqref{kerintertwine} holds, and 
    \item $K$ satisfies the BBLS-complete positivity condition 
    \eqref{BBLS1'} or equivalently \eqref{BBLS2'} on $\Omega$.
\end{enumerate}
Under these conditions, $K$ extends to a well-defined cp nc kernel $\widetilde K$ 
on the nc envelope $\Omega_{\rm nc}$ of of $\Omega$.

We mention that the reproducing kernel Hilbert space associated with 
a BBLS-completely positive kernel was studied in the paper of 
Ball-Biswas-Fang-ter Horst \cite{BBFtH} with the $\cA$-argument taken 
as part of the point set as in approach (2) in Subsection 
\ref{S:Aronszajn} above, and the BBLS-completely positive kernels 
appear prominently in the characterization of the generalized Schur 
classes and associated generalized Nevanlinna-Pick interpolation 
theory of Muhly-Solel (see \cite{MS2004, MS2008, MS2013}).

\subsection{Smoothness properties}  \label{S:smooth}

We now explore the extent to which smoothness on the 
kernel leads to smoothness for the functions in the associated 
reproducing kernel Hilbert space $\cH(K)$ and in the factor $H(Z)$ in 
the Kolmogorov decomposition for $K$.

Throughout this subsection we assume that $\Omega$ is a nc subset of 
$\cV_{nc}$ for a vector space $\cV$.
We consider three possible topologies for $\Omega$, the \textbf{finite 
topology}, the  \textbf{disjoint union topology}, and the \textbf{uniform 
topology} (see \cite{KVV-book}), described as follows.

\begin{itemize}
    \item We say that $\Omega$ is open in the \textbf{finite topology} if, 
given $W \in \Omega_{m}$ and $H_{W} \in \cV_{m}$, there exists 
$\epsilon > 0$ so that $W + t H_{W} \in \Omega_{m}$ for all $t \in 
{\mathbb C}$ with $|t| < \epsilon$.

\item Assume that $\cV$ is a Banach space equipped with an admissible 
family of matrix norms (a generalization of operator space defined in 
Section 7.1 of \cite{KVV-book}) and that $\Omega$ is an open subset 
of $\cV_{\rm nc}$.  We say that $\Omega$ is open in the \textbf{disjoint union 
topology} if, given $W \in \Omega_{m}$ there exists $\epsilon > 0$ so 
that $W + H_{W} \in \Omega_{m}$ as long as $H_{W} \in \cV_{m}$ has 
$\| H_{W} \| < \epsilon$.

\item Assume that $\cV$ and $\Omega$ are as in the immediately 
preceding definition.  We say that $\Omega$ is open in the \textbf{uniform topology} 
if, given $W \in \Omega_{m}$ there is $\epsilon > 0$ so that, for all 
$N \in {\mathbb N}$ and $D_{W}^{(N)} \in \cV^{mN \times mN}$ with $\| 
D^{(N)}_{W} \| < \epsilon $, it 
holds that $\bigoplus_{1}^{N}W + D_{W}^{(N)} \in \Omega_{nN}$.
\end{itemize}

Suppose now that $K$ is a nc cp kernel in $\widetilde \cT^{1}(\Omega; 
\cA_{\rm nc}, \cL(\cY)_{\rm nc})$.  We have three 
notions of local boundedness for $K$ (with fixed $\cA$-argument $P$) 
depending on the choice of topology on $\Omega$.

\begin{itemize}
    \item Assume that $\Omega$ is finitely open.  We say that the nc 
    cp kernel $K$ is \textbf{$P$-locally bounded along slices} if, 
    for each choice of points $Z \in \Omega_{n}$ and $W \in 
    \Omega_{m}$, $\cA$-matrix $P \in \cA^{ n \times m}$ and direction vectors 
$D_{Z} \in \cV^{n \times n}$ and $D_{W} \in  \cV^{m \times m}$, 
there is an $\epsilon > 0$ so that, for $t \in {\mathbb C}$ with $|t| < \epsilon$
we have not only $Z + t D_{Z} \in \Omega_{n}$ and $W + t D_{W} \in \Omega_{m}$ but also
$\| K(Z + t D_{Z}, W + t D_{W})(P)\|$ is bounded for $t \in {\mathbb C}$ with $|t| < 
    \epsilon$.

\item Assume that $\Omega$ is open in the disjoint union topology.  We 
say that the nc cp kernel $K$ is \textbf{$P$-locally bounded} if, 
given points $Z \in \Omega_{n}$ and $W \in \Omega_{m}$ and 
$\cA$-matrix $P \in \cA^{n \times m}$, there is an $\epsilon > 0$ so 
that $\|K(Z + D_{Z}, W + D_{W})(P)\|$ is uniformly bounded over all
direction vectors $D_{Z} \in \cV^{n \times n}$ and $D_{W} \in \cV^{m 
\times m}$ satisfying $\|D_{Z}\| < \epsilon$ and $\| D_{W} \| < \epsilon$.

\item Assume that $\Omega$ is open in the uniform topology.  We say 
that the nc cp kernel $K$ is \textbf{uniformly $P$-locally bounded} 
if, given points $Z \in \Omega_{n}$ and $W \in \Omega_{m}$ and 
$\cA$-matrix $P \in \cA^{n \times m}$, there is an $\epsilon > 0$ so 
that not only are $\bigoplus_{1}^{N} Z + D^{(N)}_{Z} \in \Omega_{Nn}$ and 
$\bigoplus_{1}^{N} W + D^{(N)}_{W} \in \Omega_{Nm}$ but also 
$\| K(\bigoplus_{1}^{N} Z + D^{(N)}_{Z}, \bigoplus_{1}^{N} W + 
D^{(N)}_{W}) ( \bigoplus_{1}^{N} P) \|$ is  uniformly bounded (independently 
of the choice of $N \in {\mathbb N}$) over all direction 
vectors $D^{(N)}_{Z} \in \cV_{Nn \times Nn}$ and $D^{(N)}_{W} \in 
\cV_{Nm \times Nm}$ 
satisfying $\| D^{(N)}_{Z} \| < \epsilon$ and $\| D^{(N)}_{W} \| < 
\epsilon$.  
\end{itemize}

Similar definitions apply to nc functions.  If $f$ is a nc function 
on nc set $\Omega$ which is open in the finite topology, we say that 
$f$ is \textbf{locally bounded on slices} if, given a point $Z \in 
\Omega_{n}$ and direction vector $D_{W} \in \cV^{n \times n}$, there 
is an $\epsilon > 0$ so that not only is $Z + t D_{Z} \in \Omega_{n}$ 
but also $\|f(Z + t D_{Z})\|$ is bounded for $t \in {\mathbb C}$ with $|t| < \epsilon$.  
Similarly, if $\Omega$ is open in the disjoint union topology and $f$ is a nc function on 
$\Omega$, we say that $f$ is \textbf{locally bounded} if, given a 
point $Z \in \Omega_{n}$ there is a $\epsilon > 0$ so that not only is
$Z + D_{Z} \in \Omega_{n}$ but also $\|f(Z + D_{Z})\| < \epsilon$ is bounded over all direction vectors 
$D_{Z} \in \cV^{n \times n}$ satisfying $\| D_{Z} \| < \epsilon$.  
Finally, if $\Omega$ is open in the uniform 
topology and $f$ is a nc function on $\Omega$, we say that $f$ is 
\textbf{uniformly locally bounded} if, given any $Z \in \Omega_{n}$,
there is an $\epsilon > 0$ so that not only is $\bigoplus_{1}^{N} Z + 
D^{(N)}_{Z} \in \Omega_{Nn}$ but also
$ \| f(\bigoplus_{1}^{N} Z + D^{(N)}_{Z})\|$ is bounded (independently 
of the choice of $N \in {\mathbb N}$) over all direction vectors 
$D^{(N)}_{Z} \in \cV^{nN \times nN}$ satisfying $\| D^{(N)}_{Z} \| < 
\epsilon$.  The significance of these various local boundedness 
conditions is that they imply corresponding analyticity properties 
for the nc function $f$: (1) if $f$ is locally 
bounded along slices,  then $f$ is G\^{a}teaux differentiable at 
each point $Z \in \Omega$ (see Theorem 7.2 in \cite{KVV-book}), (2) if 
$f$ is locally bounded, then $f$ is Frechet differentiable at each $Z 
\in \Omega$ (see Theorem 7.4 in \cite{KVV-book}), and (3) if $f$ is uniformly locally bounded, 
then $f$ is what is called ``uniformly analytic'' which in turn implies particularly 
nice convergence properties for its local Taylor-Taylor series (see 
Theorem 7.21 in \cite{KVV-book}).

Assume that $\Omega$ is finitely open and that the cp nc kernel $K$ is $P$-locally bounded
along slices. Then one can use the 
property of invariance with respect to direct sums to see that it 
suffices to assume that the $P$-local boundedness property holds with 
$Z = W \in \Omega_{n}$; indeed note that
$$ K\left( \sbm{ Z & 0 \\ 0 & W}, \sbm{ Z & 0 \\ 0 & W}\right) \left( 
\sbm{ P_{11} & P_{12} \\ P_{21} & P_{22} } \right) =
\begin{bmatrix} K(Z,Z)(P_{11}) & K(Z,W)(P_{12}) \\ K(W,Z)(P_{21}) & 
    K(W,W)(P_{22}) \end{bmatrix}
$$
If $K$ is a cp nc kernel, then $K\left( \sbm{ Z & 0 \\ 0 & W}, \sbm{ 
Z & 0 \\ 0 & W} \right)$ is a completely positive map from 
$\cA^{(n+m) \times (n+m)}$ to $\cL(\cY^{n+m})$.  This implies that 
$K(Z,W)$ is even a completely bounded map.  Thus, in working with the 
$P$-locally bounded condition, we may restrict to the case where 
$Z=W$.  It also suffices to work with $P$ equal to the identity 
matrix $P = I_{\cA^{n \times n}}$, since 
$\|K(Z,Z)\| = \| K(Z,Z)(1_{\cA^{n \times n}})\|$.  An analogous comment applies to the 
case where $K$ is locally $P$-bounded or uniformly $P$-locally 
bounded.

\begin{theorem}  \label{T:Gateaux}
    Let $\Omega \subset \cV_{nc}$ be a finitely open nc set.
    \begin{enumerate}
	\item Let $K$ be a cp nc kernel on $\Omega$ with values in 
	$\cL(\cA, \cL(\cY))_{\rm nc}$ that is $I$-locally bounded on slices.
	Let $(\cH(K), \sigma)$ be the corresponding nc reproducing 
	kernel Hilbert space and let $H$ be the factor in the 
	corresponding minimal Kolmogorov decomposition.  Then each $f 
	\in \cH(K)$ as well as $H$ are locally bounded on slices, 
	even when the completely bounded norm is used in the target 
	domain.
	More precisely, let $n \in {\mathbb N}$, $Z \in \Omega_{n}$, 
	$D_{Z} \in \cV^{n \times n}$ and assume that $K(Z + t 
	D_{Z}, Z + t D_{Z})(1_{\cA^{n \times n}})$ is bounded for $|t| < \epsilon$; then for each 
	$f \in \cH(K)$,  $\|f(Z+t D_{Z})\|$ and $\|f(Z + t 
	D_{Z})\|_{cb}$ as well as  $\|H(Z + t D_{Z})\|$ and 
	$\|H(Z + t D_{Z})\|_{\rm cb}$ are bounded for $|t| < \epsilon$.
	
	\item Conversely, let $\cH$ be a nc functional Hilbert space 
	on $\Omega$ with values in $\cL(\cA, \cY)_{\rm nc}$ with an 
	$\cA$-action.  Assume that for all $n \in {\mathbb N}$, $Z 
	\in \Omega_{n}$, $D_{Z} \in \cV^{n \times n}$ there exists an 
	$\epsilon > 0$ such that $\| f(Z + t D_{Z}) \|$ is bounded 
	for $|t| < \epsilon$ for all $f \in \cH_{K}$ (in this case we 
	say that the nc functional Hilbert space $\cH$ is {\rm \textbf{locally 
	bounded on slices}}).  Then the corresponding cp nc kernel $K$ 
	is $P$-locally bounded on slices, even when the completely 
	bounded norm is used in the target domain; more precisely, in the 
	notation of the previous sentence, $\| K(Z + t H_{Z}, Z + t 
	H_{Z})(1_{\cA^{n \times n}}) \| = \| K(Z + t D_{Z}, Z + t D_{Z})\|_{cb}$ is bounded for $|t| < \epsilon$.
\end{enumerate}
\end{theorem}

\begin{proof}
    Assume first that $K$ is a cp nc kernel on $\Omega$ and $f \in 
    \cH(K)$.  For a  point $W \in 
    \Omega_{m}$, direction vector $D_{W} \in \cV^{m \times m}$, row 
    $\cA$-tuple $v \in \cA^{m}$, and vector $y \in \cY^{m}$, by 
    Proposition \ref{P:cbRKHS} (specifically the estimate 
    \eqref{cbnormf(W)}), we have
 \begin{align*}
& \| f(W + t D_{W}) \|_{\cL(\cA^{m}, \cY^{m})} =
 \| f(W + t D_{W}) \|_{\cL_{\rm cb}(\cA^{m}, \cY^{m})} \\
& \quad  \le \| f \|_{\cH(K)} 
 \| K(W + t D_{W}, W + t D_{W})(1_{\cA^{m \times 
 m}})\|_{\cL(\cY^{m})}^{1/2}.
\end{align*}
The assumption that $K$ is $P$-locally bounded then leads to the 
conclusion that $f$ is locally bounded in cb-norm with  radius 
of tolerance $\epsilon$ equal to the radius of tolerance for $K$
independent of the choice of $f \in \cH(K)$. 

We next analyze the boundedness along slices for the minimal factor 
$H$ in the Kolmogorov decomposition for $K$. Again by Proposition 
\ref{P:cbRKHS} (specifically for this case the equality 
\eqref{normH(Z)}), we have
$$
\| H(W + t D_{W} \| = \| H(W + t D_{W} \|_{\rm cb} = \| K(W + tD_{W}, 
W + t D_{W})(1_{\cA^{m \times m}}) \|^{1/2}.
$$
We conclude that $H$ is locally bounded in cb-norm along slices as 
long as $K$ is $P$-locally bounded along slices with radius  of 
tolerance $\epsilon$ the same as for $K$.
This completes the proof of statement (1) in 
Theorem \ref{T:Gateaux}.

We next suppose that $\cH$ is a nc functional Hilbert space with 
$\cA$-action (as in Theorem \ref{T:RKHS}) which furthermore is 
bounded along slices.  More precisely, this means that for each $n 
\in {\mathbb N}$, $Z \in \Omega_{n}$, $D_{Z} \in \cV^{n \times n}$, 
there exists $\epsilon > 0$ such that $\|f(Z + t D_{Z})\|$ is bounded 
for $|t| < \epsilon$.  We emphasize that the assumption here is that 
the radius of tolerance $\epsilon$ depends  on $Z$ and $D_{Z}$ but not 
on the choice of $f$ in $\cH(K)$.  Consider the family of linear 
operators 
$$
\{ H(Z + t D_{Z}) \colon \cH(K)^{n} \to \cY^{n} \colon t 
\in {\mathbb C} \text{ with } |t| < \epsilon\}.
$$
Note that when we apply any operator in this family to a vector $f = 
\sbm{ f_{1} \\ \vdots \\ f_{n}}$ in $\cH(K)^{n}$ we get
\begin{align*}
& H(Z + t D_{Z}) \sbm{f_{1} \\ \vdots \\ f_{n} }
= \sbm{ f_{1} \\ \vdots \\ f_{n} } (Z  + t D_{Z})(1_{\cA^{n \times n}})  \\
& \quad = f_{1}(Z + t D_{Z})(E^{(n)}_{1} \otimes 1_{\cA}) + \cdots + 
    f_{n}(Z + t D_{Z})(E^{(n)}_{n} \otimes 1_{\cA}).
\end{align*}
As each term in the final expression is bounded for $|t| < \epsilon$, 
it follows that $\|H(Z + t D_{Z}) f\|$ is bounded for $|t| < 
\epsilon$ for each $f \in \cH(K)^{n}$.  It is now a consequence of 
the Principle of Uniform Boundedness (see e.g. \cite{Rudin}) that in 
fact $\|H(Z + t D_{Z}) \|$ (and hence also $\|H(Z + t D_{Z} \|_{\rm 
cb}$ by \eqref{normH(Z)}) is bounded for $|t| < \epsilon$.  This 
completes the proof of statement (2) in the Theorem.
\end{proof}

There are results analogous to that of Theorem \ref{T:Gateaux} when 
we replace the hypothesis that $\Omega$ is open in the finite 
topology by the hypothesis that $\Omega$ is open in the disjoint 
union or in the uniform topology, and replace the ``locally bounded 
on slices'' condition by ``locally bounded'' or by ``uniformly 
locally bounded'' respectively.  For the record we state these 
results.

\begin{theorem}   \label{T:queen}
 Let $\Omega \subset \cV_{\rm nc}$ be a nc set which is open in the 
 disjoint union topology.
    \begin{enumerate}
	\item Let $K$ be a cp nc kernel on $\Omega$ with values in 
	$\cL(\cA, \cL(\cY))_{\rm nc}$ that is $P$-locally bounded.
	Let $(\cH(K), \sigma)$ be the corresponding nc reproducing 
	kernel Hilbert space and let $H$ be the factor in the 
	corresponding minimal Kolmogorov decomposition.  Then each $f 
	\in \cH(K)$ as well as $H$ are both locally bounded and 
	locally completely bounded.
	More precisely, let $n \in {\mathbb N}$, $Z \in \Omega_{n}$, 
        and assume that $K(Z +  D_{Z}, 
	Z +  D_{Z})(1_{\cA^{n \times n}})$ is defined and bounded for all $D_{Z} \in \cV^{n \times 
	n}$ with $\|D_{Z}\| < \epsilon$; then for each 
	$f \in \cH(K)$,  $\|f(Z+ D_{Z})\|$ and $\|f(Z+ D_{Z})\|_{\rm 
	cb}$ as well as  $\|H(Z + D_{Z})\|$ and $\|H(Z 
	+  D_{Z})\|_{\rm cb}$ are bounded for $\|D_{Z}\| < \epsilon$.
	
	\item Conversely, let $\cH$ be a nc functional Hilbert space 
	on $\Omega$ with values in $\cL(\cA, \cY)_{\rm nc}$ with an 
	$\cA$-action.  Assume that for all $n \in {\mathbb N}$, $Z 
	\in \Omega_{n}$, there exists an 
	$\epsilon > 0$ such that $\| f(Z +  D_{Z}) \|$ is defined and bounded 
	for all $D_{Z} \in \cV^{n \times n}$ such that $\|D_{Z}\| < 
	\epsilon$ and for all $f \in \cH$ (in this case we 
	say that the nc functional Hilbert space $\cH$ is {\rm \textbf{locally 
	bounded}}).  Then the corresponding cp nc kernel $K$ 
	associated with $\cH$ by Theorem \ref{T:RKHS}
	is $P$-locally bounded and $P$-locally completely bounded; more precisely, 
	$\|K(Z +  D_{Z}, Z + D_{Z})(1_{\cA^{n \times n}}) \| = \|K(Z 
	+  D_{Z}, Z +  D_{Z})(1_{\cA^{n \times n}}) \|_{\rm cb}$ 
	is defined and  bounded for all $D_{Z} \in \cV^{n \times n}$  with $\|D_{Z}\| < \epsilon$.
\end{enumerate}
\end{theorem}

\begin{theorem}   \label{T:king}
     Let $\Omega \subset \cV_{nc}$ be a nc set which is open in the 
 uniform topology.
    \begin{enumerate}
	\item Let $K$ be a cp nc kernel on $\Omega$ with values in 
	$\cL(\cA, \cY)_{\rm nc}$ that is uniformly $P$-locally bounded.
	Let $(\cH(K), \sigma)$ be the corresponding nc reproducing 
	kernel Hilbert space and let $H$ be the factor in the 
	corresponding minimal Kolmogorov decomposition.  Then each $f 
	\in \cH(K)$ as well as $H$ are uniformly locally bounded.
	More precisely, let $Z \in \Omega_{n}$, 
        and assume that $K( \bigoplus_{1}^{N}Z +  D^{(N)}_{Z}, 
 	\bigoplus_{1}^{N} Z + D^{(N)}_{Z})(1_{\cA^{n \times n}})$ is defined and bounded for all 
	$D^{(N)}_{Z} \in \cV^{nN \times 
	nN}$ with $\|D_{Z}\| < \epsilon$ for all $N \in {\mathbb N}$; 
	then for each $f \in \cH(K)$,  $\|f(\bigoplus_{1}^{N} Z + D^{(N)}_{Z})\|$ 
	as well as  $\|H( \bigoplus_{1}^{N} Z +  
	D^{(N)}_{Z})\|$ are bounded for all $D_{Z} \in \cV^{nN \times 
	nN}$ with $\|D_{Z}\| < \epsilon$ for all $N \in {\mathbb N}$.
	
	\item Conversely, let $\cH$ be a nc functional Hilbert space 
	on $\Omega$ with values in $\cL(\cA, \cY)_{\rm nc}$ with an 
	$\cA$-action.  Assume that for each $Z \in \Omega_{n}$, there exists an 
	$\epsilon > 0$ such that $\| f(\bigoplus_{1}^{N}Z +  D^{(N)}_{Z}) \|$ is bounded 
	for all $D_{Z} \in \cV^{n \times n}$ such that $\|D_{Z}\| < 
	\epsilon$ and for all $f \in \cH_{K}$ (in this case we 
	say that the nc functional Hilbert space $\cH$ is {\rm 
	\textbf{uniformly locally 
	bounded}}).  Then the corresponding cp nc kernel $K$ 
	associated with $\cH$ as in Theorem \ref{T:RKHS}
	 is uniformly $P$-locally bounded; more precisely, $\| K(\bigoplus_{1}^{N} Z +  
	D^{(N)}_{Z}, \bigoplus_{1}^{N} Z +  
	D^{(N)}_{Z})(1_{\cA^{n \times n}}) \|$ is defined and bounded for all $D_{Z} \in \cV^{nN 
	\times nN}$ with $\|D_{Z}\| < \epsilon$ for all $N \in 
	{\mathbb N}$.
\end{enumerate}
\end{theorem}

\begin{proof}[Proof of Theorems \ref{T:queen} and \ref{T:king}]
 As the proofs of both theorems parallel closely the proof of 
 Theorem \ref{T:Gateaux}, we only sketch the main ideas.  The 
 hypothesis and conclusion of Theorem \ref{T:Gateaux}, for both 
 statements (1) and (2), involves the boundedness of a family of 
 operators parametrized by $t \in {\mathbb C}$ with $|t| < \epsilon$.
 The hypothesis and conclusion of Theorems \ref{T:queen} and 
 \ref{T:king} are the same, but with the modification that the family 
 of operators is parametrized by increment vectors $D_{Z} \in \cV^{n 
 \times n}$ with $\| D_{Z} \| < \epsilon$, or by increment vectors 
 $D^{(N)}_{Z} \in \cV^{nN \times nN}$ with $N \in {\mathbb N}$ 
 arbitrary with $\| D^{(N)}_{Z} \| < \epsilon$.  With this 
 modification of the set of operators whose boundedness is of 
 interest, the proofs of Theorems \ref{T:queen} and \ref{T:king} go 
 through in exactly the same way as in the proof of Theorem 
 \ref{T:Gateaux}.
 \end{proof}

\subsection{Functional versus formal noncommutative reproducing kernel 
Hilbert spaces}  \label{S:NFRKHS}

The goal of this subsection is to establish a dictionary between the 
global/nc functional reproducing kernel Hilbert spaces being discussed here and 
the notion of formal nc reproducing kernel Hilbert spaces introduced by 
two of the present authors in \cite{NFRKHS}.  Toward this end, we 
first need to review the setup from \cite{NFRKHS}.

We let $\free$ be the monoid on $d$ generators $\{1, \dots, d\}$ 
(also known as the unital free semigroup 
with $d$ generators).  Elements of $\free$ are written as words 
$\fa = i_{N} \cdots i_{1}$ with letters $i_{j}$ from the alphabet 
consisting of the first $d$ natural numbers $\{1, \dots, d\}$. 
Multiplication is by concatenation:
$$
  \fa \cdot \fb = i_{N} \cdots i_{1} j_{M} \cdots j_{1} \text{ if }
  \fa = i_{N} \cdots i_{1} \text{ and } \fb = j_{M} \cdots j_{1}.
$$
The empty word, denoted as $\emptyset$, serves as the unit element for 
$\free$.  For $\fa = i_{N} \cdots i_{1}$ an element of $\free$, we 
let $\fa^{\top} = i_{1} \cdots i_{N}$ denote the 
\textbf{transpose} of $\fa$ and let $|\fa| = N$ denote the 
length of (or number of letters in) $\fa$.

Given a collection of freely noncommuting indeterminates $z = (z_{1}, 
\dots, z_{d})$ and given a word $\fa = i_{N} \cdots i_{1}$ in 
$\free$, we let $z^{\fa}$ denote the noncommutative monomial
$z^{\fa} = z_{i_{N}} \cdots z_{i_{1}}$, where we take 
$z^{\emptyset} = 1$.   Given also a linear space $\cX$, 
we let $\cX\langle \langle z 
\rangle \rangle$ denote the space of all formal power series 
$f(z) = \sum_{\fb \in \free} f_{\fa} z^{\fa}$ where the 
coefficients $f_{\fa}$ are in $\cX$. Suppose that $\cX'$ is an algebra such 
that $\cX$ is a left module over $\cX'$.  Given
$$
F(z) = \sum_{\fa \in \free} F_{\fa} z^{\fa} \in \cX' \langle 
\langle z \rangle \rangle, \quad f(z) = \sum_{\fb \in \free} 
f_{\fb} z^{\fb} \in \cX\langle \langle z \rangle \rangle,
$$
we define the (noncommutative convolution) product $F \cdot f(z) \in 
\cX\langle \langle z \rangle \rangle$ by
$$
  (F \cdot f)(z) = \sum_{\gamma \in \free}\left( \sum_{\fa, \fb \in 
  \free \colon \gamma = \fa \fb} F_{\fa} f_{\fb} \right) 
  z^{\gamma}.
$$

We now suppose that we are given a Hilbert space $\cH$ 
whose elements $f(z)$ are formal power series $f(z) = \sum_{\fa 
\in \free} f_{\fa} z^{\fa}$ in $\cY \langle \langle z \rangle 
\rangle$ for a coefficient Hilbert space $\cY$. We say that $\cH$ is a 
NFRKHS ({\em noncommutative formal reproducing kernel Hilbert space})
if, for each $\fa \in \free$, the linear operator 
$\boldsymbol{\rm ev}_{\fa} \colon f \mapsto f_{\fa}$ mapping $f$ to its 
$\fa$-th formal power series coefficient in $\cY$ is continuous.  
As any such power series is completely determined by the list of its 
coefficients $\{f_{\fa} \colon \fa \in \free\}$,  equivalently 
we can view an element $f(z)$ as a function $\fa \mapsto 
f_{\fa}$ on $\free$.  Hence, by the Aronszajn theory of 
reproducing kernel Hilbert spaces (see Subsection \ref{S:Aronszajn} 
above),  there is an Aronszajn-type positive kernel $\bK \colon \free 
\times \free \to \cL(\cY)$ so that $\cH$ is the reproducing kernel 
Hilbert space associated with $\bK$.  To spell this out for this 
context, we denote the value of $\bK$ at the pair of words $(\fa, 
\fb)$ by $\bK_{\fa, \fb}$. Since we view an element $f \in \cH$ 
as a formal power series $\sum_{\fa \in \free} f_{\fa} z^{\fa}$ rather than as a 
function $\fa \mapsto f_{\fa}$, we write, for a given $\fb 
\in \free$ and $y \in \cY$, the element $\boldsymbol{\rm ev}_{\fb}^{*} y 
\in \cH$ as 
$$
  \left(\boldsymbol{\boldsymbol{\rm ev}}_{\fb}^{*}\, y \right)(z) = :  \bK_{\fb}(z) y = 
  \sum_{\fa \in \free} \bK_{\fa, \fb} y \, z^{\fa}.
$$
Then the reproducing kernel property assumes the form
\begin{equation}  \label{formal-reprod}
   \langle f,\, \bK_{\fb}(\cdot) y \rangle_{\cH} = \langle 
   f_{\fb}, y \rangle_{\cY}.
\end{equation}

Following \cite{NFRKHS}, we make the notation more suggestive of 
the classical case as follows.  We let $w = (w_{1}, \dots, w_{d})$ be 
a second $d$-tuple of freely noncommuting indeterminates.  
For suggestive formal reasons which will become clear below, we also introduce 
the conjugate $d$-tuple of freely 
noncommuting indeterminates
$$
\overline{w} = (\overline{w}_{1}, \dots, \overline{w}_{d})
$$
In general, given a coefficient Hilbert space $\cC$, we can use the 
$\cC$-inner product to define a pairing
$$
   \langle \cdot, \, \cdot \rangle_{\cC \times \cC\langle \langle 
   \overline{w} 
   \rangle \rangle} \to {\mathbb C} \langle \langle w \rangle \rangle
$$
by 
$$
  \left \langle c, \sum_{\fa \in \free} f_{\fa} \overline{w}^{\fa} 
  \right\rangle_{\cC \times \cC\langle \langle \overline{w} \rangle \rangle} =
\sum_{\fa \in \free} \langle c, f_{\fa}\rangle_{\cC} 
w^{\fa^{\top}}.
$$
We shall also have use of the reverse-order version of the pairing:
 \footnote{These can be seen as special cases of the more general pairing 
 (which we shall not need in the sequel)
$$
\left\langle \sum_{\fa \in \free} f_{\fa} w^{\fa}, \,
\sum_{\fb \in \free} g_{\fb} \overline{w}^{\fb} 
\right\rangle_{\cC\langle \langle w \rangle \rangle \times \cC\langle 
\langle \overline{w} \rangle \rangle}
= \sum_{\gamma \in \free}  \left[ \sum_{\fa, \fb \colon \gamma 
= \fb^{\top} \fa } \langle f_{\fa}, g_{\fb} \rangle_{\cC} 
\right] w^{\gamma},
$$
or equivalently
$$   
  \langle f(w), g(w) \rangle_{\cC \langle \langle w \rangle \rangle 
  \times \cC \langle \langle \overline{w} \rangle \rangle} =
  g(w)^{*} f(w) \in {\mathbb C}\langle \langle w \rangle \rangle
$$
where we set $\left(\sum_{\fa} g_{\fa} 
\overline{w}^{\fa}\right)^{*} = \sum_{\fa} g_{\fa}^{*} 
w^{\fa^{\top}} \in \cL(\cC, {\mathbb C})\langle \langle w \rangle 
\rangle$ where $g_{\fa}^{*} \in \cL(\cC, {\mathbb C})$ is given by
$g_{\fa}^{*} \colon c \mapsto \langle c, g_{\fa} \rangle_{\cC}$.
Using this formalism one can easily check the identity
$$  \langle S(w) f(w), g(w) \rangle_{\cY\langle \langle w \rangle 
\rangle \times \cY \langle \langle \overline{w} \rangle \rangle} =
\langle f(w), S(w)^{*} g(w) \rangle_{\cU\langle \langle w \rangle 
\rangle \times \cU\langle \langle \overline{w} \rangle \rangle}
$$
for $S(w) \in \cL(\cU,  \cY)\langle \langle w \rangle \rangle$,
$f(w) \in \cU\langle \langle w \rangle \rangle$ and $g(w) \in 
\cY\langle \langle \overline{w} \rangle \rangle$.}
$$
 \left\langle \sum_{\fa \in \free} f_{\fa} w^{\fa}, c 
 \right\rangle_{\cC\langle \langle z \rangle \rangle \times \cC} = 
 \sum_{\fa \in \free} \langle f_{\fa}, c \rangle_{\cC} 
 w^{\fa}.
 $$
Then the reproducing kernel property \eqref{formal-reprod} can be 
written more suggestively as
\begin{equation}   \label{formal-reprod'}
    \langle f, \bK(\cdot, w) y \rangle_{\cH \times \cH \langle \langle 
    \overline{w} \rangle \rangle} = \langle f(w), y 
    \rangle_{\cY\langle \langle w \rangle \rangle \times \cY}.
\end{equation}
Here $\bK(z,w)$ has the property that, for each $y \in \cY$, the formal power series in 
$\overline{w}$ given by
$$
\bK(z, w) y = \sum_{\fa, \fb \in \free} (\bK_{\fa, \fb} y)
\, z^{\fa} \overline{w}^{\fb^{\top}} = \sum_{\fb \in \free} \left[ 
\sum_{\fa \in \free} ( \bK_{\fa, \fb^{\top}} y) z^{\fa} 
\right] \overline{w}^{\fb}
$$
can be viewed as an element of $\cH\langle \langle \overline{w} 
\rangle \rangle$.  Whenever $\cH$ is a Hilbert space of formal power 
series with the structure as laid out above, we shall say that $\cH$ 
is the \textbf{noncommutative formal reproducing kernel Hilbert space} 
(NFRKHS) \textbf{with reproducing kernel} $\bK \in \cL(\cY)\langle \langle z, 
\overline{w} \rangle \rangle$.

The following result amounts to Theorem 3.1 from \cite{NFRKHS}.

\begin{theorem} \label{T:NFRKHS}
Let
$
\bK(z, w) = \sum_{\fa, \fb \in \free} \bK_{\fa, \fb} z^{\fa} 
\overline{w}^{\fb^{\top}}$ be a given 
element of
${\mathcal L}({\mathcal Y})\langle \langle z, \overline{w} \rangle 
\rangle$ where $z=(z_{1}, \dots, z_{d})$ 
and $w = (w_{1}, \dots, w_{d})$ are $d$-tuples of freely noncommuting 
indeterminates
Then the following conditions are equivalent: 
\begin{enumerate}
    
\item $\bK(z, w)$ is a {\em positive formal kernel} in the sense that
$$\sum_{\fa, \fb \in \free} \langle \bK_{\fa, \fb} 
y_{\fa}, y_{\fb} 
\rangle_{{\mathcal Y}} \ge 0 $$
for all finitely supported ${\mathcal Y}$-valued functions $ \fa 
\mapsto y_{\fa}$ on $\free$.
    
\item $\bK$ is the reproducing kernel for a uniquely determined NFRKHS 
${\mathcal H}(\bK)$ of formal power series in the set of noncommuting 
indeterminates $z=(z_{1}, \dots, z_{d})$.

\item There is an auxiliary 
Hilbert space $\cX$ and a noncommutative formal power series 
$\bH(z) \in {\mathcal L}(\cX, \cY)\langle \langle z \rangle \rangle$ such that 
\begin{equation} \label{formalKoldecom}
\bK(z, w) = \bH(z) \bH(w)^{*}
\end{equation}
where we use the convention $ (w^{\fb})^{*} = \overline{w}^{\fb^{\top}}$
so that 
$$  
\bH(w)^{*} = \sum_{\fb \in \free} (\bH_{\fb^{\top}})^{*} 
\overline{w}^{\fb} \text{ if }\bH(z) = \sum_{\fa \in \free} \bH_{\fa} 
z^{\fa}.
$$
\end{enumerate}
Moreover, in this case the NFRKHS ${\mathcal H}(\bK)$ can be defined 
directly in terms of the formal power series $\bH(z)$ appearing in 
condition (2) by
$$ {\mathcal H}(\bK) = \{ \bH(z) x \colon x \in {\mathcal X} \} $$
with norm taken to be the ``lifted norm'' 
\begin{equation}    \label{formalLiftedNorm}
    \| \bH(z) x \|_{{\mathcal H}(\bK)} = \| Q x \|_{{\mathcal H}} 
\end{equation}
where $Q$ is the orthogonal projection of ${\mathcal X}$ onto the 
orthogonal complement of the kernel of the map 
$M_{\bH} \colon {\mathcal X} \mapsto {\mathcal Y}\langle \langle z 
\rangle \rangle$ given by $M_{\bH}\colon x \mapsto 
\bH(z) x$.
\end{theorem}

We now suppose that we are given a formal positive kernel $\bK$ and the 
associated NFRKHS $\cH(\bK)$ as in Theorem \ref{T:NFRKHS}. We wish 
to make the connection with nc reproducing kernel Hilbert spaces 
by evaluating formal power series $f(z) \in \cH(K)$ at matrix tuple 
points $Z = (Z_{1}, \dots, Z_{d}) \in ({\mathbb C}^{n \times n})^{d} 
\cong ({\mathbb C}^{d})^{n \times n}$.  We therefore introduce the nc 
set $\Omega  \subset ({\mathbb C}^{d})_{\rm nc}$ by
\begin{align}  
\Omega = & \{ Z \in ({\mathbb C}^{d})_{\rm nc} \colon \sum_{\ell = 
0}^{\infty}  \sum_{ \fa \in \free \colon |\fa| = \ell} 
 Z^{\fa} \otimes f_{\fa} \text{ converges}   \notag \\
 &  \text{ for all } f(z) = \sum_{\fa \in \free} f_{\fa} 
 z^{\fa} \in \cH(\bK) \}.
 \label{conv-set} 
\end{align}
Here the convergence is taken in the weak topology of $\cY^{m \times m}$ if $Z \in 
({\mathbb C}^{d})^{m \times m}$.  From the lifted norm 
characterization \eqref{formalLiftedNorm} of $\cH(K)$, it is clear 
that $\Omega$ can alternatively be characterized as
\begin{equation}  \label{conv-set'}
 \Omega = \{ Z \in ({\mathbb C}^{d})_{\rm nc} \colon 
 \sum_{\ell=0}^{\infty} \sum_{\fa \in \free \colon |\fa| = 
 \ell} Z^{\fa} \otimes \bH_{\fa} \text{ converges }\}
\end{equation}
where, if $Z \in ({\mathbb C}^{d})^{m \times m}$, the convergence is  in  
$\cL(\cX^{m}, \cY^{m})$ 
with the weak operator topology.

The next result is the main tool for arriving at a nc reproducing 
kernel Hilbert space from a NFRKHS.

\begin{proposition}   \label{P:f-to-nc}  Suppose that $\bK$ is a formal 
    positive kernel with associated {\rm NFRKHS} $\cH(\bK)$.  Define 
    the nc set $\Omega$ as in either \eqref{conv-set} or 
    \eqref{conv-set'}, and suppose that $Z \in \Omega_{n}$.  Then the 
    iterated series
\begin{equation}   \label{series}
\sum_{s=0}^{\infty} \sum_{\fa \in \free \colon |\fa| = s}
\left[ \sum_{t=0}^{\infty} \sum_{\fb \in \free \colon |\fb| = t} 
Z^{\fa} P W^{* \fb^{\top}} \otimes \bK_{\fa, \fb} \right]
\end{equation}
 converges for all $W \in \Omega_{m}$ and $P \in {\mathbb C}^{n 
 \times m}$ for all $m=1,2,\dots$.  The same result holds if the 
 order of iteration is reversed.
 \end{proposition}
 
 \begin{proof}
     Given $Z \in \Omega_{n}$, $W \in \Omega_{m}$, and $P \in 
     {\mathbb C}^{n \times m}$, from the second characterization 
     \eqref{conv-set'} of $\Omega$ we see that $\sum_{s=0}^{\infty} 
     \sum_{\fa \in \free \colon |\fa| = s} Z^{\fa} \otimes 
     \bH_{\fa}$ converges and that 
     $$
    \left( \sum_{t=0}^{\infty} \sum_{\fb 
     \in \free \colon |\fb| = t} W^{\fb} \otimes 
     \bH_{\fb}\right)^{*} =
     \sum_{t=0}^{\infty} \sum_{\fb \in \free \colon |\fb| = t} 
     W^{* \fb^{\top}} \otimes \bH_{\fb}^{*}
     $$
 converges weakly, from which it follows that the iterated sum
 \begin{align*}
  & \sum_{s=0}^{\infty} \sum_{\fa \in \free \colon |\fa| = s}
\left[ \sum_{t=0}^{\infty} \sum_{\fb \in \free \colon |\fb| = t} 
Z^{\fa} P W^{* \fb^{\top}} \otimes \bK_{\fa \fb} \right] \\
& \quad =
\sum_{s=0}^{\infty} \sum_{\fa \in \free \colon |\fa| = s}
(Z^{\fa} \otimes \bH_{\fa}) P \left[ \sum_{t=0}^{\infty} 
\sum_{\fb \in \free \colon |\fb| = t}
  W^{*\fb^{\top}} \otimes \bH_{\fb}^{*}\right]\\
\end{align*}
converges in either order.
 \end{proof}
 
 Given the result of Proposition \ref{P:f-to-nc}, we can associate a cp 
 nc kernel $K$ with a given formal positive kernel $\bK$ as follows.  
 Suppose that $\bK$ has formal Kolmogorov decomposition $\bK(z,w) = 
 \bH(z) \bH(w)^{*}$ as in \eqref{formalKoldecom}. For $Z \in 
 \Omega_{n}$, we may use the convergent series in \eqref{conv-set'} 
 to define a function $Z \mapsto H(Z)$:
 $$
   H(Z) = \sum_{\ell = 0}^{\infty}
   \left[\sum_{\fa \in \free \colon |\fa| = \ell} \bH_{\fa} 
   \otimes Z^{\fa}\right]
   \in \cL(\cX^{n}, \cY^{n}).
 $$
 We then define a kernel function $K$ from $\Omega_{n} \times 
 \Omega_{m}$ to $\cL({\mathbb C}^{n \times m}, \cL(\cY)^{n \times 
 m})$ by
 \begin{equation}   \label{bK-to-K}
  K(Z,W)(P) = H(Z) (P \otimes {\rm id}_{\cL(\cX)}) H(W)^{*}.
 \end{equation}
 As $H$ is given in terms of a convergent tensor-calculus power 
 series, it follows that $H$ is a nc function from $\Omega$ to 
 $\cL(\cX, \cY)_{\rm nc}$ (see \cite{KVV-book}).
 Then $K$ is given in terms of a Kolmogorov decomposition (with $\cA 
 = {\mathbb C}$ and $\sigma$ the trivial representation of ${\mathbb C}$ 
 on ${\mathbb C}$), and hence is a cp nc kernel.  
 
 The following result establishes the precise correspondence between 
 the nc RKHS $\cH(K)$ and the formal nc RKHS $\cH(\bK)$ when $K$ and 
 $\bK$ are related as in \eqref{bK-to-K}.

 \begin{theorem}  \label{T:f-vs-nc}
 Suppose that $\bK$ is a formal positive kernel with formal 
 Kolmogorov decomposition as in \eqref{formalKoldecom}, and let $K$ be
 the function from $\Omega_{n} \times \Omega_{m}$ to $\cL({\mathbb 
 C}^{n \times m}, \cL(\cY)^{n \times m})$ given by \eqref{bK-to-K}.
 Then $K$ is a cp nc kernel.  Furthermore $\cH(K)$ is the isometric image of 
 $\cH(\bK)$ under the map 
 \begin{equation} \label{map:f-to-nonf}
      \sum_{\fa \in \free} f_{\fa} z^{\fa}   \mapsto
  \left( Z \in \Omega_{n} \mapsto\left( u \in {\mathbb C}^{n} \mapsto
  \sum_{\fa \in \free} Z^{\fa} u \otimes f_{\fa}\right)  
  \right).
 \end{equation}

 Conversely, suppose that $K$ from $\Omega_{n} \times \Omega_{m}$ to 
 $\cL({\mathbb C}^{n \times m}, \cL(\cY)^{n \times m})$ ($n , m \in 
 {\mathbb N}$ arbitrary) is a cp nc kernel defined on either (a) 
 the nc set $ \Omega$ taken to be ${\rm Nilp}^{d} = \amalg_{n=1}^{\infty}{\rm 
 Nilp}_{n}^{d}$ where 
 ${\rm Nilp}_{n}^{d}$ is defined to be the space of jointly nilpotent 
 $d$-tuples $Z = (Z_{1}, 
 \dots, Z_{d})$ of $n \times n$ matrices (so $Z^{\fa} = 0$ once 
 $|\fa|$ is sufficiently large), or (b) on the nc set $\Omega \subset 
 ({\mathbb C}^{d})_{\rm nc}$ where $\Omega$ is a finitely open set 
 containing $0$ on which $K$ is $I$-locally bounded along slices.
 Suppose that $K$ has nc Kolmogorov decomposition $K(Z,W)(P) = H(Z) 
 (P \otimes {\rm id}_{\cL(\cX)}) H(W)^{*}$ for a nc function $H \colon 
 \Omega \to \cL(\cX, \cY)_{\rm nc}$. (By the results of 
 \cite{KVV-book}, $H(Z)$ has a Taylor-Taylor series representation 
 $H(Z) = \sum_{\fa \in \free} Z^{\fa} \otimes H_{\fa}$ 
 either on all of ${\rm Nilp}^{d}$ in case (a), or on an appropriate 
 $\widetilde \Omega \subset \Omega$ in case (b)). Define a formal power 
 series $\bH(z) \in \cL(\cX, \cY) \langle \langle z \rangle \rangle$ 
 and a formal kernel $\bK \in \cY\langle \langle z, \overline{w} 
 \rangle \rangle$ by
 $$
  \bH(z) = \sum_{\fa \in \free} H_{\fa} z^{\fa}, \quad
  \bK(z,w) = \bH(z) \bH(w)^{*} = \sum_{\fa, \fb \in \free} 
  (H_{\fa} H_{\fb^{\top}}^{*}) z^{\fa} \overline{w}^{\fb}.
  $$
  Then $\bK$ is a formal positive kernel with convergence set 
  containing ${\rm Nilp}^{d}$ (in case (a)) or $\widetilde \Omega$ (in 
  case (b)), and we recover $\cH(\bK)$ from $\cH(K)$ via the inverse 
  of the map \eqref{map:f-to-nonf}. 
\end{theorem}

\begin{proof}
    As is seen from part (3) of Theorem \ref{T:NFRKHS}, the space 
    $\cH(\bK)$ can be presented directly in terms of the formal 
    Kolmogorov decomposition \eqref{formalKoldecom} of the formal 
    positive kernel $\bK$ as the space
    $$
    \cH(\bK) = \{\boldf_{x}(z) \colon x \in \cX\}
    $$
    where we set $\boldf_{x}(z) = \bH(z) x$, with norm given by
    $$ 
    \| \boldf \|^{2}_{\cH(\bK)} = \inf \{ \| x \|^{2} \colon x \in \cX 
    \text{ such that } \boldf = \boldf_{x}\}.
    $$
    On the other hand, by Theorem \ref{T:lifted}, the nc RKHS $\cH(K)$ has a similar lifted norm 
    description in terms of its nc Kolmogorov decomposition 
    \eqref{Koldecom}, namely:
    $$
    \cH(K) = \{f_{x} \colon x \in \cX\}
    $$
    where we set $f_{x}(Z) u = H(Z) ({\rm id}_{{\mathbb C}^{n}} 
    \otimes \sigma_{\cX}) (u) x$
    with norm given by
    $$
     \| f \|^{2}_{\cH(K)} = \inf \{ \| x \|^{2} \colon f = f_{x}\}.
     $$
  For the case at hand here where the cp nc kernel $K$ is 
  derived from a formal positive kernel $\bK$ as in \eqref{bK-to-K}, the 
  $C^{*}$-algebra $\cA$ is just ${\mathbb C}$ and the representation 
  $\sigma_{\cX}$ is the trivial representation.  Then the function 
  $Z \mapsto f_{x}(Z)$, acting on a vector $u \in {\mathbb C}^{n}$ 
  for $Z \in \Omega_{n}$ can be written more concretely as
  $$
  f_{x}(Z) u =  H(Z) (u \otimes x) = 
  \sum_{\fa \in \free} Z^{\fa} u \otimes H_{\fa} x.
  $$
  In terms of these respective parametrizations of the spaces 
  $\cH(\bK)$ and $\cH(K)$ via the state space $\cX$ for the 
  Kolmogorov decompositions, we see that the map \eqref{map:f-to-nonf}
  is given by
  \begin{equation}  \label{map'}
  \phi \colon  \sum_{\fa \in \free} (\bH_{\fa} x) z^{\fa} \mapsto
   (Z \in \Omega_{n} \mapsto (u \in {\mathbb C}^{n} \mapsto
    \sum_{\fa \in \free} Z^{\fa} u \otimes \bH_{\fa} x)).
  \end{equation}
  Given the lifted-norm-space characterizations of $\cH(\bK)$ and 
  $\cH(K)$, we see that indeed the map  $\phi$ \eqref{map'} maps 
  $\cH(\bK)$ onto $\cH(K)$.  To see that $\phi$ is well-defined, note 
  that $ \sum_{\fa \in \free} (\bH_{\fa} x) z^{\fa} = 0$ in 
  $\cH(\bK)$ means that $\bH_{\fa} x = 0$ for all $\fa \in \free$.
  It then follows that $\sum_{\fa \in \free} Z^{\fa} u \otimes \bH_{\fa} x
  = 0$ for all $Z \in \Omega_{n}$ and all $u \in {\mathbb C}^{n}$, 
  i.e., the right-hand side of \eqref{map'} is the zero element of 
  $\cH(K)$.  Conversely, suppose that 
  $\sum_{\fa \in \free} Z^{\fa} u \otimes \bH_{\fa} x
  = 0$ for all $Z \in \Omega_{n}$ and all $u \in {\mathbb C}^{n}$.
  We identify the map $u \mapsto \sum_{\fa \in \free} Z^{\fa} u 
  \otimes \bH_{\fa} x$ with an element of ${\mathbb C}^{n \times 
  n} \otimes \cY$, namely, $\sum_{\fa \in \free} Z^{\fa} 
  \otimes \bH_{\fa} x$.  As a function of $Z$, this can be 
  identified as a nc function $g \colon \Omega \to \cY_{\rm nc}$:
  \begin{equation} \label{g}
    g(Z) = \sum_{\fa \in \free} Z^{\fa} \otimes \bH_{\fa} x.
  \end{equation}
  By results from \cite{KVV-book} on noncommutative Taylor series, we have the identification
  \begin{equation}   \label{id} 
  Z^{\fa} \otimes \bH_{\fa} x = Z^{\fa} 
  \Delta^{\fa^{\top}}_{R} g( \underbrace{ \hbox{0, \dots, 0}}_{ |\fa| + 1 \text{ times 
  }}).
  \end{equation}
  Whether we are in the case where 
  $\Omega = {\rm Nilp}^{d}$ or where $\Omega$ is  a nc ball 
  around the origin where all the series $f(Z) = \sum_{\fa \in 
  \free} Z^{\fa} \otimes f_{\fa}$ converge, the function $g$ is 
  also given as the sum of its Taylor series (finite in the nilpotent 
  case)
  $$
   g(Z) = \sum_{\fa \in \free} Z^{\fa}  
   \Delta_{R}^{\fa^{\top}}g(0, \dots, 0).
 $$
  Moreover, the higher-order nc derivatives 
  $\Delta^{\fa^{\top}}_{R} g$ are uniquely determined by the 
  associated nc function $g$.  Hence the function $Z \mapsto g(Z)$ in 
  \eqref{g} being identically equal to zero forces all the higher 
  order derivatives $\Delta^{\fa^{\top}}_{R} g(0, \dots, 0)$ to be 
  zero, which in turn, due to the identity \eqref{id}, forces 
  $\bH_{\fa} x = 0$ for all $\fa \in \free$. Hence
  $\sum_{\fa \in \free} (\bH_{\fa} x) z^{\fa}  = 0$ as an 
  element of $\cY\langle \langle z \rangle \rangle$.  In this way we 
  see that the map $\phi$ \eqref{map'} is injective as well.
  
  Finally to show that $\phi$ is an isometry from $\cH(\bK)$ onto 
  $\cH(K)$, it suffices to show the set identity
  $$
  \{ x \in \cX \colon \boldf_{x} = \boldf\} = \{ x \in \cX \colon 
  f_{x} = f\}
  $$
  whenever $\boldf \in \cH(\bK)$ and $f \in \cH(K)$ are related as in 
  \eqref{map:f-to-nonf}.  This is a consequence of the fact that 
  $\boldf_{x}= 0$ in $\cH(\bK)$ exactly when $f_{x} = 0$ in 
  $\cH(K)$.  This in turn is a consequence of the analysis done in 
  the previous paragraph.
  
  Conversely, let $K$ be a cp nc kernel either on $\Omega = {\rm 
  Nilp}^{d}$ (case (a)) or on a finitely open set $\Omega  \subset 
  {\mathbb C}^{d}$ containing the origin where $K$ is $I$-locally 
  bounded (case (b)). Let $K(Z,W)(P) = H(Z) (P \otimes {\rm 
  id}_{\cL(\cY)}) H(W)^{*}$ be a minimal Kolmogorov decomposition
  (i.e. $\overline{\rm span} \{ {\rm Ran}\, H(W)^{*} \colon W \in 
  \Omega_{1}\} = \cX$ from which it follows that
  $\overline{\rm span} \{ {\rm Ran}\, (P \otimes {\rm id}_{\cL(\cX)}) 
  H(W)^{*} \colon P \in {\mathbb C}^{n \times n}, \, W \in 
  \Omega_{n}\} = \cX^{n}$ for each $n \in {\mathbb N}$), where, in 
  case (b), $H$ is locally bounded along slices.  By results from 
  \cite{KVV-book}, $H(Z) = \sum_{\fa \in \free} Z^{\fa} \otimes 
  H_{\fa}$ either on all of ${\rm Nilp}^{d}$ (case (a)) or on an 
  appropriate $\widetilde \Omega \subset \Omega$ (in case (b)).  
  Let $\bH(z) = \sum_{\fa \in \free} H_{\fa} z^{\fa}$, 
  $\bK(z,w) = \sum_{\fa, \fb \in \free} K_{\fa, \fb} 
  z^{\fa} \overline{w}^{\fb^{\top}}$ where $K_{\fa, \fb} = 
  H_{\fa} H_{\fb}^{*}$.  Then $\bK$ is a formal positive 
  nc kernel with convergence set containing either ${\rm Nilp}^{d}$ or 
  $\widetilde \Omega$.
   \end{proof}

\begin{remark}  \label{R:KVV} We here note some results from the work 
    of Kaliuzhnyi-Verbovetskyi--Vinnikov \cite{KVV-PAMS} which may be 
    viewed as formal analogues of various parts of Theorem
    \ref{T:cp=p}.
    
    \smallskip
    
    \noindent
    \textbf{1. Formal version of Corollary \ref{C:nilp}:} 
   
    \begin{theorem}   \label{T:KVVnilp} 
	The formal power series 
    $$
    \bK(z,w) = \sum_{\fa, \fb} 
    \bK_{\fa, \fb} z^{\fa} \overline{w}^{\fb^{\top}} 
    \in \cL(\cY)\langle \langle z, \overline{w} \rangle \rangle
    $$
    is a formal positive kernel if and only if, for every $n \in 
    {\mathbb N}$ and $Z \in {\rm Nilp}^{d}_{n}$,
    $$
    K_{I_{n}}(Z,Z) = \sum_{\fa, \fb \in \free} \bK_{\fa, 
    \fb} \otimes Z^{\fa} Z^{* \fb^{\top}} \in \cL(\cY \otimes 
    {\mathbb C}^{n})
    $$
    (note that the a priori infinite series is actually finite since 
    $Z \in {\rm Nilp}^{d}_{n}$) is a positive semidefinite operator.
    \end{theorem}
    
    Using the connection between formal positive kernels and cp nc 
    kernels given by Theorem \ref{T:f-vs-nc}, we see that Theorem 
    \ref{T:KVVnilp} may be considered as equivalent to Corollary 
    \ref{C:nilp}.  A direct proof of Theorem \ref{T:KVVnilp} can be 
    found in \cite[Theorem 3]{KVV-PAMS}.
    
    \smallskip
    
    \noindent
    \textbf{2. Formal version of Theorem \ref{T:cp=p} part (a):}
    The following formal version of Theorem \ref{T:cp=p} part (a) 
    appears as Theorem 2 in \cite{KVV-PAMS}.
    
    \begin{theorem} \label{T:KVV-KI}   Suppose that $\bK(z,w) \in 
	\cL(\cY) \langle \langle z, \overline{w} \rangle \rangle$ is 
	a formal power series which is uniformly convergent when a 
	$(n \times n)$-matrix $d$-tuple pair $(Z,W)$ is substituted 
	for the formal indeterminates $(z,w)$, as long as $Z,W$ are 
	in some norm-open ball around the origin $U_{n}$ in ${\mathbb 
	C}^{n \times n}$ for each $n \in {\mathbb N}$.  Suppose also 
	that the associated kernel $(Z,W) \mapsto K(Z,W)$ obtained by 
	this substitution is a positive kernel in the sense of 
	Aronszajn on each open ball $U_{n}$ for all $n \in {\mathbb 
	N}$.  Then it follows that $\bK$ is a positive kernel.
\end{theorem}

We are now ready to complete the proof of Theorem \ref{T:cp=p} 
part (a) by using its formal analogue Theorem \ref{T:KVV-KI} as the 
basic ingredient.

\begin{proof}[Proof of Theorem \ref{T:cp=p} part (a):]
  Suppose that the  nc set $\Omega$ is a uniform ball $N(0; 
  \epsilon)$ around $0$ in $({\mathbb C}^{d})_{\rm nc}$ and that the 
  kernel $K_{I_{n}}$ is Aronszajn-positive on $N(0; \epsilon)_{n}$ 
  for each $n \in {\mathbb N}$.  We also assume that $K$ is locally 
  $I$-bounded on slices.  Then we use the fact (not proven 
  in \cite{KVV-book} and verifiable by results done there for the 
  order-0 case)  that $K$ has a Taylor-Taylor series expansion 
  centered at the origin:
  $$
    K(Z,W)(P) = \sum_{\fa, \fb} K_{\fa, \fb} \otimes 
    Z^{\fa} P  W^{* \fb^{\top}}
  $$
  for some Taylor coefficient moments $K_{\fa, \fb} \in 
  \cL(\cY)$.
  We then may associate a formal kernel $\bK(z,w) = \sum_{\fa, 
  \fb} K_{\fa, \fb} z^{\fa} \overline{w}^{\fb^{\top}}$.
  The hypothesis that $K_{I_{n}}$ is a positive Aronszajn kernel on 
  each $N(0; \epsilon)_{n}$ implies that $\bK$ satisfies the 
  hypotheses of Theorem \ref{T:KVV-KI}.  We conclude that $\bK$ is a 
  positive formal kernel, hence has a formal Kolmorgorov decomposition
  $$
    \bK(z,w) = \bH(z) \bH(w)^{*} = \sum_{\fa, \fb} H_{\fa} 
    H_{\fb}^{*} \, z^{\fa} \overline{w}^{\fb^{\top}}.
  $$
  Plugging in matrix pairs $(Z,W)$ for the formal indeterminates 
  $(z,w)$ into this relation gives us
  $$
    K(Z,W)(I_{n}) = H(Z) H(W)^{*} = \sum_{\fa, \fb} H_{\fa} 
    H_{\fb}^{*} \otimes Z^{\fa} W^{* \fb^{\top}}.
  $$ 
  
  Our goal is to show that $K$ is a cp nc kernel, i.e., that 
  $K(Z,Z)(P) \succeq 0$ whenever $Z \in N(0; \epsilon)_{n}$ and $P 
  \succeq 0$ in ${\mathbb C}^{n \times n}$.  Fix $Z \in N(0; 
  \epsilon)_{n}$.  While $N(0; \epsilon)_{n}$ is not 
  similarity-invariant, it is invariant under local similarities, 
  i.e., given $Z \in N(0; \epsilon)$, there is a $\eta >0$ so that
  $S \in {\mathbb C}^{n \times n}$ with $\|S\|, \|S^{-1}\| $ both at 
  most  $\eta$ implies that $\widetilde Z = S^{-1} Z S \in N(0; 
  \epsilon)$.  If $P$ has the form $S S^{*}$ with $S$ as above, then
  the nc kernel properties of $K$ imply that
  $$
   K(Z,Z)(P) = K(Z,Z)(S S^{*}) = S\,  K(\widetilde Z, \widetilde Z)(I)\, 
   S^{*} \succeq 0.
   $$
   On the other hand, we also know that
   $$
    S \, K(\widetilde Z, \widetilde Z)(I)\, S^{*} = S\, H(\widetilde 
    Z) H(\widetilde Z)^{*} S^{*} = H(Z) \, S S^{*} \, H(Z)^{*}
  $$
  since $H$ is a nc function.  We conclude that 
  \begin{equation}   \label{id'}
    K(Z,Z)(P) = H(Z) \, P \, H(Z)^{*}
  \end{equation}
  for $P \succeq 0$ in a sufficiently small neighborhood around $I_{n}$ (where 
  the  neighborhood depends on the fixed point $Z \in N(0; \epsilon)$). 
  But both sides of \eqref{id'} are linear in $P$, in particular, 
  entire in the matrix entries of $P$.  We conclude that the identity 
  \eqref{id'} actually holds for all Hermitian $P \in {\mathbb C}^{n 
  \times n}$, and then, by linearity, for all $P \in {\mathbb C}^{n 
  \times n}$.
  We have thus exhibited a nc Kolmogorov decomposition for $K$ and we 
  conclude that $K$ is cp as wanted.
\end{proof}

As is pointed out in \cite{KVV-PAMS}, the hypothesis that $K_{I_{n}}$ 
is a positive kernel in the sense of Aronszajn cannot to weakened to 
the hypothesis that $K(Z,Z)(I_{n}) \succeq 0$ for all $Z \in N(0; 
\epsilon)$; one must use the full force of Aronszajn-positivity of 
$K_{I_{n}}$ to deduce that the associated formal kernel $\bK$ is a 
positive formal kernel.

\smallskip

\noindent
\textbf{3.  The polynomial case.}  Another interesting case of 
Theorem \ref{T:f-vs-nc} is the case where $\bK$ is a formal 
polynomial, i.e., $\bK_{\fa, \fb} = 0$ for all but finitely 
many $\fa, \fb \in \free$.  By Theorem \ref{T:KVVnilp}, it 
follows that the assumption that $K_{I_{n}}(Z,Z) \succeq 0$ for all 
$Z \in {\rm Nilp}^{d}_{n}$ for all $n \in {\mathbb N}$ is enough to 
imply that $\bK$ is a positive kernel and hence has a formal 
Kolmogorov decomposition $\bK(Z,W) = \bH(Z) \bH(W)$ for some $\bH \in 
\cL(\cX, \cY)\langle \langle z \rangle \rangle$.  The result of 
Theorem 4 from \cite{KVV-PAMS} is that the Kolmogorov factor $\bH$ 
can be taken also to be a polynomial; in addition there are estimates 
on the degree of $\bH$ in terms of th degree of $\bK$.  By the 
correspondence between  formal positive kernels and cp nc kernels 
given by Theorem \ref{T:f-vs-nc},  it is clear that one can also 
formulate a non-formal version of this result: {\em if $K(Z,W)(P) = 
K_{\fa, \fb} \otimes Z^{\fa} P W^{* \fb^{\top}}$ is a nc 
kernel such that $K(Z,Z)(I_{n}) \succeq 0$ for $Z \in ({\rm 
Nilp})^{d}_{n}$ and $n \in {\mathbb N}$, then $K$ has a nc Kolmogorov 
decomposition $K(Z,W)(P) = H(Z) P H(W)^{*}$ where $H(Z) = 
\sum_{\fa \in \free} H_{\fa} \otimes Z^{\fa}$ is also a 
polynomial ($H_{\fa} = 0$ for all but finitely many words $\fa 
\in \free$).}
\end{remark}
  
  We note some additional corollaries of Theorem \ref{T:f-vs-nc}.

\begin{corollary}  \label{C:f-vs-nc}
    Let $K$ be a nc kernel in either case (a) or case (b) as in 
    the converse side of Theorem \ref{T:f-vs-nc}, so 
    $K$ itself has a Taylor-Taylor series (a verifiable fact not proved in \cite{KVV-book})
    $$
      K(Z,W)(P) = \sum_{\fa, \fb \in \free} Z^{\fa} P W 
      ^{\fb^{\top}} \otimes \bK_{\fa, \fb}.
      $$
      Then $K$ is a cp nc kernel if and only if $[\bK_{\fa, 
      \fb}]_{\fa, \fb \in \free} \ge 0$ in the sense that 
      statement (1) in Theorem \ref{T:NFRKHS} holds.
\end{corollary}

\begin{corollary}   \label{C:f-vs-nc'}
    Let  $\cH$ be a Hilbert space of nc functions with values in 
    $\cY_{\rm nc}$ with bounded point evaluations on a nc set $\Omega 
    \subset {\mathbb C}^{d}_{\rm nc}$ containing $0$.  Let $f(z) = 
  \sum_{\fa \in \free} Z^{\fa} \otimes f_{\fa} $ be the 
  Taylor-Taylor series for $f$ centered at $0$.  Then the 
  Taylor-Taylor-coefficient maps $f \mapsto f_{\fa}$ and $f \mapsto
  \Delta^{\fa}f(0.\dots, 0)$ are all bounded.
 \end{corollary}
 
 We note that it is also possible to prove Corollary 
 \ref{C:f-vs-nc'} directly, a good exercise for those 
 having some facility with the techniques developed in 
 \cite{KVV-book}.
 
 \begin{remark}  \label{R:NFRKHSquestion}   The following two 
     questions remain open.
     \begin{itemize}
	 \item We suppose that we are
    given a formal positive kernel $\bK$ with associated 
     convergence set $\Omega$ \eqref{conv-set} or \eqref{conv-set'}
     and assume that the point $Z \in ({\mathbb C}^{d})^{n \times n} \cong ({\mathbb 
     C}^{n \times n})^{d}$ is such that the iterated series 
     \eqref{series} converges for all $W \in \Omega$. Does it then 
     follow that $Z$ itself is in $\Omega$?
     
     \item If $Z \in ({\mathbb C}^{n \times n})^{d}$ is such that 
     $K(Z,Z)$ converges, does it follow that $Z \in \Omega$?
     \end{itemize}
     In the corresponding commutative situation, the answer to both 
     questions is positive, as can be seen by using the fact that the 
     domain of convergence of the kernel function is a logarithmically 
     convex complete Reinhard domain combined with the symmetry 
     property $K(Z,W)(P)^{*} = K(W,Z)(P^{*})$ of the kernel function.
     \end{remark}
     
\section{Multipliers between nc reproducing kernel 
Hilbert spaces}  \label{S:mult}

\subsection{Characterization of contractive multipliers}

Let us suppose that we are given two cp nc kernels $K'$ and 
$K$, both defined on the Cartesian product $\Omega \times \Omega$ of 
a nc set $\Omega \subset \cS_{nc}$ with values in $\cL(\cA, \cL(\cU))_{\rm nc}$ 
and $\cL(\cA, \cL(\cY))_{\rm nc}$ respectively, where $\cA$ is a 
$C^{*}$-algebra and $\cU$ and $\cY$ are two 
auxiliary Hilbert spaces.  Suppose next that $S$ is a global function 
on $\Omega$ with values in $\cL(\cU, \cY)_{\rm nc}$.  We say that $S$ 
is a \textbf{multiplier} from $\cH(K')$ to $\cH(K)$, 
written as $S \in \cM(K',K)$ if the operator $M_{S}$ given by
$$
  (M_{S} f)(W) = S(W) f(W) \text{ for }  W \in \Omega_{m}
$$
maps $\cH(K')$ boundedly into $\cH(K)$.  Here are a few preliminary 
observations concerning such operators $M_{S}$.
\begin{itemize}
    \item
{\em If $S$ is a global function with values in $\cL(\cU, \cY)_{\rm nc}$ and 
$f$ is a global function with values in $\cL(\cA, \cU)_{\rm nc}$, then 
$W \mapsto S(W) f(W)$ is  a global function with values in 
$\cL(\cA, \cY)_{\rm nc}$.} Simply compute, for $Z \in \Omega_{n}$ and $W \in 
\Omega_{m}$,
\begin{align*}
    \left( M_{S} f\right) \left( \sbm{ Z & 0 \\ 0 & W}\right) & =
    S\left( \sbm{ Z & 0 \\ 0 & W} \right) f\left( \sbm{Z & 0 \\ 0 & 
    W} \right) \\
    & = \begin{bmatrix} S(Z) & 0 \\ 0 & S(W) \end{bmatrix}
    \begin{bmatrix} f(Z) & 0 \\ 0 & f(W) \end{bmatrix} \\
	& = \begin{bmatrix} S(Z) f(Z) & 0 \\ 0 & S(W) f(W) 
    \end{bmatrix} = \begin{bmatrix} (M_{S}f)(Z) & 0 \\ 0 & 
    (M_{S}f)(W) \end{bmatrix}.
\end{align*}

\item
{\em If $\cS$ is a vector space $\cV$ and we make the 
stronger assumption that both $S$ and $f$ are nc functions, then 
$M_{S}f$ is also a nc function.} If $Z \in \Omega_{n}$, $\widetilde Z \in 
\Omega_{\widetilde n}$ and $\alpha \in {\mathbb C}^{\widetilde n \times 
n}$ is such that $\alpha Z = \widetilde Z \alpha$, then 
\begin{align*}
    \alpha\cdot (M_{S}f)(Z) &  = \alpha S(Z) f(Z) = S(\widetilde Z) \alpha f(Z) = 
    S(\widetilde Z) f(\widetilde Z) \alpha \\
    & = (M_{S}f)(\widetilde Z) \cdot \alpha.
\end{align*}

\item {\em Let $\sigma'$ and $\sigma$ denote the canonical $\cA$-actions 
\eqref{rep1} on $\cH(K')$ and $\cH(K)$ respectively.  Suppose that 
$S$ is a global/nc function on $\Omega$ with 
values in $\cL(\cU, \cY)_{\rm nc}$ and  $f$ is a global/nc 
function on $\Omega$.  Then, for all $a \in \cA$,
$$ 
\sigma'(a)(M_{S} f) = M_{S} (\sigma(a) f),
$$
i.e., $M_{S}$ intertwines $\sigma(a)$ with $\sigma'(a)$ for all $a 
\in \cA$.}
Indeed, compute
\begin{align*}
    \left( \sigma'(a) (M_{S}f) \right)(W)(v) & = 
    \left(M_{S}f\right)(W)(va) = S(W)f(W)(va) \\
    & = \left(M_{S} \sigma(a)f\right)(W)(v).
\end{align*}

\item
{\em Suppose  that $S$ is a global/nc function on $\Omega$ with values 
$\cL(\cU, \cY)_{\rm nc}$ with the property that $M_{S} f \in \cH(K)$ 
for each $f \in \cH(K')$, i.e., $M_{S}$ is well defined as an 
operator from $\cH(K')$ into $\cH(K)$.  Then $S \in \cM(K',K)$.}
By the Closed Graph Theorem (see e.g.\ \cite{Rudin}), it suffices to 
check that $M_{S}$ is closed as an operator from $\cH(K')$ into 
$\cH(K)$.  We therefore assume that $\{f_{n}\}$ is a sequence in $\cH(K')$ 
with $\lim_{n \to \infty} f_{n} = f$ in $\cH(K')$ and that $M_{S}f_{n}$ 
converges in $\cH(K)$ to  the global/nc function $g \in \cH(K)$. Then, due to the 
boundedness of the point-evaluation maps we have, for each $W \in 
\Omega_{m}$ and $v \in \cA^{m}$,
\begin{align*}
g(W) v  & = \lim_{n \to \infty} (M_{S} f_{n})(W) v  =
\lim_{n \to \infty} S(W) f_{n}(W) v = S(W) \left( \lim_{n \to \infty} 
f_{n}(W) v\right)  \\
& = S(W) f(W) v = (M_{S} f)(W) v
\end{align*}
from which it follows that $M_{S}f = g$ in $\cH(K)$.
\end{itemize}

Given $K'$, $K$, and $S$ as above, we shall say that $S$ is a 
\textbf{contractive multiplier} from $\cH(K')$ to $\cH(K)$, written 
as $S \in \overline{\cB}\cM(K',K)$, if $S \in \cM(K',K)$ with  
operator norm of $M_{S}$ at most 1:  $\|M_{S}\|_{\cL(\cH(K'), \cH(K))} \le 1$.
Our main result concerning contractive multipliers is the following.

\begin{theorem}  \label{T:contmult}
    Given global/nc kernels $K'$ and $K$ from $\Omega \times \Omega$ 
    to $\cL(\cA, \cL(\cU))_{\rm nc}$ and $\cL(\cA, \cL(\cY))_{\rm nc}$ respectively 
    and given a global/nc function 
    $S$ from $\Omega$ to $\cL(\cU, \cY)_{\rm nc}$, the following are 
    equivalent:
    \begin{enumerate}
	\item $S \in \overline{\cB}\cM(K',K)$.
	
	\item The kernel $K_{S}$ from $\Omega$ to $\cL(\cA, 
	\cL(\cY))_{\rm nc}$ given by
\begin{equation}   \label{KS}
  K_{S}(Z,W)(P) = K(Z,W)(P) - S(Z) \, K'(Z,W)(P)\, S(W)^{*}
\end{equation}
is a cp global/nc kernel.
\end{enumerate}
\end{theorem}

\begin{proof}  It is easily verified that if $K$ and $K'$ are 
    global/nc kernels and $S$ is a global/nc function, then $K_{S}$ 
    is a global/nc kernel.  
    
    If $M_{S} \in \cM(K',K)$, we compute the 
    action of $M_{S}^{*}$ on a kernel element $K_{W,v,y} \in \cH(K)$
    \eqref{kerel} of $\cH(K)$ as follows: for $f \in  \cH(K')$, $W 
    \in \Omega_{m}$, $v \in \cA^{1 \times m}$ we have
    \begin{align*}  
	\langle f, (M_{S})^{*} K_{W,v,y} \rangle_{\cH(K)} & =
	\langle M_{S}f, K_{W,v,y} \rangle_{\cH(K)}  \\
	& = \langle (M_{S}f)(W)(v^{*}), y \rangle_{\cY^{m}} \text{ (by 
	\eqref{reprod} for $\cH(K)$)}  \\
& = \langle S(W) f(W) (v^{*}), y \rangle_{\cY^{m}}  \\
& = \langle f(W)(v^{*}), S(W)^{*} y \rangle_{\cU^{m}}  \\
& = \langle f, K_{W,v, S(W)^{*}y} \rangle_{\cH(K')} \text{ (by 
\eqref{reprod} for $\cH(K')$)}
\end{align*}
and hence 
\begin{equation}  \label{MSadj}
  (M_{S})^{*}\colon K_{W,v,y} \mapsto K_{W,v, S(W)^{*} y}.
\end{equation}
It follows that
\begin{align*} & \langle K_{S}(W,W)(v^{*}v) y, y \rangle_{\cY^{m}} \\ 
 & \quad   = \langle K(W,W)(v^{*}v) y, y \rangle_{\cY^{m}} - \langle 
    K''(W,W)(v^{*}v) S(W)^{*} y, S(W)^{*} y \rangle_{\cU^{m}} \\
    & \quad = \| K_{W,v,y} \|^{2}_{\cH(K)} - \langle 
    K'_{W,v,S(W)^{*}y}, K'_{W,v,S(W)^{*}y} \rangle_{\cH(K')} \\
    & \quad = \| K_{W,v,y} \|^{2}_{\cH(K)} - \| M_{S}^{*} K_{W,v,y} 
    \|^{2}_{\cH(K')} \ge 0
 \end{align*}
and we conclude that $K_{S}$ is cp.  This completes the proof of 
necessity (1) $\Rightarrow$ (2).

Conversely, assume only that $K_{S}$ is cp.  As we are still assuming that 
$K'$ and $K$ are global/nc kernels and that $S$ is a global/nc 
function, we already know that $K_{S}$ is a global/nc kernel.  The 
proof of the necessity direction motivates us to define an operator 
$\Gamma$ on kernel elements $K_{W, v, y}$ of $\cH(K)$ by
$$
  \Gamma \colon K_{W,v,y} \mapsto K'_{W,v, S(W)^{*}y}
$$
and then extend to finite linear combinations of kernel elements by 
linearity. The computation, for $W^{(j)} \in \Omega_{m_{j}}$ and $v_{j} 
\in \cA^{1 \times m_{j}}$,
\begin{align*}
   & \left\| \sum_{j=1}^{N} K_{W^{(j)}, v_{j}, y_{j}} 
   \right\|^{2}_{\cH(K)} -
    \left\| \Gamma \left( \sum_{j=1}^{N} K_{W^{(j)}, v_{j}, y_{j}} 
    \right) \right\|^{2}_{\cH(K')} \\
 & = \sum_{i,j=1}^{N} \langle K(W^{(i)},W^{(j)})(v_{i}^{*} 
 v_{j})y_{j}, y_{i}\rangle_{\cY^{m_{i}}} \\
 & \quad \quad  -  \sum_{i,j=1}^{N} \langle  
 S((W^{(i)}) K'(W^{(i)}, W^{(j)})(v_{i}^{*}v_{j}) 
 S(W^{(j)})^{*} y_{j}, y_{i} \rangle_{\cY^{m_{i}}} \\
 & = \sum_{i,j=1}^{N} \langle K_{S}(W^{(i)}, W^{(j)})(v_{i}^{*} v_{j}) 
 y_{j}, y_{i} \rangle_{\cY} \ge 0 \text{ (by property 
 \eqref{cpker-expanded}).}
 \end{align*}
We conclude that $\Gamma$ is well-defined on the span of the kernel 
elements in $\cH(K)$ and extends by continuity to a well-defined 
contraction operator from all of $\cH(K)$ into $\cH(K')$.  
Furthermore, by reading the computation \eqref{MSadj} backwards, we see that 
$\Gamma^{*} \colon \cH(K') \to \cH(K)$ is given by
$M_{S} \colon f(Z) \mapsto S(Z) f(Z)$.  Hence $M_{S} = \Gamma^{*}$ is 
a contraction from $\cH(K')$ to $\cH(K)$, i.e., $S \in \overline{\cB} 
\cM(K',K)$.
\end{proof}

\begin{remark} \label{R:MSintertwaine} We note that the formula for $M_{S}^{*}$ on kernel 
    elements gives us a second way to see that $M_{S}$ intertwines 
    $\sigma(a)$ with $\sigma'(a)$ for each $a \in \cA$, once we 
    recall  the action of $\sigma(a)$ and $\sigma'(a)$ on kernel 
    elements given by \eqref{kernelaction}:
 \begin{align*}
     \sigma'(a) (M_{S})^{*} K_{W,v,y} & = \sigma'(a) K'_{W, v, 
     S(W)^{*}y}  = K'_{W, a v, S(W)^{*}y} \\
     & = (M_{S})^{*} K_{W,av, y} 
     = (M_{S})^{*} \sigma(a) K_{W,v,y}.
 \end{align*}
 \end{remark}
 
 \subsection{The global/nc reproducing kernel Hilbert spaces 
 associated with a contractive multiplier $S$}
 In the classical setting where $S$ is a contractive multiplier 
 between the Hardy spaces $H^{2}\otimes \cU$ and $H^{2} \otimes \cY$
 over the unit disk (equal to reproducing kernel Hilbert spaces with 
 Szeg\H{o} kernel $k(z,w) = \frac{1}{1-z \overline w}$ tensored with 
 either $I_{\cU}$ or $I_{\cY}$), the associated kernel $K_{S}(z,w) = 
 \frac{I_{\cY} - S(z) S(w)^{*}}{1 - z \overline{w}}$ has become known 
 as a \textbf{de Branges-Rovnyak kernel} and the associated Hilbert 
 space $\cH(K_{S})$ as a \textbf{de Branges-Rovnyak space}, due to 
 the fundamental work of de Branges-Rovnyak \cite{dBR1,dBR2} (see 
 also \cite{BB-HOT1,BK,NV0,NV1,NV2,Sarason}).  Much of the theory associated 
 with these spaces goes through in our more general global/nc 
 setting. 
 
 As a starting point for the discussion, we review the
 general theory of minimal decompositions for Hilbert 
 spaces $\cH'$ which are contractively included in another Hilbert 
 space $\cH$, summarized as follows.  The following general 
 formulation comes from \cite{BB-HOT1}.
 
 \begin{proposition}  \label{P:BB-HOT1}  Let $\cK$ be a Hilbert 
     space which is contractively included in a larger Hilbert space 
     $\cH$ but with its own possibly distinct norm:
 $$ k \in  \cK \Rightarrow k \in \cH \text{ and then }
 \|k\|_{\cH} \le \| k\|_{\cK}.
 $$
 Define another Hilbert space $\cK^{\perp_{\rm dBR}}$ (the 
 \textbf{Brangesian complement to $\cK$}) by
 $$
  \cK^{\perp_{\rm dBR}} = \left\{  h \in \cH \colon \| h 
  \|^{2}_{\cK^{\perp_{\rm dBR}}}: = \sup \{ \| h + k \|^{2}_{\cH} - 
  \| k \|^{2}_{\cK} \colon k \in \cK \} < \infty\right\} 
$$
Then $\cK$ and $\cK^{\perp_{\rm dBR}}$ are complementary subspaces of 
$\cH$ in the sense that each $h \in \cH$ has a decomposition $h = k + 
k'$ with $k \in \cK$ and $k' \in \cK^{\perp_{\rm dBR}}$.  Then also 
the norm of any $h \in \cH$ is given by
\begin{equation}   \label{norm-ineq}
    \| h\|^{2}_{\cH} = \inf \{ \|k\|^{2}_{\cK} + \| 
    k'\|^{2}_{\cK^{\perp_{\rm dBR}}} \colon k \in \cK,\, k' \in 
    \cK^{\perp_{\rm dBR}} \text{ such that } h = k + k'\}.
\end{equation}
Moreover:
\begin{enumerate}
    \item There is a unique choice of vectors $(k,k') \in \cK \times 
    \cK^{\perp_{\rm dBR}}$ for which the infimum in \eqref{norm-ineq} 
    is attained, namely: $k = \iota \iota^{*} h$ and $k' = (I_{\cH} - 
    \iota \iota^{*}) h$ where $\iota \colon k \mapsto k$ is the 
    inclusion map considered as a contraction operator from $\cK$ 
    into $\cH$.
    
    \item The space $\cK$ can be characterized as the Brangesian 
    complement $\cK'': = (\cK^{\perp_{\rm dBR}})^{\perp_{\rm dBR}}$ of 
    $\cK': =\cK^{\perp_{\rm dBR}}$:
 $$
 \cK =
  \left\{ h \in \cH \colon \| h \|^{2}_{\cK''}: = 
  \sup \{ \| h + k'\|^{2}_{\cH} - \| 
  k'\|^{2}_{\cK^{\perp_{\rm dBR}}} \colon k' \in \cK^{\perp_{\rm 
  dBR}} \} < \infty \right\}
  $$
  and then $\| k \|^{2}_{\cK} = \| k \|^{2}_{\cK''}$.
 \end{enumerate}
 \end{proposition}
 
 A particular application of Proposition \ref{P:BB-HOT1} is to the 
 case where $A$ is a contraction operator from a Hilbert space $\cK$ 
 to a Hilbert space $\cH$ and we define two linear submanifolds of 
 $\cH$ by
 \begin{equation}   \label{pb1}
 \cH_{A} = {\rm Ran}\, (I  - A A^{*})^{1/2}, \quad \cM_{A} = {\rm Ran} 
 \, A
 \end{equation}
 with respective pull-back norms
 \begin{equation}   \label{pb2}
 \| (I - A A^{*})^{1/2} g \|_{\cH_{A}} = \| Q g \|_{\cH}, \quad
 \| A h \|_{\cM_{A}} = \| Q' h \|_{\cK}
 \end{equation}
 where $Q \colon \cH \to ({\rm Ker}\, (I - A A^{*})^{1/2})^{\perp}$ and $Q' 
 \colon \cK \to  ({\rm Ker}\, A)^{\perp}$ are the orthogonal 
 projections.  Then it is easily verified that both $\cH_{A}$ and 
 $\cM_{A}$ are themselves Hilbert spaces (in particular, complete in 
 their respective norms) and are each contractively included in 
 $\cH$ with respective contractive adjoint inclusion maps
 $(\iota_{\cH_{A}})^{*} \colon \cH \to \cH_{A}$ and 
 $(\iota_{\cM_{A}})^{*} \colon \cH \to \cM_{A}$ given by
 $$
 (\iota_{\cH_{A}})^{*} \colon h \mapsto (I - A A^{*}) h, \quad
 (\iota_{\cM_{A}})^{*} \colon h \mapsto A A^{*} h.
 $$
 One can then verify that $\cH_{A}$ and $\cM_{A}$ are Brangesian 
 complements of each other.
 
 Suppose now that $K'$ and $K$ are cp global/nc kernels from $\Omega 
 \times \Omega$ to respectively $\cL(\cA_{\rm nc}, \cL(\cU)_{\rm 
 nc})$ and $\cL(\cA_{\rm nc}, \cL(\cY)_{\rm nc})$ respectively, 
 and that $S \in \overline{\cB}\cM(K',K)$ is a contractive 
 multiplier.  Our interest is to apply the discussion of the 
 preceding paragraph to the case where $A = M_{S} \colon \cH(K') \to 
 \cH(K)$.  The ensuing result gives an explicit geometric 
 characterization of the global/nc version of the de Branges-Rovnyak 
 space $\cH(K)$.
 
 \begin{theorem}  \label{T:KSspace}
     Suppose that $K',K$ are two cp global/nc kernels and $S \in 
     \overline{\cB}\cM(K',K)$ is a contractive multiplier as above.
     Then the global/nc reproducing kernel Hilbert space $\cH(K_{S})$ 
     associated with the cp global/nc kernel $K_{S}$ \eqref{KS} is 
     isometrically equal to the pull-back space $\cH_{M_{S}}$ defined 
     by \eqref{pb1}, \eqref{pb2} with $A = M_{S}$, specifically:
 $\cH(K_{S}) = {\rm Ran}\, (I - M_{S} M_{S}^{*})^{1/2}$ with 
 $$
  \| (I - M_{S} M_{S}^{*})^{1/2} h \|_{\cH(K_{S})} = \| Q h \|_{\cH(K)}
 $$
 where $Q\colon \cH(K) \to ({\rm Ker}\, (I - M_{S} 
 M_{S}^{*})^{1/2})^{\perp}$ is the orthogonal projection.
 
 Equivalently, $\cH(K_{S})$ is the Brangesian complement of the 
 Hilbert space $\cM_{M_{S}}$ contractively included in $\cH(K)$ defined by
 $\cM_{S} = {\rm Ran}\, M_{S}$ with pull-back norm given by
 $$
    \| M_{S} g \|_{\cM_{M_{S}}} = \| Q' g \|_{\cH(K')}
 $$
 where $Q' \colon \cH(K') \to ({\rm Ker}\, M_{S})^{\perp}$ is the 
 orthogonal projection.  The space $\cM_{S}$ is itself a global/nc 
 reproducing kernel Hilbert space with reproducing kernel $K^{0}_{S}$ 
 given by
 $$
   K^{0}_{S}(Z,W)(P): = S(Z) \, K'(Z,W)(P)\, S(W)^{*}.
 $$
 \end{theorem}
 
 \begin{proof}
     All these results follow from the general discussion preceding 
     the theorem once we verify that $\cH(K_{S})$ is isometrically 
     equal to $\cH_{M_{S}}$ and that $\cH(K^{0}_{S})$ is 
     isometrically equal to $\cM_{M_{S}}$.  We do only the first case 
     as the second is similar.
     
The starting point is the following consequence of the formula \eqref{MSadj} 
for the action of $M_{S}^{*}$ on a kernel element in $\cH(K)$: 
for $W \in \Omega_{m}$, $v \in \cA^{1 \times m}$, and $y \in \cY^{m}$ we have
  $$  (K_{S})_{W,v,y} = (I - M_{S} M_{S}^{*}) K_{W,v,y}.
  $$
  Hence we can compute the norm-squared of a finite linear 
  combination of kernel elements in $\cH(K_{S})$ as follows.
  If $f = \sum_{j=1}^{N} (K_{S})_{W^{(j)}, v_{j}, y_{j}}$ for some 
  $W^{(j)} \in \Omega_{m_{j}}$, $v_{j} \in \cA^{1 \times m_{j}}$ and 
  $y_{j} \in \cY^{m_{j}}$, then $f$ has the form $f = (I - M_{S} 
  M_{S}^{*})g$ where $g = \sum_{j=1}^{N} K_{W^{(j)}, v_{j}, y_{j}}$ 
  is the corresponding finite linear combination of kernel elements 
  from $\cH(K)$.  Furthermore,
  \begin{align*}
    & \| f \|^{2}_{\cH(K_{S})} = \| (I - M_{S} M_{S}^{*}) g 
    \|^{2}_{\cH(K_{S})} \\
   & =  \| \sum_{j=1}^{N}  (K_{S})_{W^{(j)}, v_{j},y_{j}}\|^{2}_{\cH(K_{S})} 
  =\sum_{i,j=1}^{N} \langle K_{S}(W^{(i)}, W^{(j)})(v_{i}^{*} v_{j})  y_{j}, y_{i} 
  \rangle_{\cY^{m_{i}}}  \\
 & = \sum_{i,j=1}^{N}  \langle K(W^{(i)},W^{(j)})(v_{i}^{*} 
 v_{j})y_{j}, y_{i} \rangle_{\cY^{m_{i}}} \\
 & \quad \quad - \sum_{i,j=1}^{N}\langle  K'(W^{(i)}, W^{(j)})(v_{i}^{*} v_{j}) 
 S(W^{(j)})^{*}y_{j}, S(W^{(i)})^{*}y_{i} \rangle_{\cU^{m_{i}}}  \\
 & = \sum_{i,j=1}^{N} \left( \langle K_{W^{(j)},v_{j},y_{j}}, 
 K_{W^{(i)},v_{i},y_{i}} \rangle_{\cH(K)} - \langle M_{S}^{*} 
 K_{W^{(j)}, v_{J}, 
 y_{j}}, M_{S}^{*} K_{W^{(i)}, v_{i}, y_{i}} \rangle_{\cH(K')} \right)  \\
 & = \left\langle (I - M_{S} M_{S}^{*}) \left( \sum_{j=1}^{N} 
 K_{W^{(j)}, v_{j}, y_{j}}\right), \, 
 \sum_{i=1}^{N} K_{W^{(i)}, v_{i}, y_{i}} 
 \right\rangle_{\cH(K)}  = \| f \|^{2}_{\cH_{M_{S}}}.
 \end{align*}
 and we conclude that a dense subset of $\cH(K_{S})$ is isometrically equal 
 to a dense subset of $\cH_{M_{S}}$.  By taking limits and using the 
 boundedness of the respective point-evaluations, we get that $\cH(K_{S})$ is equal 
 to $\cH_{M_{S}}$ isometrically as claimed.
 \end{proof}
 
 As a special case, consider the situation where $K'$ and $K$ are two 
 cp global/nc kernels from $\Omega \times \Omega$ 
 to $\cL(\cA_{\rm nc}, \cL(\cY)_{\rm nc})$ (i.e., in this case $\cU = \cY$) and suppose 
 that $\cH(K')$ is contractively included in $\cH(K)$.  This is the 
 precise situation where $S = I$ is in the contractive multiplier 
 class $\overline{\cB}\cM(K',K)$, where we use the notation $I$ for the 
 global/nc function $S(Z) = I_{\cY^{n}}$ if $Z \in \Omega_{n}$.
 We thus arrive at the following corollary concerning contractive 
 inclusions of global/nc reproducing kernel Hilbert spaces.
 
 \begin{corollary} \label{C:cont-incl}
     Suppose that $\cH(K')$ and $\cH(K)$ are two global/nc 
     reproducing kernel Hilbert spaces of functions with values in 
     $\cL(\cA, \cY)_{\rm nc}$. Then $\cH(K')$ is contained contractively 
     in $\cH(K)$ if and only the global/nc kernel $K''$ given by
     $$
     K''(Z,W)(P) = K(Z,W)(P) - K'(Z,W)(P)
     $$
     is cp.  In this case the Brangesian complement 
     $\cH(K')^{\perp_{\rm dBR}} = \cH(K) \ominus_{\rm dBR} \cH(K')$ 
     is also a global/nc reproducing kernel Hilbert space with 
     associated cp kernel equal to $\cH(K - K')$:
  $$
  \cH(K) \ominus_{\rm dBR} \cH(K') = \cH(K - K').
  $$
 \end{corollary}
 
 \begin{remark}  Section 3 of the \cite{NFRKHS}  develops many of the 
     results of this section for the setting of formal rather than 
     concrete nc RKHSs.  In particular \cite[Theorem 3.15]{NFRKHS} is the formal analogue of 
     Theorem \ref{T:contmult} for the special case where both $K$ and 
     $K'$ are taken to be the formal nc Szeg\H{o} kernel and 
     \cite[Theorem 3.4]{NFRKHS} is roughly a formal analogue of 
     Proposition \ref{P:BB-HOT1}. The formal setting of Theorem 
     \ref{T:KSspace} for the special case where $K'$ and $K$ are 
     formal Szeg\H{o} kernels is worked out in \cite[Proposition 
     4.1]{BBF3}.  Additional results for the formal setting can be 
     obtained by using the correspondence between functional and formal 
     RKHSs explained in Subsection \ref{S:NFRKHS} to transfer results 
     here for the functional setting to the formal setting.
\end{remark}

\noindent 
\textbf{Acknowledegment:}  The research of the first and third 
authors was partially supported by the US-Israel Binational Science 
Foundation.

\end{document}